\documentclass[11pt]{amsart} 
\usepackage[T1]{fontenc}
\usepackage{changebar}
\usepackage{amsmath,amssymb,amsthm,amsfonts,stmaryrd,rotating,dsfont}
\usepackage[latin9]{inputenc} 
 \usepackage[a4paper%,left=1cm,right=1cm
]{geometry}
\usepackage[matrix,arrow,curve]{xy}
\CompileMatrices
\DeclareMathOperator{\supp}{supp}

\newcommand{\intd}{\,\mathrm{d}}

\newcommand{\BimodKMS}{\mathrm{(Bimod_{KMS})}}

\newcommand{\eps}{\varepsilon}
\newcommand{\heps}{\hat{\varepsilon}}
\newcommand{\hbeps}{{_{\hat{B}}\hat{\varepsilon}}}
\newcommand{\hepsc}{\hat{\varepsilon}_{\hat{C}}}
\newcommand{\hceps}{{_{\hat{C}}\hat{\varepsilon}}}
\newcommand{\hepsb}{\hat{\varepsilon}_{\hat{B}}}

\newcommand{\rtintb}{(\dbA)^{\Delta_{B}}}
\newcommand{\ltintb}{{^{\Delta_{B}}(\dcA)}}
\newcommand{\rtintc}{(\dAb)^{\Delta_{C}}}
\newcommand{\ltintc}{{^{\Delta_{C}}(\dAc)}}

\newcommand{\bphi}{{_{B}\phi}}

\newcommand{\cphi}{{_{C}\phi}}
\newcommand{\bpsi}{{_{B}\psi}}
\newcommand{\cpsi}{{_{C}\psi}}
\newcommand{\phib}{{\phi_{B}}}
\newcommand{\psib}{{\psi_{B}}}
\newcommand{\phic}{{\phi_{C}}}
\newcommand{\psic}{{\psi_{C}}}

\newcommand{\hbpsib}{{_{\hat{B}}\hat{\psi}}_{\hat{B}}}
\newcommand{\hbpsi}{{_{\hat{B}}\hat{\psi}}}
\newcommand{\hpsib}{\hat{\psi}_{\hat{B}}}
\newcommand{\hcpsi}{{_{\hat{C}}\hat{\psi}}}
\newcommand{\hpsic}{\hat{\psi}_{\hat{C}}}
\newcommand{\hhat}[1]{\hat{\hat{#1}}}

\newcommand{\dual}[1]{#1^{\vee}}

\newcommand{\dA}{\dual{A}}
\newcommand{\dB}{\dual{B}}
\newcommand{\dC}{\dual{C}}

\newcommand{\dmAm}{{_{\mu}A^{\vee}_{\mu}}}
\newcommand{\hdmAm}{{_{\hat \mu}\hat{ A}^{\vee}_{\hat\mu}}}

\newcommand{\dbA}{\dual{\bA}}
\newcommand{\dcA}{\dual{\cA}}
\newcommand{\dAb}{\dual{\Ab}}
\newcommand{\dAc}{\dual{\Ac}}

\newcommand{\dbAb}{\dual{\bA_{B}}}
\newcommand{\dcAc}{\dual{\cA_{C}}}

\newcommand{\comega}{{_{C}\omega}}
\newcommand{\bomega}{{_{B}\omega}}
\newcommand{\omegac}{\omega_{C}}
\newcommand{\omegab}{\omega_{B}}
\newcommand{\cupsilon}{{_{C}\upsilon}}
\newcommand{\bupsilon}{{_{B}\upsilon}}
\newcommand{\upsilonc}{\upsilon_{C}}
\newcommand{\upsilonb}{\upsilon_{B}}

\newcommand{\cphic}{\cphi{}_{C}}
\newcommand{\bpsib}{\bpsi{}_{B}}
\newcommand{\GstG}{G {_{s}\times_{t}} G}
\newcommand{\GssG}{G {_{s}\times_{s}} G}
\newcommand{\GttG}{G {_{t}\times_{t}} G}

 \usepackage[pagewise,displaymath,mathlines]{lineno}
\numberwithin{equation}{section}

\theoremstyle{definition} %\newcommand{\fixme}[1]{}

\swapnumbers
\newtheorem{remark}{Remark}[subsection]
\newtheorem{example}[remark]{Example}

\newtheorem*{notation*}{Notation}

\theoremstyle{plain}
\newtheorem{definition}[remark]{Definition}
\newtheorem{theorem}[remark]{Theorem}
\newtheorem{proposition}[remark]{Proposition}
\newtheorem{corollary}[remark]{Corollary}
\newtheorem{lemma}[remark]{Lemma}

\newtheorem*{assumption*}{Assumption}

\newcommand{\C}{\mathds{C}}

\newcommand{\id}{\iota}
\DeclareMathOperator{\Hom}{Hom}
\DeclareMathOperator{\End}{End}

\newcommand{\bB}{{_{B}B}}
\newcommand{\Bb}{B_{B}}
\newcommand{\cC}{{_{C}C}}
\newcommand{\Cc}{C_{C}}

\newcommand{\hAltkA}{\hat{A} {_{\hat{B}}\times } \hat{A}} % left Takeuchi product
\newcommand{\hArtkA}{\hat{A} \times_{\hat{ C}} \hat{A}} % right Takeuchi product
\newcommand{\AltkA}{A {_{B}\times } A} % left Takeuchi product
\newcommand{\ArtkA}{A \times_{C} A} % right Takeuchi product
\newcommand{\Lreg}{L_{\mathrm{reg}}} % regular left multipliers
\newcommand{\Rreg}{R_{\mathrm{reg}}} % regular right multipliers
\newcommand{\oo}{\otimes}

\newcommand{\osub}[1]{\underset{#1}{\otimes}}
\newcommand{\osup}[1]{\stackrel{#1}{\otimes}}
\newcommand{\oos}{\osub{s}}
\newcommand{\ool}{\osub{l}}
\newcommand{\oor}{\osub{r}}
\newcommand{\oot}{\osub{t}}

\newcommand{\ooS}{\osup{s}}
\newcommand{\ooT}{\osup{t}}

\newcommand{\As}{A_{s}}
\newcommand{\At}{A_{t}}
\newcommand{\sA}{{_{s}A}}
\newcommand{\tA}{{_{t}A}}

\newcommand{\AsA}{A \oos A}
\newcommand{\AlA}{A \ool A}
\newcommand{\AtA}{A \oot A}

\newcommand{\ArA}{A \oor A}

\newcommand{\ATA}{A \ooT A}
\newcommand{\ASA}{A \ooS A}

\newcommand{\ATAsA}{A \ooT A \oos A}
\newcommand{\ATAlA}{A \ooT A \ool A}
\newcommand{\AlAsA}{A \ool A \oos A}
\newcommand{\AlAlA}{A \ool A \ool A}
\newcommand{\AsAsA}{A \oos A \oos A}

\newcommand{\AsoAlsA}{\As \oo {_{(1 \oo s)}(\AlA)}}

\newcommand{\sAsAotA}{{_{(s\oo 1)}(\AsA)} \oo \tA}

\newcommand{\bA}{{_{B}A}}
\newcommand{\Ab}{A_{B}}
\newcommand{\cA}{{_{C}A}}
\newcommand{\Ac}{A_{C}}
\newcommand{\sbA}{{_{S(B)}A}}

\newcommand{\hbA}{{_{\hat B}\hat A}}
\newcommand{\hAb}{\hat A_{\hat B}}
\newcommand{\hcA}{{_{\hat C}\hat A}}
\newcommand{\hAc}{\hat A_{\hat C}}

\newcommand{\oob}{\osub{B}}
\newcommand{\ooB}{\osup{B}}
\newcommand{\ooc}{\osub{C}}
\newcommand{\ooC}{\osup{C}}

\newcommand{\oomb}{\osub{M(B)}}
\newcommand{\oomB}{\osup{M(B)}}
\newcommand{\oomc}{\osub{M(C)}}
\newcommand{\oomC}{\osup{M(C)}}
\newcommand{\oomhb}{\osub{M(\hat B)}}
\newcommand{\oomhc}{\osub{M(\hat C)}}

\newcommand{\AbA}{A \oob A}

\newcommand{\ABA}{A \ooB A}

\newcommand{\AcA}{A \ooc A}

\newcommand{\ACA}{A \ooC A}

\newcommand{\AcAcA}{A \ooc A \ooc A}

\newcommand{\AcAbA}{A \ooc A \oob A}
\newcommand{\AcAlA}{A \ooc A \ool A}
\newcommand{\ArAbA}{A \oor A \oob A}
\newcommand{\ArAlA}{A \oor A \ool A}

\newcommand{\ACABA}{A \ooC A \ooB A}

\newcommand{\AlABA}{A \ool A \ooB A}

\newcommand{\ACArA}{A \ooC A \oor A}
\newcommand{\AlArA}{A \ool A \oor A}

\newcommand{\AcArA}{A \ooc A \oor A}

\newcommand{\oohc}{\underset{\hat{C}}{\otimes}}
\newcommand{\oohb}{\underset{\hat{B}}{\otimes}}
\newcommand{\oohl}{\underset{\hat{l}}{\otimes}}
\newcommand{\oohr}{\underset{\hat{r}}{\otimes}}
\newcommand{\oohC}{\stackrel{\hat{C}}{\otimes}}
\newcommand{\oohB}{\stackrel{\hat{B}}{\otimes}}
\newcommand{\hAcA}{ \hat{A} \oohc \hat{A}}
\newcommand{\hAbA}{\hat{ A} \oohb \hat{A}}
\newcommand{\hAlA}{\hat{ A} \oohl \hat{ A}}
\newcommand{\hArA}{\hat{ A} \oohr \hat{ A}}

\newcommand{\op}{\mathrm{op}}
\newcommand{\co}{\mathrm{co}}

\newcommand{\lmultl}{v_{\lambda}}
\newcommand{\lmultr}{v_{\rho}}
\newcommand{\lmult}{v}

\newcommand{\rmultl}{{_{\lambda}v}}
\newcommand{\rmultr}{{_{\rho}v}}
\newcommand{\rmult}{v}
\newcommand{\Tl}{T_{\lambda}}
\newcommand{\Tr}{T_{\rho}}
\newcommand{\lT}{{_{\lambda}T}}
\newcommand{\rT}{{_{\rho}T}}
\newcommand{\hTl}{\hat{T}_{\lambda}}
\newcommand{\hTr}{\hat{T}_{\rho}}
\newcommand{\hlT}{{_{\lambda}\hat{T}}}
\newcommand{\hrT}{{_{\rho}\hat{T}}}

\newcommand{\beps}{{_{B}\varepsilon}}
\newcommand{\epsc}{\varepsilon_{C}}
\newcommand{\ceps}{{_{C}\varepsilon}}
\newcommand{\epsb}{\varepsilon_{B}}

\newcommand{\mb}{\mu_{B}}
\newcommand{\mc}{\mu_{C}}

\newcommand{\mult}{m}
\newcommand{\actb}{\triangleleft}
\newcommand{\actc}{\triangleright}

\title[Regular multiplier Hopf
  algebroids II. Integration and duality]{Regular multiplier Hopf
  algebroids II. \\ Integration on and duality of algebraic quantum groupoids}

\author{Thomas Timmermann} 
 \address{FB Mathematik und Informatik, University of Muenster \\ Einsteinstr.\ 62, 48149
   Muenster, Germany}
 \email{timmermt@math.uni-muenster.de}

 \thanks{Supported by the SFB 878
     ``Groups, geometry and actions'' funded by the DFG}
\date{\today}


\begin{document}

\begin{abstract}
  A fundamental feature of quantum groups is that many come in pairs
  of mutually dual objects, like finite-dimensional Hopf algebras and
  their duals, or quantisations of function algebras and of universal
  enveloping algebras of Poisson-Lie groups.

  The same phenomenon was studied for quantum groupoids in various
  settings. In the purely algebraic setup, the construction of a dual
  object was given by Schauenburg and by Kadison and Szlach\'anyi, but
  required the quantum groupoid to be finite with respect to the
  base. A sophisticated duality for measured quantum groupoids was
  developed by Enock, Lesieur and Vallin in the setting of von Neumann
  algebras.

  We propose a purely algebraic duality theory without any finiteness
  assumptions, generalising Van Daele's duality theory of multiplier
  Hopf algebras and borrowing ideas from the theory of measured
  quantum groupoids.  Our approach is based on the multiplier Hopf
  algebroids recently introduced by Van Daele and the author, and on a
  new approach to integration on algebraic quantum groupoids. The main
  concept are left and right integrals on regular multiplier Hopf
  algebroids that are adapted to quasi-invariant weights on the basis.
  Given such integrals, we show that they are unique up to rescaling, admit
  modular automorphisms, and that left and right ones are related by
  modular elements. Then, we construct,  without  any finiteness or
  Frobenius assumption, a dual multiplier Hopf
  algebroid with integrals  and prove biduality.
\end{abstract}
\maketitle

\tableofcontents

\section{Introduction}

Quantum groupoids generalize quantum groups in a way that, briefly,
amounts to replacing the ground field with a pair of anti-isomorphic
base algebras which may be non-commutative.  They appeared in a
variety of situations, for example, as generalized Galois symmetries
for depth 2 inclusions of factors or algebras \cite{boehm:extensions},
\cite{enock:inclusions3}, \cite{enock:actions},
\cite{kadison:extensions}, \cite{kadison:inclusions},
\cite{nikshych:inclusions}, as dynamical quantum groups \cite{donin},
\cite{etingof:qdybe}, \cite{koelink:su2}, or as Tannaka-Krein duals of
certain tensor categories of bimodules \cite{hai:tannaka},
\cite{mccurdy:tannaka}, \cite{pfeiffer:tannaka}, and were studied by
methods of pure algebra \cite{boehm:hopf}, \cite{boehm:weak1},
\cite{boehm:bijective}, \cite{lu:hopf,xu}, \cite{daele:weakmult0},
\cite{daele:weakmult}, \cite{schauenburg:comparison}, \cite{vainer},
category theory \cite{bruguieres:hopf-monads},
\cite{chikhladze:quantum-categories}, \cite{chikhladze:hopf-comonads},
\cite{day:quantum} and operator algebras \cite{enock:inclusions3},
\cite{enock:action}, \cite{lesieur}.

A fundamental feature of quantum groups is that many come in pairs of
mutually dual objects, like finite-dimensional Hopf algebras and their
dual spaces, function and convolution algebras of groups, or
quantisations of function algebras and of universal enveloping
algebras of Lie-Poisson groups \cite{chari}, \cite{korogodski}.  The
same phenomenon was observed for quantum groupoids, where examples
include Hopf algebroids that are fiber-wise finite-dimensional
\cite{kadison:inclusions}, the Galois symmetries for depth 2
inclusions mentioned above, and the function and the convolution
algebras of groupoids.

In the finite-dimensional case, this duality can be explained as
follows.  A (finite-dimensional) quantum group or Hopf algebra is a
vector space $A$ with a multiplication $m\colon A \otimes A \to A$ and
a comultiplication $\Delta \colon A \to A\otimes A$ satisfying a few
conditions, and the dual space $\dA$ turns out to be a Hopf algebra
again with respect to the dual maps $\dual{\Delta}$ and
$\dual{m}$. The key is that one can identify $\dual{A}\otimes
\dual{A}$ with the dual of $A\otimes A$. Iterating this construction,
one obtains the bidual $\dual{(\dual{A})}$ which is naturally
isomorphic to $A$.  A simple example is given by the group algebra
$A=\C G$ of a finite group $G$, where the comultiplication $\Delta$
extends the diagonal map $G \to G\times G$, and the dual $\dual{A}$ is
the function algebra $C(G)$ with the comultiplication
$\dual{\Delta}\colon \delta_{x} \mapsto \sum_{x=yz}\delta_{y}\oo
\delta_{z}$.

Replacing the Hopf algebra with a weak Hopf algebra, one obtains the
corresponding duality for finite-dimensional quantum groupoids. More
generally, Schauenburg \cite{schauenburg} and Kadison and Szlach\'anyi
\cite{kadison:inclusions} extended this duality to Hopf algebroids
under the assumption that the underlying algebra $A$ is projective and
finite as a module over the base algebras, see also \cite{boehm:integrals}.

In the infinite-dimensional case, the situation is more delicate. One
regards two Hopf algebras $A$ and $D$ as mutually dual if one has a
non-degenerate bilinear form on $A\times D$ such that the embeddings
$A \hookrightarrow \dual{D}$ and $D\hookrightarrow \dual{A}$ make the
following diagrams commute,
\begin{align*}
\xymatrix@R=10pt{
  A \ar[r]^(0.4){\Delta_{A}} \ar@{^(->}[d] & A \otimes A \ar@{^(->}[d] \\
  \dual{D} \ar[r]_(0.4){\dual{m_{D}}} & \dual{(D \otimes D)},} &&
\xymatrix@R=10pt{
  D \ar[r]^(0.4){\Delta_{D}} \ar@{^(->}[d] & D \otimes D \ar@{^(->}[d] \\
  \dual{A} \ar[r]_(0.4){\dual{m_{A}}} & \dual{(A \otimes A)},}
\end{align*}
and this definition carries over to Hopf algebroids. But given
a Hopf algebra or Hopf algebroid $A$, the existence of a dual object
$D$ is not known without further assumptions. For example, if $A=\C G$
is the group algebra of an infinite group $G$, the natural choice for
$D$ is the algebra $C_{c}(G)$ of finitely supported functions on $G$
with pointwise multiplication, but the comultiplication $\delta_{x}
\mapsto \sum_{x=yz} \delta_{y} \oo \delta_{z}$ does not take values in
$D \oo D$.

\medskip

For quantum groups, an elegant algebraic and a powerful
operator-algebraic approach to duality were developed by Van Daele
\cite{daele} and Kustermans and Vaes \cite{vaes:1,kustermans:12}. Both
approaches are closely related \cite{kustermans:analytic-1} and share
the main assumption, which is the existence of left- and
right-invariant integrals $\phi$ and $\psi$ on the quantum
group. These integrals correspond to integration of functions on a
group with respect to the Haar measure, are essentially unique,
possess  modular automorphisms and are related by a modular element.
In the purely algebraic theory of Van Daele \cite{daele}, the second
main idea is to allow the algebra $A$ to be non-unital and the
comultiplication $\Delta$ to take values in a multiplier algebra such
that the maps
\begin{align*}
  T_{1} &\colon a\otimes b \mapsto \Delta(a)(1 \otimes b) &&\text{and}
  &
  T_{2} &\colon a \otimes b \mapsto (a \otimes 1)\Delta(b)
\end{align*}
define bijections of $A \otimes A$. For example, the function
algebra $C_{c}(G)$ of an infinite group $G$ fits into this framework.
Then, the subspace $\hat{A} \subseteq \dual{A}$ spanned by all
functionals on $A$ of the form $a\cdot \phi \colon a' \mapsto
\phi(a'a)$ carries the structure of a multiplier Hopf algebra such
that the associated bijections $\hat{T}_{1}$ and $\hat{T}_{2}$ make
the following diagrams commute,
\begin{align*}
  \xymatrix@R=10pt{
    \hat{A} \otimes \hat{A} \ar@{^(->}[d]  \ar[r]^{\hat{T}_{1}} & \hat
    A \otimes \hat A \ar@{^(->}[d] \\
    \dual{(A \otimes A)} \ar[r]_{\dual{(T_{2})}} & \dual{(A\otimes
      A)},  } &&
  \xymatrix@R=10pt{
    \hat{A} \otimes \hat{A} \ar@{^(->}[d]  \ar[r]^{\hat{T}_{2}} & \hat
    A \otimes \hat A \ar@{^(->}[d] \\
    \dual{(A \otimes A)} \ar[r]_{\dual{(T_{1})}} & \dual{(A\otimes
      A)},  }
\end{align*}
and this dual multiplier Hopf algebra $\hat{A}$ has left and right
integrals $\hat \phi$ and $\hat \psi$.  Finally, the bidual $\hat{\hat{A}}$
is naturally isomorphic to $A$ \cite{daele}.

\medskip

In this paper, we extend the algebraic approach of Van Daele to
quantum groupoids and obtain an algebraic duality theory without any
finiteness restrictions, but at the cost of assuming the existence of
integrals.  Let us note that the operator-algebraic approach of
Kustermans and Vaes has been generalised to quantum groupoids by
Enock, Lesieur and Vallin \cite{enock:inclusions3},
\cite{enock:action}, \cite{lesieur} using a formidable array of
sophisticated von Neumann algebra techniques like Tomita-Takesaki
theory.

\medskip

Let us outline our approach, the main results and the organisation of
this article in some detail.

Among the different algebraic approaches to quantum groupoids, we
choose the framework of multiplier Hopf algebroids developed in
\cite{timmermann:regular}. The latter generalize Hopf algebroids in a
similar way like multiplier Hopf algebras generalize Hopf algebras,
and weak multiplier Hopf algebras in a similar way like Hopf
algebroids generalize weak Hopf algebras; see
\cite{daele:comparison}. 

The main definitions and results of \cite{timmermann:regular} are
summarised in \emph{Section \ref{section:multiplier-bialgebroids}}. Roughly,
a \emph{regular multiplier Hopf algebroid} consists of an algebra $A$,
possibly without unit, two anti-isomorphic subalgebras $B,C \subseteq
M(A)$, and a left- and a right-handed comultiplication $\Delta_{B}$
and $\Delta_{C}$ such that the maps
\begin{align} \label{eq:intro-canonical}
  \begin{aligned}
    \Tl &\colon a \otimes b \mapsto \Delta_{B}(b)(a \otimes 1), &
    \Tr&\colon a \otimes b\mapsto \Delta_{B}(a)(1 \otimes b), \\
    \lT &\colon a\otimes b\mapsto (a \otimes 1)\Delta_{C}(b), & \rT
    &\colon a\otimes b \mapsto (1 \otimes b)\Delta_{C}(a)
  \end{aligned}
\end{align}
induce bijections between various tensor products of $A$, where $A$ is
regarded as a module over $B$ or $C$, respectively. This bijectivity
condition is equivalent to the existence of a left- and a right-handed
counit and an antipode.

\medskip

The first step towards the desired duality of such multiplier Hopf
algebroids, taken in \emph{Section \ref{section:integration}}, is to identify
a good notion of left and right integrals on them and to prove that,
similar as in the theory of quantum groups \cite{vaes:1},
\cite{kustermans:12}, \cite{daele:1} and of measured quantum groupoids
\cite{enock:actions}, \cite{lesieur}, such integrals are unique up to
rescaling, admit modular automorphisms, and are related by modular
elements.

The main property of integrals, which is {left-} or
{right-invariance} with respect to the comultiplications
$\Delta_{B}$ and $\Delta_{C}$, can at first sight only be formulated
for $C$- or $B$-bilinear maps $\cphic \colon A \to C$ or $\bpsib
\colon A \to B$, respectively, and amounts to the relations
\begin{align*}
  (\id \otimes \cphic)((a \otimes 1)\Delta_{C}(b)) &= a\cphic(b)
  &&\text{or} &
  (\bpsib \otimes \id)(\Delta_{B}(a)(1 \otimes b)) &= \bpsib(a)b
\end{align*}
for all $a,b\in A$, respectively. These invariance conditions are
studied in \emph{Subsection \ref{subsection:invariant-elements}}.

To construct a dual multiplier Hopf algebroid, we need to complement
such invariant, faithful maps $\cphic\colon A \to C$ and $\bpsib
\colon A \to B$ with faithful functionals $\mu_{B}$ and $\mu_{C}$ on
$B$ and $C$, respectively, which are \emph{quasi-invariant} with respect to
$\cphic$ and $\bpsib$ in the sense that the compositions
$\phi:=\mu_{C} \circ \cphic$ and $\psi :=\mu_{B} \circ \bpsib$ satisfy
\begin{align}\label{eq:intro-ha}
  \phi \cdot A = A \cdot \phi = \psi \cdot A = A \cdot \psi 
\end{align}
as subsets of the dual space $\dual{A}$. Here, $\phi \cdot A$ denotes
the space of all functionals of the form $b \mapsto \phi(ab)$ with $a
\in A$, and the other spaces are defined similarly.  If  $\phi$ and
$\psi$ are faithful, then equation
\eqref{eq:intro-ha} implies the existence of modular automorphisms for
$\phi$ and $\psi$ and of a modular element $\delta=\intd\psi/\intd\phi
\in M(A)$ satisfying $\phi(-\delta)=\psi$. 

In the setting of multiplier Hopf algebras, Van Daele showed that
equation \eqref{eq:intro-ha} is automatically satisfied as soon as
$\phi$ and $\psi$ are non-zero.  In the present context, the situation
is more delicate. Equation \eqref{eq:intro-ha} implies that the
functionals $\phi$ and $\psi$ can be written in the form
\begin{align} \label{eq:factorise}
  \phi &= \mu_{C} \circ \cphic = \mu_{B} \circ \bphi = \mu_{B}
\circ \phib,  & \psi = \mu_{B} \circ \bpsib = \mu_{C} \circ
\cpsi = \mu_{C} \circ \psic
\end{align}
with a left-invariant map $\cphic \colon A \to C$ and a right-invariant
map $\bpsib \colon A \to B$, respectively, and   uniquely determined maps of left or right $B$- or $C$-modules
\begin{align} \label{eq:intro-factorisations-2}
    \bphi &\colon \bA \to \bB, & \phib &\colon \Ab \to \Bb, &
 \cpsi &\colon \cA \to \cC, & \psic &\colon \Ac\to \Cc.
\end{align}

Instead of assuming \eqref{eq:intro-ha}, we reverse the approach and
first fix faithful functionals $\mu_{B}$ and $\mu_{C}$ on the base
algebras $B$ and $C$, respectively. In setting of weak (multiplier)
Hopf algebras, such functionals are given canonically, but in the
present context, there is potentially a large choice, similar as in
the case of classical groupoids. They key assumption that we impose on
$\mu_{B}$ and $\mu_{C}$ is the left or the right counit $\beps \colon
A \to B$ and $\epsc \to A \to C$, respectively, yield the same
functional $\eps:=\mu_{B} \circ \beps = \mu_{C} \circ \epsc$ on
$A$. This assumption also implies that $S^{-2}$ and $S^{2}$ form
modular automorphisms for $\mu_{B}$ and $\mu_{C}$, respectively.
Then, we take the existence of factorisations as in
\eqref{eq:factorise} as the definition of a \emph{left integral
  $\phi$} and a \emph{right integral $\psi$} on $A$. This is done in
\emph{Subsection \ref{subsection:base-weights}}.

The main result of \emph{Section \ref{section:integration}} is that equation
\eqref{eq:intro-ha} holds automatically if the integrals $\phi$ and
$\psi$ are faithful in the ordinary sense and \emph{full} in the sense
that the associated maps in \eqref{eq:intro-factorisations-2} are
surjective. The first step is taken in \emph{Subsection
\ref{subsection:uniqueness}}, where we also show that integrals are
unique up to rescaling, or, more precisely, that the maps
\begin{align*}
 M(B) &\to \dual{A}, \  x \mapsto x\cdot \phi = \phi(-x), &
 M(C) &\to \dual{A}, \ y\mapsto y\cdot \psi  =\psi(-y)
\end{align*}
define bijections between $M(B)$ or $M(C)$ on one side and the space
of left or right integrals on $A$, respectively, on the other
side. The second step is taken in \emph{Subsection
\ref{subsection:modular-automorphism}} and involves the study of
natural convolution operators associated to elements of the space
\eqref{eq:intro-ha}. 

Given a full and faithful left integral $\phi$, the compositions $\psi
= \phi \circ S^{-1}$ and $\psi^{\dag}:=\phi \circ S$ are right
integrals, and equation \eqref{eq:intro-ha} implies that there exist
modular elements $\delta,\delta^{\dag} \in M(A)$ such that $\psi =
\phi \cdot \delta$ and $\psi^{\dag} = \delta^{\dag } \cdot \phi$.   In
\emph{Subsection \ref{subsection:modular-element}}, we study the interplay of
these modular elements with the comultiplication, antipode and counit,
and establish all the formulas that one would expect from the theory
of multiplier Hopf algebras.

Finally, we show that a full integral is automatically faithful if $A$ is
projective as a module over the base algebras $B$ and $C$. This result
generalizes the corresponding result of Van Daele for multiplier Hopf
algebras and is proved in \emph{Subsection \ref{subsection:faithful}}.

\medskip

The main result of this article, which is the duality of measured regular
multiplier Hopf algebroids, is established in \emph{Section
\ref{section:duality}}. 

We define a \emph{measured regular multiplier Hopf algebroid} to be a
regular multiplier Hopf algebroid equipped with faithful functionals
$\mu_{B},\mu_{C}$ as above and full and faithful left and right
integrals $\phi$ and $\psi$.  Given such a measured regular multiplier
Hopf algebroid, we equip the subspace $\hat A:=\phi \cdot A =A\cdot
\phi = \psi \cdot A = A\cdot \psi$ of $\dual{A}$ with a convolution
product $(\upsilon,\omega) \mapsto \upsilon \ast \omega$ such that,
roughly,
\begin{align*}
  (\upsilon \oo \omega) \circ \Delta_{B}= \upsilon \ast \omega =
  (\upsilon \oo \omega) \circ \Delta_{C},
\end{align*}
and with natural embeddings of $\hat C:=B$ and $\hat B:=C$ into
$M(\hat A)$; see \emph{Subsection \ref{subsection:dual-graphs}}.  Note that
the compositions on the left and on the right hand side do not make
sense as they stand because $\Delta_{B}$ and $\Delta_{C}$ take values
in multiplier algebras of certain balanced tensor products of $A$. The
main difficulty, however, is to show that the two compositions
coincide. This problem is solved in \emph{Subsection
\ref{subsection:modular-automorphism}} by a careful analysis of the
left and right convolution operators on $A$ given by the compositions
\begin{align*}
(\upsilon \oo \id) \circ \Delta_{B}, &&
 (\upsilon \oo \id) \circ \Delta_{C}, &&
 (\id \oo \omega) \circ \Delta_{B}, &&
 (\id \oo \omega) \circ \Delta_{C},
\end{align*}
and here the assumption $\mu_{B} \circ \beps = \mu_{C} \circ \epsc$
becomes crucial.

 We then equip $\hat A \subseteq \dual{A}$ with comultiplications
$\hat\Delta_{\hat B}$ and $\hat\Delta_{\hat C}$ such that the
 maps
\begin{align*}
  \hTl &\colon \upsilon \oo \omega \mapsto
  \hat\Delta_{\hat B}(\omega)(\upsilon \otimes 1), & \hTr &\colon
  \upsilon \oo \omega \mapsto \hat\Delta_{\hat B}(\upsilon)(1 \otimes
  \omega), \\
  \hlT &\colon \upsilon \oo \omega \mapsto (\upsilon \oo
  1)\hat\Delta_{\hat C}(\omega), &
  \hrT &\colon \upsilon \oo \omega \mapsto (1 \oo
  \omega)\hat\Delta_{\hat C}(\upsilon)
\end{align*}
dualise the maps in \eqref{eq:intro-canonical}. More precisely, we
first show that the transposes of the maps in
\eqref{eq:intro-canonical} restrict to bijections between certain
balanced tensor products of $\hat A$ with itself, and then that these
restrictions form multiplicative pairs in the sense of
\cite{timmermann:regular} so that, by results in
\cite{timmermann:regular}, they correspond to comultiplications
$\hat\Delta_{\hat B}$ and $\hat\Delta_{\hat C}$ as stated above. 
This is the contents of \emph{Subsection \ref{subsection:dual-pairs}}.

To obtain the full duality result, we  show \emph{in Subsection \ref{subsection:dual-integrals}} that, as in the case of multiplier Hopf algebras
\cite{daele:1}, the maps
\begin{align*}
  \hat\psi &\colon a \cdot \phi \mapsto \eps(a) &&\text{and} &
  \hat\phi &\colon \psi \cdot a \mapsto \eps(a)
\end{align*}
form a right and a left integral on $\hat A$ so that one obtains a
measured regular multiplier Hopf algebroid again, and 
that the natural map $A \to \dual{(\hat A)}$ identifies $A$ with the
bidual $\hat{\hat A}$ as a measured regular multiplier Hopf
algebroid. Finally, we study the duality between the multiplication
and the comultiplication on $A$ and $\hat A$ on the level of the
multiplier algebras $M(A)$ and $M(\hat A)$; see \emph{Subsection \ref{subsection:multipliers}}.

\medskip

Two examples of mutually dual measured multiplier Hopf algebroids are
discussed in the final \emph{Section \ref{section:examples}} --- the function
algebra and the convolution algebra of an \'etale, locally compact,
Hausdorff groupoid, and a two-sided crossed product which  appeared in
\cite{connes:rankin}.

\medskip

The material developed in this article raises several questions that
will have to be addressed separately.

Given the theory of integrals developed here, one should be able to
construct (operator-algebraic) measured quantum groupoids out of
suitable (algebraic) measured multiplier Hopf $*$-algebroids,
similarly as it was done for multiplier Hopf algebras by Kustermans
and Van Daele in  \cite{kustermans:analytic-2} and \cite{kustermans:algebraic}, and for dynamical quantum groups in \cite{timmermann:dynamical}.
An important assumption in the theory of measured quantum groupoids is
that the modular automorphism groups of the left and of the right
integral commute, and it would be desirable to study the implications
of this assumption in the present context.

An important question concerns the implications of existence of
integrals on the theory of corepresentations. It seems natural to
expect that under suitable assumptions on the base algebras $B$ and
$C$, many results on the corepresentation theory of compact
quantum groups carry over to measured regular Hopf $*$-algebroids.

One should also clarify the relation to the duality obtained by
Schauenburg \cite{schauenburg}, Kadison and Szlach\'anyi
\cite{kadison:inclusions} and Böhm \cite{boehm:integrals} in the case
where $A$ is unital and finitely generated projective as a module over
$B$ and $C$.

\subsection*{Preliminaries}

We shall use the following conventions and terminology.

All algebras and modules will be complex vector spaces and all
homomorphisms will be linear maps, but much of the theory
developed in this article applies in wider generality.

The identity map on a set $X$ will be denoted by $\iota_{X}$ or simply
$\iota$.

Let $B$ be an algebra, not necessarily unital. We denote by $B^{\op}$
the \emph{opposite algebra}, which has the same underlying vector
space as $B$, but the reversed multiplication.

Given a right module $M$ over $B$, we write $M_{B}$ if we want to
emphasize that $M$ is regarded as a right $B$-module. We call $M_{B}$
\emph{faithful} if for each non-zero $b\in B$ there exists an $m\in M$
such that $mb$ is non-zero, \emph{non-degenerate} if for each non-zero
$m\in M$ there exists a $b \in B$ such that $mb$ is non-zero,
\emph{idempotent} if $MB=M$, and we say that $M_{B}$ \emph{has local
  units in $B$} if for every finite subset $F\subset M$ there exists a
$b\in B$ with $mb=m$ for all $m\in F$. Note that the last property
implies the preceding two. We denote by $\dual{M_{B}}
:=\Hom(M_{B},B_{B})$ the dual module, and by $\dual{f} \colon
\dual{N_{B}} \to \dual{M_{B}}$ the dual of a morphism $f\colon M\to N$
of right $B$-modules, given by $\dual{f}(\chi)= \chi \circ f$.  We use
the same notation for duals of vector spaces and of linear maps. We
furthermore denote by $L(M_{B}) := \Hom(\Bb,M_{B})$ the space of
\emph{left multipliers} of the module $M_{B}$.

For left modules, we obtain the corresponding notation and terminology
by identifying left $B$-modules with right $B^{\op}$-modules.  We
denote by $R(_{B}M):=\Hom(\bB,{_{B}M})$ the space of \emph{right
  multipliers} of a left $B$-module $_{B}M$.

We write $B_{B}$ or ${_{B}B}$ when we regard $B$ as a right or left
module over itself with respect to right or left multiplication. We
say that the algebra $B$ is \emph{non-degenerate}, \emph{idempotent},
or \emph{has local units} if the modules ${_{B}B}$ and $B_{B}$ both
are non-degenerate, idempotent or both have local units in $B$,
respectively. Note that the last property again implies the preceding
two.

We denote by $L(B)=\End(B_{B})$ and
$R(B)=\End({_{B}B})^{\op}$ the algebras of left or right
multipliers of $B$, respectively, where the multiplication in the
latter algebra is given by $(fg)(b):=g(f(b))$. Note that $B_{B}$ or
${_{B}B}$ is non-degenerate if and only if the natural map from $B$ to
$L(B)$ or $R(B)$, respectively, is injective. If $B_{B}$ is
non-degenerate, we define the multiplier algebra of $B$ to be the
subalgebra $M(B) :=\{ t\in L(B) : Bt\subseteq B\} \subseteq L(B)$,
where we identify $B$ with its image in $L(B)$. Likewise we could
define $M(B) =\{ t\in R(B) : tB \subseteq B\}$ if ${_{B}B}$ is
non-degenerate.  If both definitions make sense, that is, if $B$ is
non-degenerate, then they evidently coincide up to a natural
identification, and a multiplier is given by a pair of maps
$t_{R},t_{L}\colon B\to B$  satisfying $t_{R}(a)b=at_{L}(b)$ for all
$a,b\in B$.

Given a left or right $B$-module $M$ and a space $N$, we regard the
space of linear maps from $M$ to $N$ as a right or left $B$-module, where
$(f \cdot b)(m)=f(bm)$ or $(b\cdot f)(m)=f(mb)$ for all maps $f$ and
all elements $b\in B$ and $m\in M$, respectively. 

In particular, we regard the dual space $\dual{B}$ of a
non-degenerate, idempotent algebra $B$ as a bimodule over $M(B)$,
where $(a \cdot \omega \cdot b)(c)=\omega(bca)$, and call a functional
$\omega \in \dual{B}$ \emph{faithful} if the maps $B \to \dual{B}$
given by $d \mapsto d\cdot\omega$ and $d \mapsto\omega \cdot d$ are
injective, that is, $\omega(dB)\neq 0$ and $\omega(Bd) \neq 0$
whenever $d\neq 0$.

We say that a functional $\omega \in \dual{B}$ \emph{admits a modular
  automorphism} if there exists an automorphism $\sigma$ of $B$ such
that $\omega(ab)=\omega(b\sigma(a))$ for all $a,b\in B$. One easily
verifies that this condition holds if and only if $B\cdot \omega =
\omega \cdot B$, and that then $\sigma$ is characterised by the
relation $\sigma(b) \cdot \omega = \omega \cdot b$ for all $b\in B$.

Assume that $B$ is a $*$-algebra. We call a functional $\omega \in
\dual{B}$ \emph{self-adjoint} if it coincides with $\omega^{*}=\ast
\circ \omega \circ \ast$, that is, $\omega(a^{*})=\omega(a)^{*}$ for
all $a\in B$, and \emph{positive} if additionally $\omega(a^{*}a)\geq
0$ for all $a\in A$.

\section{Regular multiplier Hopf algebroids}
\label{section:multiplier-bialgebroids}

This section summarises the definition and main properties of regular
multiplier Hopf algebroids. The latter were introduced in
\cite{timmermann:regular} as non-unital generalizations of Hopf
algebroids, and are special cases of multiplier bialgebroids. As such,
they consist of a left and a right multiplier bialgebroid, which in
turn consist of a left or a right quantum graph with a left or right
comultiplication that are compatible in a certain sense.

\subsection{Left multiplier bialgebroids}

A \emph{left quantum graph} is a tuple $\mathcal{A}_{B}=(B,A,s,t)$,
where (1) $B$ and $A$ are algebras and $A_{A}$ is non-degenerate as a
right $A$-module, (2) $s\colon B \to M(A)$ is a homomorphism and
$t\colon B\to M(A)$ is an anti-homomorphism such that their images
commute, (3) the $B$-modules $\sA$ and $\tA$ are faithful,
non-degenerate and idempotent, and (4a) the $A$-module $A_{A}$ has
local units in $A$ or (4b) the $B$-modules $\sA$ and $\tA$ are flat.
Here, we write $\sA$, $\At$, $\tA$ or $\As$ when we regard $A$ as a
left or right $B$-module via $x \cdot a :=s(x)a$, $y \cdot a:=at(y)$,
$a \cdot x :=as(x)$ or $ a\cdot y:=t(y)a$, respectively.  By
assumption on $\sA$ and $\tA$, the maps $s$ and $t$ extend to maps
from $M(B)$ to $L(A)$.

Given a left quantum graph $\mathcal{A}_{B}=(B,A,s,t)$, we denote by $\AlA:=\sA \oo \tA$ the
tensor product of $B$-modules, or rather $B^{\op}$-modules, regard this as  a right $M(A) \oo M(A)$-module in
the canonical way, and denote by $\Lreg(\AltkA) \subseteq \End(\AlA)$
the subalgebra consisting of all maps $v$ such that
$\lmultl(a)(1 \oo b) =  \lmult(a\oo b) = \lmultr(b)(a \oo 1)$
for some well-defined maps $\lmultl,\lmultr \colon A\to \AlA$ and all $a,b
\in A$. A \emph{left comultiplication}
for the left quantum graph  $\mathcal{A}_{B}$ is a homomorphism
$\Delta_{B} \colon A \to \Lreg(\AltkA)$ satisfying 
\begin{align}
  \label{eq:left-comult-module}
  \Delta_{B}(s(x)t(y)as(x')t(y')) &= (t(y) \otimes
  s(x))\Delta_{B}(a)(t(y') \otimes s(x')), \\ \label{eq:left-comult-coass}
((\Delta_{B} \otimes \id)(\Delta_{B}(b)(1 \otimes c)))(a
    \otimes 1 \otimes 1) &= ((\id \otimes \Delta_{B})(\Delta_{B}(b)(a
    \otimes 1)))(1 \otimes 1 \otimes c) 
\end{align}
for all $ a,b,c\in A$ and $x,x',y,y' \in B$.  

A \emph{left multiplier bialgebroid} is a left quantum graph
$\mathcal{A}_{B}$ with a left comultiplication $\Delta_{B}$ as above.
Associated to such a left comultiplication are the
\emph{canonical maps}
\begin{align} \label{eq:left-galois-maps}
    \Tl \colon \ATA &\to \AlA,  &
    \Tr \colon \AsA &\to \AlA,  \\
 \label{eq:left-galois-definition}
 \Tl(a \oo b)&= \Delta_{B}(b)(a \oo 1), & \Tr(a \oo b) &=
 \Delta_{B}(a) (1\oo b),
\end{align}
where $\ATA:=\tA \oo \At$ and $\AsA:=\As\oo\sA$. These maps satisfy
\begin{align} \label{eq:left-galois-module}
  \begin{aligned}
    \Tr(s(x)t(y)at(y') \oo t(z)b) &= (t(y) \oo s(x))\Tr(a\oo
    b)(t(y')s(z) \oo 1), \\ 
    \Tl(s(z)a \oo t(y)s(x)bs(x')) &= (t(y) \oo s(x))\Tl(a \oo b)(1\oo
    t(z)s(x'))
  \end{aligned}
\end{align}
for all $x,x',y,y',z\in B$, $a,b\in A$ and make the following diagrams
commute:
\begin{gather} \label{dg:left-galois-1}
\xymatrix@R=15pt{
      \ATAsA \ar[r]^{\id \oo \Tr} \ar[d]_{\Tl \oo \id}
      & \ATAlA \ar[d]^{\mult\Sigma \oo \id} \\
      \AlAsA \ar[r]^(0.55){\id \oo \mult} & \AlA,
    }
\qquad
  \xymatrix@R=15pt{
    \ATAsA \ar[r]^{\id \oo \Tr} \ar[d]_{\Tl \oo \id}
    &
    \ATAlA  \ar[d]^{\Tl \oo \id} \\
    \AlAsA  \ar[r]^{\id \oo \Tr}  & \AlAlA,
} \\ % \intertext{and}
 \label{dg:left-galois-2}
  \xymatrix@C=15pt@R=15pt{
    \AsAsA \ar[r]^(0.55){\mult \oo \id}  \ar[d]_{\id \oo
      \Tr} & \AsA \ar[r]^{\Tr} & \AlA, \\
    \AsoAlsA \ar[rr]^{(\Tr)_{13}}
&&    \sAsAotA \ar[u]_{\mult \oo \id} }
\end{gather}
where $(\Tr)_{13}=(\Sigma \oo \id)(\id \oo \Tr)(\Sigma \oo \id)$.
Conversely, each pair of maps $(\Tl,\Tr)$ as in
\eqref{eq:left-galois-maps} that satisfy \eqref{eq:left-galois-module}
and make the diagrams \eqref{dg:left-galois-1} and
\eqref{dg:left-galois-2} commute define a left comultiplication
$\Delta_{B}$ via \eqref{eq:left-galois-definition}.

A \emph{left counit}  for a left multiplier bialgebroid
as above is
map
  $\beps \in \Hom(\sA,\bB) \cap
  \Hom(\tA,\Bb)$
  that makes the following diagrams commute,
\begin{gather} \label{dg:left-counit}
      \xymatrix@C=10pt@R=12pt{ \AsA \ar[r]^{\Tr}
        \ar[d]_{\mult} & \AlA
        \ar[d]^{\beps \oo \id}  \\
        A & \bB \otimes \tA \ar[l],} \qquad
      \xymatrix@C=10pt@R=12pt{ \ATA \ar[d]_{\mult\Sigma}
        \ar[r]^{\Tl} & \AlA
        \ar[d]^{\id \oo \beps} \\
        A & \sA \otimes \Bb \ar[l],} \qquad
      \xymatrix@R=12pt@C=15pt{ \AsA \ar[r]^{\mult}
        \ar[d]_{\id \oo \beps} & A \ar[d]^{\beps} &
        \AtA \ar[l]_{\mult} \ar[d]^{\iota
          \oo \beps} \\
 A
        \ar[r]^{\beps} & B, & A \ar[l]_{\beps}  }
    \end{gather}
    where the lower horizontal maps from $\bB\oo \tA$, $\sA \oo \Bb$,
    $\At \oo \Bb$ and $\At \oo \Bb$ to $A$ are given by $y\oo a \mapsto
  t(y)a$, $a \oo x \mapsto s(x)a$, $a \oo x \mapsto as(x)$ and $a \oo
  y \mapsto at(y)$, respectively, and $m$ denotes the multiplication
  map $a \oo b \mapsto ab$.

  \subsection{Right multiplier bialgebroids} The notion of a right
  multiplier bialgebroid is opposite to the notion of a left
  multiplier bialgebroid in the sense that in all assumptions, left
  and right multiplication are reversed.

  Accordingly, a \emph{right quantum graph} is a tuple
  $\mathcal{A}_{C}=(C,A,s,t)$, where (1) $C$ and $A$ are algebras and
  ${_{A}A}$ is non-degenerate as a left $A$-module, (2) $s\colon C \to
  M(A)$ is a homomorphism and $t\colon C\to M(A)$ is an
  anti-homomorphism such that their images commute, (3) the
  $C$-modules $\As$ and $\At$ are faithful, non-degenerate and
  idempotent, and (4a) the $A$-module ${_{A}A}$ has local units in $A$
  or (4b) the $C$-modules $\As$ and $\At$ are flat. As before, we write
  $\sA$, $\At$, $\tA$ or $\As$ when we regard $A$ as a left or right
  $C$-module via $x \cdot a :=s(x)a$, $y \cdot a:=at(y)$, $a \cdot x
  :=as(x)$ or $ a\cdot y:=t(y)a$, respectively.
  
  Given such a right quantum graph $\mathcal{A}_{C}=(C,A,s,t)$, we
  denote by $\ArA:=\At \oo \As$ the tensor product of $B$-modules, regard
  this as a left $M(A)\oo M(A)$-module, and denote by $\Rreg(\ArtkA)
  \subseteq \End(\ArA)^{\op}$ the subalgebra consisting of all maps
  $\rmult$ such that
  $(1 \oo b)  \rmultl(a) = \rmult(a \oo b) = (a \oo 1)\rmultr(b)$
for some well-defined maps $\rmultl,\rmultr\colon A\to \ArA$ and all $a,b
\in A$. A \emph{right comultiplication} for the right quantum graph
$\mathcal{A}_{C}$  is a homomorphism $\Delta_{C}\colon A
  \to \Rreg(\ArtkA)$ satisfying
  \begin{align}
    \label{eq:right-comult-module}
    \Delta_{C}(t(x)s(y)at(x')s(y')) &= (s(y) \otimes
    t(x))\Delta_{C}(a)(s(y') \otimes t(x'))
    \\ \label{eq:right-comult-coass}
(a \otimes 1 \otimes 1) ((\Delta_{C} \otimes \id)((1
    \otimes c)\Delta_{C}(b))) &= (1 \otimes 1 \otimes c)((\id
    \otimes \Delta_{C})((a \otimes 1)\Delta_{C}(b)))
  \end{align}
  for all $a,b,c \in A$ and $x,x', y,y'\in C$.  A \emph{right
    multiplier bialgebroid} is a right quantum graph with a right
  comultiplication.  Associated to such a right comultiplication are the
 \emph{canonical maps}
  \begin{align} \label{eq:right-galois-maps}
    \lT \colon \AtA &\to \ArA, &
    \rT \colon \ASA &\to \ArA, \\
\label{eq:right-galois-definition}
  \lT(a\oo b) &= (a \oo 1)\Delta_{C}(b), &
  \rT(a \oo b) &= (1 \oo b)\Delta_{C}(a),
  \end{align}
  where $\AtA = \At \oo \tA$ and $\ASA = \sA \oo \As$.  These maps
  satisfy analogues of the relations \eqref{eq:left-galois-module} and
  make similar diagrams to those in \eqref{dg:left-galois-1} and
  \eqref{dg:left-galois-2} commute. Conversely, each pair of maps
  $(\lT,\rT)$ as in \eqref{eq:right-galois-maps} that satisfy the
  analogues of \eqref{eq:left-galois-module}, \eqref{dg:left-galois-1}
  and \eqref{dg:left-galois-2} define a right comultiplication
  $\Delta_{C}$ via \eqref{eq:right-galois-definition}.

A \emph{right counit} for a right multiplier bialgebroid
$(\mathcal{A}_{C},\Delta_{C})$ is a 
map
  $\epsc \in \Hom(\As,\Cc) \cap
  \Hom(\At,\cC)$
  that makes the following diagrams commute,
\begin{gather} \label{eq:right-counit}
    \xymatrix@C=15pt@R=15pt{ \ATA \ar[r]^{\rT} \ar[d]_{\mult \Sigma} & \ArA
      \ar[d]^{\epsc \oo \id}  \\
      A & \cC \otimes \As \ar[l],} \quad
    \xymatrix@C=15pt@R=15pt{ \AsA \ar[d]_{\mult} \ar[r]^{\lT} &
      \ArA
      \ar[d]^{\id \oo \epsc} \\
      A & \At \otimes \Cc \ar[l],} \quad
      \xymatrix@R=18pt@C=10pt{ \AsA \ar[r]^{\mult}
        \ar[d]_{\epsc \oo \id} & A \ar[d]^{\epsc} &
        \AtA \ar[l]_{\mult} \ar[d]^{\epsc \oo \id} \\
 A
        \ar[r]^{\epsc} & C, & A \ar[l]_{\epsc}  }
    \end{gather}
  where the lower horizontal maps to $A$ are the natural ones.

\subsection{Regular multiplier Hopf algebroids}
\label{subsection:hopf-algebroids}
A left and a right quantum graph $\mathcal{A}_{B}=(B,A,s_{B},t_{B})$
and $\mathcal{A}_{C}=(C,A,s_{C},t_{C})$ are \emph{compatible} if the
underlying algebra $A$ is the same and the relations
$t_{B}(B)=s_{C}(C)$ and $t_{C}(C)=s_{B}(B)$ hold. In that case, we
identify $B$ and $C$ with their images $s_{B}(B)$ and $s_{C}(C)$ in
$M(A)$, so that the maps $t_{B}$ and $t_{C}$ get identified with the
anti-isomorphisms $S_{B}:=t_{B} \circ s_{B}^{-1} \colon B\to C$ and
$S_{C}:=t_{C} \circ s_{C}^{-1} \colon C \to B$, respectively.
Thus, compatible left and right quantum graphs are given by a
non-degenerate algebra $A$ and subalgebras $B,C \subseteq M(A)$ such
that $A$ is  idempotent (and automatically non-degenerate and faithful) as
a left and as a right module over $B$ and over $C$, and such that the
assumptions (4a) or (4b) above concerning local units or flatness hold.

A \emph{multiplier bialgebroid} consists of a left multiplier
bialgebroid $(\mathcal{A}_{B},\Delta_{B})$ and a right multiplier
bialgebroid $(\mathcal{A}_{C},\Delta_{C})$, where 
$\mathcal{A}_{B}$ and $\mathcal{A}_{C}$ are compatible and 
\begin{align} \label{eq:compatible} 
     \begin{aligned}
       ((\Delta_{B} \oo \id)((1 \oo c)\Delta_{C}(b)))(a \oo 1 \oo 1)
       &= (1 \oo 1 \oo c)((\id \oo
       \Delta_{C})(\Delta_{B}(b)(a \oo 1))), \\
       (a \oo 1 \oo 1)((\Delta_{C} \oo \id)(\Delta_{B}(b)(1 \oo c)))
       &= ((\id \oo \Delta_{B})((a \oo 1)\Delta_{C}(b)))(1 \oo 1 \oo
       c)
     \end{aligned}
\end{align}
 for all $a,b,c \in A$ or, equivalently,  the following diagrams commute:
\begin{align} \label{dg:compatible}
  \xymatrix@R=15pt{ \AcAbA \ar[r]^{\id \oo \Tr} \ar[d]_{\lT \oo
      \id} & \AcAlA \ar[d]^{\lT \oo \id}
&&    \ACABA \ar[r]^{\id \oo \rT} \ar[d]_{\Tl \oo \id} &
    \ar[d]^{\Tl \oo \id} \ACArA
    \\
    \ArAbA \ar[r]^{\id \oo \Tr} & \ArAlA,  && 
    \AlABA \ar[r]^{\id \oo \rT} & \AlArA.  }
\end{align}
We call
$\mathcal{A}:=((\mathcal{A}_{B},\Delta_{B}),(\mathcal{A}_{C},\Delta_{C}))$
a \emph{regular multiplier Hopf algebroid} if the canonical maps
$\Tl,\Tr,\lT,\rT$ are bijective. The main result of
\cite{timmermann:regular} says that this condition is equivalent to
existence of an invertible antipode, which is an anti-homomorphism $S
\colon A\to A$ satisfying the following conditions:
  \begin{enumerate}
  \item $S(xyax'y')=S_{C}(y')S_{B}(x')S(a)S_{C}(y)S_{B}(x)$ for all
    $x,x'\in B, y,y' \in C, a \in A$;
  \item there exist a left counit $\beps$ for
    $(\mathcal{A}_{B},\Delta_{B})$ and a right counit $\epsc$ for
    $(\mathcal{A}_{C},\Delta_{C})$ such that the following diagrams commute,
\begin{align} \label{dg:antipode}
        \xymatrix@C=10pt@R=12pt{ \AbA \ar[rr]^{\Tr} \ar[d]_{S_{C}
            \circ \epsc
            \oo \id} && \AlA
          \ar[d]^{S \oo \id}  \\
          \Bb \oo \bA \ar[r] & A & \ar[l] M(A)
          \ooc A} \qquad \xymatrix@C=10pt@R=12pt{\AcA
          \ar[rr]^{\lT} \ar[d]_{\id \oo
            S_{B}\circ \beps} && \ArA \ar[d]^{\id \oo S} \\
          \Ac \oo \cC \ar[r] & A& \ar[l] A
          \oob M(A)}
    \end{align}
  where the lower horizontal maps are given by multiplication. 
\end{enumerate}
The counits $\beps$ and $\epsc$ then are uniquely determined.

The canonical maps $\Tl,\Tr,\lT,\rT$ and the antipode $S$ of a regular
multiplier Hopf algebroid $\mathcal{A}$ as above are related by the
following commutative diagrams,
\begin{gather} \label{dg:galois-inverse}
        \xymatrix@R=10pt{\ABA \ar[d]_{\rT} \ar[r]^{\id \oo S} & \AlA,
          \\
          \ArA \ar[r]_{\id \oo S} & \AbA \ar[u]_{\Tr} } \qquad
        \xymatrix@R=10pt{\ACA \ar[d]_{\Tl} \ar[r]^{S \oo \id} & \ArA,
          \\
          \AlA \ar[r]_{S \oo \id} & \AcA \ar[u]_{\lT} } \\
\label{dg:galois-aux}
      \xymatrix@C=20pt@R=5pt{ \AcA \ar[dd]_{\lT} \ar[rd]_{\Tl \circ\Sigma}
        \ar[rr]^{(S \oo \id) \circ\Sigma} &&
        \ArA  \\
        & \AlA  &\\
        \ArA \ar[rr]_{(S \oo \id) \circ\Sigma} && \ABA \ar[lu]_{\Tr \circ\Sigma}
        \ar[uu]_{\rT} } \qquad \xymatrix@C=20pt@R=5pt{ \AbA \ar[dd]_{\Tr}
        \ar[rr]^{(\id \oo S) \circ\Sigma}
        \ar[rd]_{\rT \circ\Sigma} &  & \AlA \\
        & \ArA  & \\
        \AlA \ar[rr]_{(\id \oo S) \circ\Sigma} && \ACA 
        \ar[uu]_{\Tl} \ar[lu]_{\lT \circ\Sigma} } \\
 \label{dg:galois-aux2}
    \xymatrix@C=30pt@R=10pt{
      \ACA \ar[d]_{\Tl} \ar[r]^{\Sigma \circ(S \oo S)} &       \ABA
      \ar[d]^{\rT} \\
      \AlA \ar[r]_{\Sigma \circ (S \oo S)} & \ArA, }
  \qquad
    \xymatrix@C=30pt@R=10pt{
      \AcA \ar[d]_{\lT} \ar[r]^{\Sigma \circ(S \oo S)} &       \AbA
      \ar[d]^{\Tr} \\
      \ArA \ar[r]_{\Sigma \circ (S \oo S)} & \AlA,  } 
  \end{gather}
see Theorem 6.8, Proposition 6.11 and Proposition 6.12 in \cite{timmermann:regular}.

% Reversing the multiplication or comultiplication of a multiplier
% bialgebroid $\mathcal{A}$ as above, one 
% obtains an \emph{opposite multiplier bialgebroid}  $\mathcal{A}^{\op} = ((\mathcal{A}_{B^{\op}}^{\op},\Delta_{C}^{\op}),
%   (\mathcal{A}_{C^{\op}}^{\op},\Delta_{B}^{\op}))$, where
% \begin{align*}
%   \mathcal{A}_{C}^{\op} &= (B^{\op},A^{\op},\iota,(S_{C}^{\op})^{-1}),
% &  \mathcal{A}_{C}^{\op} &= (C^{\op},A^{\op},\iota,(S_{B}^{\op})^{-1}),
% \end{align*}
% a \emph{co-opposite multiplier  bialgebroid} $\mathcal{A}^{\co} = ((\mathcal{A}_{C}^{\co},\Delta_{C}^{\co}),
%   (\mathcal{A}_{B}^{\co},\Delta_{B}^{\co}))$, where
% \begin{align*}
%   \mathcal{A}_{C}^{\co} &= (C,A,\iota,S_{B}^{-1}), & \Delta_{C}^{\co}
%   &= \Sigma \circ \Delta_{B}, & \mathcal{A}_{B}^{\co} &=
%   (B,A,\iota,S_{C}^{-1}), & \Delta_{B}^{\co} &= \Sigma \circ
%   \Delta_{C},
% \end{align*}
% and a \emph{bi-opposite multiplier bialgebroid}
%   $\mathcal{A}^{\op,\co} = (\mathcal{A}^{\co})^{\op} = (\mathcal{A}^{\op})^{\co}$;
% see Propositions 3.6, 5.10, 5.18 and Remark 6.5 in
% \cite{timmermann:regular}. If $\mathcal{A}$ is a regular multiplier
% Hopf algebroid, then so are $\mathcal{A}^{\op}$, $\mathcal{A}^{\co}$
% and $\mathcal{A}^{\op,\co}$.

Let $\mathcal{A}_{B}=(B,A,\id,S_{B})$ and
$\mathcal{A}_{C}=(C,A,\id,S_{C})$ be compatible left and right quantum
graphs.  An involution $*$ on $A$ is \emph{admissible with respect to
  $\mathcal{A}_{B}$ and $\mathcal{A}_{C}$} if (1) $B$ and $C$ are
$^{*}$-subalgebras of $M(A)$ and (2) $S_{B}\circ \ast \circ S_{C}
\circ \ast =\id_{C}$ and $S_{C}\circ \ast \circ S_{B} \circ \ast
=\id_{B}$.  A \emph{multiplier $^{*}$-bialgebroid/multiplier Hopf
  $^{*}$-algebroid} is a multiplier bialgebroid/regular multiplier
Hopf algebroid
$\mathcal{A}=((\mathcal{A}_{B},\Delta_{B}),(\mathcal{A}_{C},\Delta_{C}))$
with an admissible involution on the underlying algebra $A$ satisfying
$(*\oo *)\circ \Delta_{B}\circ *=\Delta_{C}$.  In that case, the left
and right counits $\beps,\epsc$ and the antipode $S$ satisfy
\begin{align}
  \label{eq:involutions}
 \epsc\circ * &= *\circ S_{B}\circ \beps, & \beps\circ *&=*\circ
S_{C}\circ \epsc, & S\circ *\circ S \circ *&=\id_{A}.
\end{align}

We end this review with a few observations and definitions not
contained in \cite{timmermann:regular}.
\begin{lemma} \label{lemma:counits-antipode}
  Let $\mathcal{A}$ be a regular multiplier Hopf algebroid with
  antipode $S$ and left and right counits $\beps$ and $\epsc$,
  respectively. Then  $\epsc \circ 
  S =S \circ \beps$ and $\beps \circ S = S \circ \epsc$.
\end{lemma}
\begin{proof}
  Essentially, this follows from the fact that $S$ is an
  anti-isomorphism of $\mathcal{A}$ with its bi-opposite. More
  explicitly, the outer, upper, left and right cells of the
  diagram
  \begin{align*}
    \xymatrix@C=10pt@R=10pt{ \ACA \ar[dd]_{m \circ\Sigma} \ar[rd]^{\Tl}
      \ar[rrrr]^{\Sigma\circ (S \oo S)} &&&& \ABA
      \ar[ld]_{\rT} \ar[dd]^{m\circ \Sigma}  \\
      &  \AlA \ar[rr]^{\Sigma \circ (S\oo S)} \ar[ld]^{\id \oo \epsb} &\quad&
      \ArA  \ar[rd]_{ \epsc \oo \id} & \\
      A \ar[rrrr]^{S} &&&& A }
  \end{align*}
commute,  whence the lower cell commutes as well, whence $\epsc \circ S = S
  \circ \beps$, and the second relation follows similarly.
\end{proof}

In our study of integration, we will make use of the following conditions.
\begin{definition}
  We call a regular multiplier Hopf algebroid $\mathcal{A}$
  \emph{projective} or \emph{flat} if all of the modules
  $\bA,\Ab,\cA,\Ac$ are projective or flat, respectively.
\end{definition}
Using the antipode, one easily verifies that the module
$\bA$ or $\Ab$ is projective/flat if and only if $\Ac$ or $\cA$ is
projective/flat, respectively.

We next show that the counits of a regular multiplier Hopf algebroid
can be assumed to be surjective without much loss of
generality. Denote by $\dbA,\dAb,\dcA,\dAc$ the dual modules of
$\bA,\Ab,\cA,\Ac$, respectively; for example, $ \dbA = \Hom(\bA,\bB)$
and $\dAb =\Hom(\Ab,\Bb)$.
\begin{proposition} \label{proposition:counits-full}
  Let $\mathcal{A}$ be a flat regular multiplier Hopf algebroid. Then
  \begin{enumerate}
  \item $B_{0}:=\beps(A) \subseteq B$ and $C_{0}:=\epsc(A) \subseteq
    C$ are two-sided ideals and satisfy
    \begin{align*}
\sum_{\omega \in \dAb} \omega(A)=      B_{0} &= \sum_{\omega \in \dbA}
\omega(A), & \sum_{\omega \in \dAc} \omega(A) = C_{0}
      &=\sum_{\omega \in \dAc} \omega(A).
    \end{align*}
  \item $S_{B}$ and $S_{C}$ restrict to anti-isomorphisms $S_{B_{0}}
    \colon B_{0} \to C_{0}$
    and $S_{C_{0}} \colon C_{0} \to B_{0}$;
  \item 
    $\mathcal{A}_{B_{0}}:=(A,B_{0},\iota_{B_{0}},S_{B_{0}})$ and
    $\mathcal{A}_{C_{0}}:=(A,C_{0},\iota_{C_{0}},S_{C_{0}})$ are 
    compatible left and right quantum graphs;
  \item the natural maps ${_{B_{0}}A} \oo {_{S(B_{0})}A} \to \AlA$ and
    ${A_{S(C_{0})}} \oo A_{C_{0}} \to \ArA$ are isomorphisms so that
    $\Delta_{B}$ and $\Delta_{C}$ can be regarded as a left and a right
    comultiplication for  $\mathcal{A}_{B_{0}}$ and
    $\mathcal{A}_{C_{0}}$, respectively;
  \item
    $\mathcal{A}_{0}:=((\mathcal{A}_{B_{0}},\Delta_{B}),(\mathcal{A}_{C_{0}},\Delta_{C}))$
    is a flat regular multiplier Hopf algebroid.
  \end{enumerate}
\end{proposition}
\begin{proof}
  (1) By Lemma 4.4 in \cite{timmermann:regular}, $B_{0}\subseteq B$ is
  a two-sided ideal and equal to $\sum \omega(A)$, where the sum is
  taken over all $\omega \in \dbA$ or over all $\omega \in
  \Hom(\sbA,\Bb)$.  Since $\Hom(\sbA,\Bb) = \{ \omega \circ S^{-1} :
  \omega \in \dAb\}$, the equations involving $B_{0}$ follow. The
  assertions concerning $C_{0}$ follow similarly.

  (2) Immediate from  Lemma \ref{lemma:counits-antipode}.

  (3) By Lemma 4.4 in \cite{timmermann:regular},
  $A=B_{0}A=S_{B}(B_{0})A=C_{0}A$. Using the antipode, one finds that
  $A=AB_{0}=AC_{0}$. Since $A$ is non-degenerate as an algebra, we can
  conclude from these relations that it is also non-degenerate as a
  left module and as a right module over both $B_{0}$ and $C_{0}$,
  respectively. Since every idempotent $B_{0}$-module also is a
  $B$-module and $\bA$ is flat as a $B$-module, $_{B_{0}}A$ must be
    flat as a $B_{0}$-module. Likewise, the modules $_{C_{0}}A$,
    $A_{B_{0}}$ and $A_{C_{0}}$ are flat.

    (4) This follows easily from the relation $B_{0}A=A=S_{B}(B_{0})A$
    and the fact that $B_{0} \subseteq B$ is a two-sided ideal.

    (5) Immediate from (1)--(4).
\end{proof}

\section{Integration}
\label{section:integration}

This section introduces left and right integrals on multiplier
bialgebroids and establishes the key properties of such integrals.

As outlined in the introduction, integration on a multiplier
bialgebroid involves maps from $A$ to $B$ and $C$ that are left- or
right-invariant and correspond to fiberwise integration, and
functionals on $B$ and $C$ which then yield ``total'' integrals on $A$ by
composition.

We first formulate the appropriate left- and right-invariance for maps
from $A$ to $B$ and $C$ (Subsection
\ref{subsection:invariant-elements}). In the setting of Hopf
algebroids, such invariant maps were studied already in
\cite{boehm:integrals}.  We then fix functionals on $B$ and $C$ and
introduce several compatibility conditions on these functionals and
the invariant maps on $A$ that seem necessary to obtain a good
integration theory (Subsection \ref{subsection:base-weights}).  Using
these conditions and following the work of Van Daele on multiplier
Hopf algebras \cite{daele:1}, we establish uniqueness of integrals up
to scaling (Subsection \ref{subsection:uniqueness}) and existence of
modular elements (Subsection \ref{subsection:modular-element}). As a
tool, we use natural convolution operators which in Section
\ref{section:duality} reappear in the definition of the dual algebra
of the dual multiplier Hopf algebroid.

\subsection{Invariant elements of dual modules}
\label{subsection:invariant-elements}

Let
$\mathcal{A}=((\mathcal{A}_{B},\Delta_{B}),(\mathcal{A}_{C},\Delta_{C}))$
be a multiplier bialgebroid. We use the notation introduced in Section
\ref{section:multiplier-bialgebroids}. Regarding $A$ as a module over
$B$ and $C$, one obtains four associated dual modules, and we
abbreviate
\begin{align*}
  \dbA &:= \Hom(\bA,\bB), & \dAb &:= \Hom(\Ab,\Bb),  &
\dbAb &:= \dbA \cap \dAb, &
\\
  \dcA &:= \Hom(\cA,\Cc), & \dAc &:=\Hom(\Ac,\Cc), &\dcAc &:= \dcA \cap \dAc.
\end{align*}

 \begin{definition}\label{definition:invariant-elements}
   Let $\mathcal{A}$ be a multiplier bialgebroid. We call
  \begin{enumerate}
  \item $\cphi \in \dcA$ \emph{left-invariant (w.r.t.\
      $\Delta_{B}$)} if $ (\iota \otimes
    S_{B}^{-1}\circ\cphi)(\Delta_{B}(b)(a\otimes 1))=\cphi(b)a$,
  \item $\phic \in \dAc$ \emph{left-invariant (w.r.t.\
      $\Delta_{C}$)} if $(\iota \otimes \phic)((a\otimes
    1)\Delta_{C}(b)) = a\phic(b)$,
  \item $\bpsi \in \dbA$ \emph{right-invariant (w.r.t.\
      $\Delta_{B}$)} if $(\bpsi \otimes \iota)(\Delta_{B}(a)(1 \otimes
    b)) = \bpsi(a)b$,
  \item $\psib \in \dAb$ \emph{right-invariant (w.r.t.\
      $\Delta_{C}$)} if $(S_{C}^{-1}\circ\psib \otimes \iota)((1 \otimes
    b)\Delta_{C}(a)) = b\psib(a)$
  \end{enumerate}
  for all $a,b\in A$.  We denote by $\ltintb$, $\ltintc$, $\rtintb$
  and $\rtintc$ the spaces consisting of all maps as in (1)--(4),
  respectively.
\end{definition}
The purpose of the lower subscripts on $\phi$ and $\psi$ will become
clearer in the next subsection.

Regard $\dcAc$ as an $M(B)$-bimodule and $\dbAb$ as an
$M(C)$-bimodule, where $(x \cdot \cphic \cdot x')(a) = \cphic(x' a x)$
and $(y \cdot \bpsib \cdot y') (a) = \bpsib(y' a y)$.
\begin{proposition} \label{proposition:invariant-elements-bimodule}
  Let $\mathcal{A}$ be
  a multiplier bialgebroid.  Then $\ltintb$ and $\ltintc$ are
  $M(B)$-sub-bimodules of $\dcAc$, and $\rtintb$ and $\rtintc$ are
  $M(C)$-sub-bimodules of $\dbAb$.
\end{proposition} 
\begin{proof}
  In Sweedler notation, the conditions in (1)--(4) in Definition
  \ref{definition:invariant-elements} take the form
  \begin{linenomath*}
    \begin{align} \label{eq:invariant-elements-sweedler-left} \sum
      S_{B}^{-1}(\cphi(b_{(2)}))b_{(1)} a &= \cphi(b)a, & \sum ab_{(1)}
      S_{C}(\phic(b_{(2)})) &= a\phic(b),
      \\ \label{eq:invariant-elements-sweedler-right} \sum
      ba_{(2)}S_{C}^{-1}(\psib(a_{(1)})) &= b\psib(a), & \sum
      S_{B}(\bpsi(a_{(1)}))a_{(2)}b &= \bpsi(a)b.  \end{align}
  \end{linenomath*}
  The assertions then follow easily from the relations
  \eqref{eq:left-comult-module} and   \eqref{eq:right-comult-module}, for example,
  \begin{align*}
    \cphi(by)a &= \sum S_{B}^{-1}(\cphi(b_{(2)}))b_{(1)}ya =
    \cphi(b)ya, \\
    \cphi(x' b x)a &= \sum S_{B}^{-1}(\cphi(x' b_{(2)}x))b_{(1)}a
  \end{align*}
  for all $a,b\in A$, $y\in M(C)$, $x,x' \in M(B)$ and $\cphi \in \ltintb$.
\end{proof}
To stress that left-invariant maps $A\to C$ and right-invariant maps
$A \to B$ are morphisms of bimodules, we shall use double subscripts
from now on
and write $\cphic$ or $\bpsib$ instead of $\cphi,\phic,\bpsi$ or
$\psib$, respectively.

If $\mathcal{A}$ is a regular multiplier Hopf algebroid, then left or
right invariance with respect to $\Delta_{B}$ is equivalent to left or
right invariance with respect to $\Delta_{C}$ and to certain strong
invariance relations, just as in the case of Hopf algebroids; see
Scholium 2.10 in \cite{boehm:integrals}.  We include the proof, which
uses the following observation. In terms of the canonical maps of the
multiplier bialgebroid $\mathcal{A}$, the invariance conditions
(1)--(4) in Definition \ref{definition:invariant-elements} amount to
commutativity of the diagrams
\begin{align} \label{eq:invariance-canonical-left} \xymatrix@R=0pt{
    \ACA \ar[rd]_{\iota \otimes \cphic} \ar[rr]^{\Tl} && \AlA
    \ar[ld]^{\iota\otimes S_{B}^{-1}\circ \cphic} \\ & A, &} &&
  \xymatrix@R=0pt{ \AcA \ar[rd]_{\iota \otimes \cphic} \ar[rr]^{\lT}
    && \ArA \ar[ld]^{\iota\otimes \cphic}
    \\  & A, &} \\ \label{eq:invariance-canonical-right}
  \xymatrix@R=0pt{\AbA \ar[rd]_{\bpsib \otimes \iota} \ar[rr]^{\Tr} &&
    \AlA
    \ar[ld]^{\bpsib \otimes \iota}  \\
    & A, &} && \xymatrix@R=0pt{ \ABA \ar[rd]_{\bpsib \otimes \iota}
    \ar[rr]^{\Tr} && \ArA
    \ar[ld]^{S_{C}^{-1} \circ \bpsib \otimes \iota}  \\
    & A, &}
\end{align}
where $\cphic \in \dcAc$ and $\bpsib \in \dbAb$, respectively.
\begin{proposition} \label{proposition:invariant-elements-hopf} Let
  $\mathcal{A}$ be a regular multiplier Hopf algebroid with antipode
  $S$. Then
  \begin{enumerate}
  \item $\ltintb=\ltintc = \{ \cphic \in \dcAc \mid \text{ diagram
      \eqref{dg:strong-invariance-left} commutes}\}$;
    \begin{linenomath*}
      \begin{align} \label{dg:strong-invariance-left}
        \xymatrix@R=12pt@C=45 pt{  \AbA \ar[r]^{\Tr}
          \ar[d]_(0.55){\rT\Sigma} & \AlA \ar[d]^(0.55){S\circ (\iota
            \otimes S_{B}^{-1}\circ\cphic)}\\
      \ArA \ar[r]_(0.55){\iota\otimes \cphic} & A; }
      \end{align}
    \end{linenomath*}
  \item $\rtintb=\rtintc=\{\bpsib \in \dbAb \mid \text{ diagram
      \eqref{dg:strong-invariance-right} commutes}\}$;
    \begin{linenomath*}
      \begin{align} \label{dg:strong-invariance-right}
        \xymatrix@R=12pt@C=45 pt{ \AcA \ar[r]^{\lT} \ar[d]_{\Tl\Sigma}
          & \ArA
          \ar[d]^{S\circ (S_{C}^{-1}\circ\bpsib \otimes\iota)}\\
          \AlA \ar[r]_(0.55){\bpsib\otimes \iota} &
          A. }
      \end{align}
    \end{linenomath*}
  \item The maps $\cphic \mapsto S^{-1}\circ \cphic \circ S$
    and $\bpsib \mapsto S^{-1} \circ \bpsib \circ S$ are
    bijections between the sets in (1) and (2).
  \end{enumerate}
\end{proposition}
\begin{proof}
  We first prove (2), and assertion (1) follows similarly.  Let
  $\bpsib \in \dbAb$ and consider following diagram:
  \begin{linenomath*}
    \begin{align*}
        \xymatrix@R=10pt@C=45pt{\ABA \ar[dd]_{\rT} \ar[rrr]^{\id
            \oo S}
          \ar[rd]^(0.65){S_{C}^{-1}\circ\bpsib \otimes  \iota}
          &&& \AlA           \ar[ld]_(0.65){\bpsib \otimes  \iota}
          \\ &  A \ar[r]^{S} & A &  \\
          \ArA \ar[rrr]_{\id \oo S}
          \ar[ru]_(0.65){S_{C}^{-1} \circ \bpsib \otimes  \iota}
 && &\AbA \ar[uu]_{\Tr} \ar[lu]^(0.65){\bpsib \otimes \iota}}      
\end{align*}
\end{linenomath*}
The large rectangle commutes by \eqref{dg:galois-inverse}, and the upper and lower
cells commute by inspection. Since all horizontal and
vertical arrows are bijections, the triangle on the left hand side
commutes if and only if the triangle on the right hand side does.

Next, consider the following diagram:
 \begin{linenomath*}
   \begin{align*}
     \xymatrix@R=15pt@C=45 pt{ \ACA \ar[rr]^{\Tl} \ar[d]_{\lT\Sigma}
       && \AlA
       \ar[d]^{\bpsib\otimes \iota} \\
       \ArA \ar[r]^{\iota \otimes S} \ar[rd]_{S_{C}^{-1} \circ \bpsib
         \otimes\iota}       & \AbA \ar[r]^{\bpsib \oo \id} \ar[ur]^{\Tr} & A \\
 & A \ar[ru]_{S} } 
   \end{align*}
 \end{linenomath*}
 The upper left cell commutes by \eqref{dg:galois-aux}, and the lower
 cell commutes by inspection. Since $\Tr,\Tl,\lT$ and $S$ are
 isomorphisms, the outer cell commutes if and only if the triangle on
 the right commutes.

 We next prove the assertion on the second map in (3); the case of the
 first map is treated similarly.  Let $\bpsib \in \dbAb$ again and consider
 following diagram:
  \begin{linenomath*}
    \begin{align*}
        \xymatrix@R=15pt@C=55pt{ \AlA \ar[rd]^(0.65){\bpsib \otimes  \iota}
 &&& \ArA \ar[lll]_{\Sigma (S \otimes S)}    \ar[ld]_(0.65){\id \otimes
  S_{C}^{-1}\circ \bpsib\circ S\quad}
          \\ &  A & A \ar[l]_{S} & \\
\AbA \ar[uu]_{\Tr} \ar[ru]_(0.65){\bpsib \otimes \iota} &&&
\AcA \ar[uu]_{\lT} \ar[lll]^{\Sigma(S\otimes S)} \ar[lu]^(0.65){\id \otimes
  S_{C}^{-1} \circ \bpsib \circ S\quad}}      
\end{align*}
\end{linenomath*}
The large rectangle commutes by \eqref{dg:galois-aux2}, and the upper
and lower cell by inspection. Therefore, the triangle on the left hand
side commutes if and only if the triangle on the right hand side does.
\end{proof}
\begin{remark}
  In Sweedler notation, commutativity of the diagrams
  \eqref{dg:strong-invariance-left} and
  \eqref{dg:strong-invariance-right} amounts to the relations
  \begin{linenomath*}
    \begin{align*}
      \sum S(a_{(1)})\cphic(a_{(2)}b) &=
      b_{(1)}S_{C}(\cphic(ab_{(2)})), \\
      \sum \bpsib(ba_{(1)})S(a_{(2)}) &=
      S_{B}(\bpsib(b_{(1)}a))b_{(2)}.
    \end{align*}
  \end{linenomath*}
\end{remark}

\subsection{Base weights and adapted integrals}

\label{subsection:base-weights}

Let $\mathcal{A}$ be regular multiplier Hopf algebroid. To obtain from
left- or right-invariant maps $\cphic\colon A\to C$ or $\bpsib \colon
A\to B$ scalar-valued integrals $\phi$ or $\psi$ on $A$, we compose
the former with suitable functionals $\mu_{B}$ and $\mu_{C}$ on the
base algebras $B$ and $C$, respectively. We first formulate our
assumptions on the functionals $\mu_{B}$ and $\mu_{C}$, and then
define the notion of a left and of a right integral.

Throughout this section, let 
  $\mathcal{A}=((A_{B},\Delta_{B}), (\mathcal{A}_{C},\Delta_{C}))$
  be a  regular multiplier Hopf algebroid.
  \begin{definition} \label{definition:base-weight} A \emph{base
      weight} for a regular multiplier Hopf algebroid $\mathcal{A}$ is
    a pair of faithful functionals $\mu_{B} \in \dB$ and $\mu_{C}\in
    \dC$ satisfying $\mu_{B} \circ S_{C} = \mu_{C}$, $\mu_{C} \circ
    S_{B} = \mu_{B}$ and $\mu_{B} \circ \beps = \mu_{C} \circ \epsc$,
    where $\beps$ and $\epsc$ denote the left and the right counit of
    $\mathcal{A}$, respectively.
\end{definition}
% \begin{remark} \label{remark:base-opposite} Let
%   $\mu=(\mu_{B},\mu_{C})$ be a base weight for a regular multiplier
%   Hopf algebroid $\mathcal{A}$ and regard $\mu_{B}$ and $\mu_{C}$ also
%   as functionals on $B^{\op}$ and $C^{\op}$, respectively. Then
%   $\mu=(\mu_{B},\mu_{C})$ also is a base weight for the opposite
%   $\mathcal{A}^{\op}$ defined in Subsection
%   \ref{subsection:hopf-algebroids}, and
%   $\mu^{\co}:=(\mu_{C},\mu_{B})$ is a base weight for the co-opposite
%   $\mathcal{A}^{\co}$ and the bi-opposite $\mathcal{A}^{\co,\op}$, as
%   one can easily check.
% \end{remark}
\begin{remark} \label{remark:base-weight-positive}
  If $\mathcal{A}$ is a multiplier Hopf $*$-algebroid, we shall
  usually assume that $\mu_{B}$ and $\mu_{C}$ are positive  in the
  sense that they coincide with $\mu_{B}^{*}:=\ast \circ \mu_{B} \circ
  \ast$ and $\mu_{C}^{*}:=\ast \circ \mu_{C} \circ \ast$,
  respectively, and satisfy $\mu_{B}(x^{*}x) \geq
  0$ and $\mu_{C}(y^{*}y) \geq 0$ for all $x\in B$, $y\in C$.
\end{remark}
 The functionals comprising a base weight automatically have
modular automorphisms, given by the square of the antipode or its
inverse. 
\begin{proposition} \label{proposition:counit-kms} Let $\mathcal{A}$
  be a regular multiplier Hopf algebroid with base weight $\mu=(\mu_{B},\mu_{C})$. Then
  $\mu_{B}$ and $\mu_{C}$ have modular automorphisms, given by
 $\sigma^{\mu}_{B}:=S^{-1}_{B}\circ S^{-1}_{C}$ and
 $\sigma^{\mu}_{C}:=S_{B} \circ S_{C}$, respectively.
\end{proposition}
\begin{proof}
  We only prove the assertion concerning $\sigma^{\mu}_{B}$. For all
  $x \in B$, $a\in A$,
  \begin{align*}
    \mu_{B}(x\beps(a)) = \mu_{B}(\beps(xa)) &= \mu_{C}(\epsc(xa)) \\ &=
    \mu_{C}(\epsc(S_{C}^{-1}(x)a)) \\ &= \mu_{B}(\beps(S_{C}^{-1}(x)a)) =
    \mu_{B}(\beps(a)S_{B}^{-1}(S_{C}^{-1}(x))).  \qedhere
  \end{align*}
\end{proof}
  \begin{remark}
  In the setting of weak multiplier Hopf algebras, the base algebras
  $B$ and $C$ carry canonical functionals $\mu_{B}$ and $\mu_{C}$
  which satisfy the assumptions in Definition
  \ref{definition:base-weight}, see Proposition 4.8 in
  \cite{daele:separability}.
\end{remark}
\begin{remark}
  The preceding result fits with the theory of measured quantum
  groupoids, where the square of the antipode generates the scaling
  group and the latter restricts to the modular automorphism groups on
  the base algebras, that is, in the notation of \cite{enock:action},
  $S^{2} = \tau_{i}$ and $\tau_{t} \circ \alpha = \alpha \circ
  \sigma^{\nu}_{t}$, $\tau_{t} \circ \beta = \beta \circ
  \sigma^{\nu}$.
\end{remark}

Let $\mathcal{A}$ be a regular multiplier Hopf algebroid with base
weight $\mu=(\mu_{B},\mu_{C})$. Assume that a functional $\omega \in
\dA$ and maps $\bomega,\omegab \colon A \to B$ and
$\comega,\omegac\colon A \to C$ are given. Then
\begin{align} \label{eq:bomega}
  \omega &= \mu_{B} \circ \bomega, & \bomega &\in \dbA &
\ &\Leftrightarrow \ &  \omega(xa) &= \mu_{B}(x\bomega(a))
 && \text{for all
  } x\in B, a\in A, \\  \label{eq:omegab}
  \omega &= \mu_{B} \circ \omegab, & \omegab &\in \dAb & \
  &\Leftrightarrow \ & \omega(ax) &= \mu_{B}(\omegab(a)x) && \text{for
    all
  } x\in B, a\in A, \\  \label{eq:comega}
  \omega &= \mu_{C} \circ \comega, & \comega &\in \dcA &\
  &\Leftrightarrow \ & \omega(ya) &= \mu_{C}(y\comega(a)) &&\text{for
    all
  } y\in C, a\in A, \\  \label{eq:omegac}
  \omega &= \mu_{C} \circ \omegac, & \omegac &\in \dAc& \
  &\Leftrightarrow \ & \omega(ay) &= \mu_{C}(\omegac(a)y)  && \text{for
    all } y\in C, a\in A.
\end{align}
The verification is straightforward and uses the relations
$A=BA=AB=CA=AC$ and the fact that $\mu_{B}$ and $\mu_{C}$ are
faithful.  Since $\mu_{B}$ and $\mu_{C}$ are faithful, maps
$\bomega,\omegab,\comega$ or $\omegac$ satisfying the equivalent
conditions in \eqref{eq:bomega},
\eqref{eq:omegab}, \eqref{eq:comega} or \eqref{eq:omegac},
respectively, are uniquely determined by $\omega$ if they exist.
\begin{definition}
  Let $\mathcal{A}$ be a regular multiplier Hopf algebroid 
  with base weight $\mu=(\mu_{B},\mu_{C})$. A functional $\omega \in
  \dA$ is \emph{adapted to $\mu$} if there exist maps
  $\bomega,\omegab\colon A \to B$ and $\comega,\omegac\colon A \to C$
  such that the equivalent conditions in
  \eqref{eq:bomega}--\eqref{eq:omegac} hold. We denote by $\dmAm
  \subseteq \dA$ the subspace of all functionals adapted to $\mu$.
\end{definition}
\begin{remark}
  In the setting of weak multiplier Hopf algebras, every functional on
  the total algebra is adapted to the canonical functionals $\mu_{B}$
  and $\mu_{C}$ on the base algebras in the sense above.  This follows
  from Proposition 3.3 in \cite{boehm:comodules} and the comments
  thereafter.
\end{remark}

The definition above can be rephrased as follows.  Composing maps from
$A$ to $B$ or $C$ with
$\mu_{B}$ or $\mu_{C}$, respectively, we obtain push-forward maps
\begin{align} \label{eq:base-pushforwards} (\mu_{B})_{*} &\colon
  \dbA,\dAb \to \dA & &\text{and} & (\mu_{C})_{*} &\colon \dcA, \dAc
  \to \dA,
\end{align}
and $\dmAm$ is the intersection of the images of these four maps.

By assumption, the left and right counits of $\mathcal{A}$ yield a
counit functional $\eps \in \dmAm$: 

\begin{example} \label{example:base-counit} Let $\mathcal{A}$ be a
  regular multiplier Hopf algebroid with base weight
  $\mu=(\mu_{B},\mu_{C})$ and left and right counits $\beps$ and
  $\epsc$, respectively.  Then the functional $\eps:=\mu_{B} \circ
  \beps = \mu_{C} \circ \epsc$ lies in $\dmAm$ and the associated maps
  $\epsb$ and $\ceps$ are given by $\epsb = S_{C} \circ \epsc$ and
  $\ceps = S_{B} \circ \beps$, respectively, because
    \begin{align*}
      \eps(ax) &=\mu_{C}(\epsc(ax))=\mu_{C}(S_{C}^{-1}(x)\epsc(a)) =
      \mu_{B}(S_{C}(\epsc(a))x), \\
      \eps(ya) &= \mu_{B}(\beps(ya)) = \mu_{B}(\beps(a)S_{B}^{-1}(y))
      = \mu_{C}(y S_{B}(\beps(a)))
    \end{align*}  
    for all $x\in B$, $y\in C$, $a\in A$.   Moreover, $\eps \circ S=
    \eps$ by Lemma \ref{lemma:counits-antipode}.
\end{example}

We can now define left and right integrals adapted to a base weight.
Let $\mathcal{A}$ be a regular multiplier Hopf algebroid with base
weight $\mu$ as before, and let $\phi,\psi \in \dmAm$. By
Proposition \ref{proposition:invariant-elements-hopf}, the following
implications hold:
\begin{align}
  \label{eq:base-phi-invariance} 
  \cphi  &\text{ is left-invariant} & \
  &\Leftrightarrow \ & \phic &\text{ is left-invariant}
 & \ &\Rightarrow \ & \phic=\cphi \in \dcAc \\
  \bpsi &\text{ is right-invariant} & \
  &\Leftrightarrow \ & \psib &\text{ is right-invariant} & \
  &\Rightarrow & \ \bpsi = \psib \in \dbAb.
  \label{eq:base-psi-invariance} 
\end{align}
For example, if $\cphi$ is left-invariant, then $\cphi \in \dAc$ and
hence $\cphi=\phic$.

\begin{definition}
  A \emph{left (right) integral} on a regular multiplier Hopf
  algebroid $\mathcal{A}$ with base weight $\mu$ is a $\phi \in \dmAm$
  (or $\psi \in \dmAm$) satisfying the equivalent conditions in
  \eqref{eq:base-phi-invariance} (or \eqref{eq:base-psi-invariance},
  respectively).
\end{definition}

Like invariant elements, left and right integrals form bimodules over
$B$ or $C$, respectively, and stand in a bijective correspondence via
the antipode. The proof uses the following result.
\begin{lemma} \label{lemma:base}
 Let $\mathcal{A}$ be a regular multiplier Hopf algebroid with base
 weight $\mu$ and let $\omega \in \dmAm$.
 \begin{enumerate}
 \item Let $x,x' \in M(B)$, $y,y' \in M(C)$ and $\omega':= xy \cdot
   \omega \cdot x'y'$. Then $\omega' \in \dmAm$ and
   \begin{align*}
     \bomega'(a) &=\bomega(y'axy)\sigma^{\mu}_{B}(x'), & \omegab'(a)
     &=
     (\sigma^{\mu}_{B})^{-1}(x)\omegab(x'y'ay), \\
     \comega'(a) &= \comega(x'axy)\sigma^{\mu}_{C}(y'), & \omegac'(a)
     &= (\sigma^{\mu}_{C})^{-1}(y) \omegac(x'y'ax) 
   \end{align*}
for all $a\in A$.
 \item The compositions $\omega  \circ S$ and $\omega\circ S^{-1}$ lie in
   $\dmAm$, 
  \begin{align*}
    _{B}(\omega \circ S) &= S_{B}^{-1} \circ \omegac \circ
    S, & (\omega \circ S)_{B}&= S_{B}^{-1} \circ \comega \circ S, \\
    _{C}(\omega \circ S) &= S_{C}^{-1} \circ \omegab \circ
    S, & (\omega \circ S)_{C} &= S_{C}^{-1} \circ \bomega \circ S,
  \end{align*}
  and the relations still hold if $S,S^{-1}_{B},S^{-1}_{C}$ are
  replaced by $S^{-1},S_{C},S_{B}$, respectively.
\item Assume that $\mathcal{A}$ is a multiplier Hopf $*$-algebroid and
  that $\mu_{B}$ and $\mu_{C}$ are self-adjoint. Then the composition
  $\omega^{*}:= \ast \circ \omega \circ \ast$ lies in $\dmAm$ and
  \begin{align*}
    _{B}(\omega^{*}) &= \ast \circ \omegab \circ \ast, & (\omega^{*})_{B} &=
    \ast \circ \bomega \circ \ast, \\  _{C}(\omega^{*}) &= \ast \circ
    \omegac \circ \ast, &(\omega^{*})_{C} &=
    \ast \circ \comega \circ \ast.
  \end{align*}
 \end{enumerate}
\end{lemma}
\begin{proof}
  The verification consists of straightforward calculations, for
  example,
  \begin{align*}
    (xy \cdot \omega \cdot x'y')(x''a) &= \mu_{B}(\bomega(x'y'x''axy))
    = \mu_{B}(x''\bomega(y'axy)\sigma^{\mu}_{B}(x')), \\
    (\omega \circ S)(x''a)&=\omega(S(a)S_{B}(x'')) =
    \mu_{C}(\omegac(S(a))S_{B}(x'')) =
    \mu_{B}(x''S_{B}^{-1}(\omegac(S(a))))
  \end{align*}
for all $a,x,y,x',y'$ as in (1) and $x''\in B$.
\end{proof}
\begin{proposition} \label{proposition:integrals-bimodule} All left
  (right) integrals on a regular multiplier Hopf algebroid
  $\mathcal{A}$ with base weight $\mu$ form an $M(B)$-sub-bimodule
  ($M(C)$-sub-bimodule) of $\dmAm$.
\end{proposition}
\begin{proof}
  Immediate from Lemma \ref{lemma:base} and Proposition \ref{proposition:invariant-elements-bimodule}.
\end{proof}
\begin{proposition} \label{proposition:integrals-antipode} Let
  $\mathcal{A}$ be a regular multiplier Hopf algebroid with base
  weight $\mu$. Then the maps $\phi \mapsto \phi \circ S$ and $\psi
  \mapsto \psi\circ S$ form bijections between all left and all right
  integrals on $\mathcal{A}$ adapted  to $\mu$.
\end{proposition}
\begin{proof}
Immediate from Lemma \ref{lemma:base} and Proposition
  \ref{proposition:invariant-elements-hopf}.
\end{proof}

\subsection{Uniqueness of integrals}

\label{subsection:uniqueness}

Left and right integrals on quantum groups are, like the Haar measure
of a locally compact group, unique up to scaling. Proposition
\ref{proposition:integrals-bimodule} shows that in the present
context, we can only expect uniqueness up to scaling by multipliers of
the base algebras $B$ or $C$, respectively, and only under a suitable
non-degeneracy assumption on the integrals. 

The proof involves the following relative tensor products of
functionals adapted to base weights. Let $\mathcal{A}$ be a regular
multiplier Hopf algebroid with base weight $\mu=(\mu_{B},\mu_{C})$ and
let $\upsilon,\omega \in \dmAm$. Then
  \begin{align*}
    \upsilon\circ(\id \otimes \bomega) &= \mu_{B} \circ (\upsilonb
    \otimes \bomega) =
    \omega \circ (\upsilonb \otimes \id) &&    \text{in } \dual{(\AbA)}, \\
\upsilon \circ (\id \otimes \omegab) &= \mu_{B} \circ
    (\bupsilon \otimes \omegab) = \omega \circ (\bupsilon
    \otimes \id)  &&  \text{in } \dual{(\ABA)}, \\
  \upsilon \circ (\id \otimes S_{B}^{-1}\circ\comega) &=
    \mu_{B} \circ (\bupsilon \otimes S_{B}^{-1} \circ \comega) =
 \omega \circ
    (\bupsilon \otimes \id) & &  \text{in } \dual{(\AlA)}
  \end{align*}
We denote these compositions  by
\begin{align} \label{eq:base-oo-b}
  &\upsilon \oob \omega \in \dual{(\AbA)}, & & \upsilon \ooB \omega
  \in \dual{(\ABA)}, && \upsilon \ool \omega \in
  \dual{(\AlA)},
\end{align}
respectively, and similarly define
\begin{align} \label{eq:base-oo-c}
  &\upsilon \ooc \omega \in \dual{(\AcA)}, & & \upsilon \ooC \omega
  \in \dual{(\ACA)}, &&  \upsilon \oor \omega \in
  \dual{(\ArA)}.
\end{align}
The construction of these tensor products can be regarded as  a
special case of a general bi-categorical construction that is
summarised in Appendix \ref{section:factorisation}.

\begin{example} \label{example:base-counit-2} The left and the right
  counit of $\mathcal{A}$ satisfy $\beps(ab)=\beps(a\beps(b))$ and
  $\epsc(ab)=\epsc(\epsc(a)b)$  by
  \eqref{dg:left-counit} and \eqref{eq:right-counit}. Hence, the
  functional $\eps:=\mu_{B} \circ \beps = \mu_{C} \circ \epsc$
  satisfies
  \begin{align} \label{eq:base-counit-mult}
    \eps \circ m_{B} &= \eps \oob \eps, & \eps \circ m_{C} &= \eps
    \ooc \eps,
  \end{align}
  where $m_{B} \colon \AbA \to A$ and $m_{C} \colon \AcA \to A$ denote
  the multiplication maps.
\end{example}

If $\phi$ is a left integral on $\mathcal{A}$ adapted to $\mu$, then
commutativity of the diagrams \eqref{eq:invariance-canonical-left} implies that
    \begin{align} \label{eq:integrals-oo-left} (\upsilon \ool \phi)
      \circ \Tl &= \upsilon \ooC \phi, & (\upsilon' \oor
      \phi) \circ \lT &= \upsilon' \ooc \phi  
  \end{align}
 for all $\upsilon \in
    (\mu_{C})_{*}(\dcA)$, $\upsilon' \in (\mu_{C})_{*}(\dAc)$. Likewise, if $\psi$ is a right integral, then
    \begin{align} \label{eq:integrals-oo-right} (\psi \ool \omega)
      \circ \Tr &= \psi \oob \omega, & (\psi \oor \omega')
      \circ \rT &= \psi \ooB \omega'
\end{align}
 for all
 $\omega \in (\mu_{B})_{*}(\dbA)$, $\omega' \in
    (\mu_{B})_{*}(\dAb)$ by commutativity of \eqref{eq:invariance-canonical-right}.

The first step towards the proof of uniqueness is the following
result.
\begin{proposition} \label{proposition:uniqueness-phi-psi} Let $\phi$
  be a left and $\psi$ a right integral on a regular
  multiplier Hopf algebroid $\mathcal{A}$ with base weight $\mu$. Then
  $(A \bphi(A)) \cdot \psi = (A \cpsi(A)) \cdot \phi$ and $\psi \cdot
  (\phib(A)A) = \phi \cdot (\psic(A)A)$ as subsets of $\dA$.
\end{proposition}
\begin{proof}
  We only prove the first equation.  Let $a\in A$ and $b\otimes c\in
  \AlA$ and write
  \begin{align*}
    b\otimes c=\sum_{i} \Delta_{B}(d_{i})(1 \otimes e_{i}) = \sum_{j}
    \Delta_{B}(f_{j}) (g_{j} \otimes 1)
  \end{align*}
  with $d_{i},e_{i},f_{j},g_{j} \in A$.  Then  $(\psi \ool \phi)(\Delta_{B}(a)(b\otimes c))$ is equal to
  \begin{align*}
    \sum_{i} (\psi \ool \phi)(\Delta_{B}(ad_{i})(1 \otimes e_{i})))
    &= \sum_{i} (\psi \oob \phi)(ad_{i} \otimes e_{i}) =
\sum_{i}    \psi(ad_{i}\bphi(e_{i}))
  \end{align*}
  by \eqref{eq:integrals-oo-left}, and to
  \begin{align*}
    \sum_{j} (\psi \ool \phi)(\Delta_{B}(af_{j})(g_{j}\otimes 1))) &=
    \sum_{j} (\psi \ooC \phi)(g_{j} \oo af_{j}) =
\sum_{j}    \phi(af_{j}\cpsi(g_{j}))
  \end{align*}
by \eqref{eq:integrals-oo-right}. Since the maps $\Tl$ and $\Tr$ are
bijective, we can conclude $(A \bphi(A)) \cdot \psi = (A
  \cpsi(A)) \cdot \phi$.
\end{proof}

\begin{corollary} \label{corollary:uniqueness-phi-phi}
Let $\phi$ and $\phi'$ be left integrals on a regular multiplier Hopf
algebroid $\mathcal{A}$ with base
  weight $\mu$.   If $A\bphi'(A) \subseteq A\bphi(A)$ or $\phib'(A)A
  \subseteq \phib(A)A$, then
  \begin{align*}
    \sum_{\psi} A\cpsi(A) \cdot \phi' & \subseteq 
    \sum_{\psi} A\cpsi(A) \cdot \phi &&  \text{or} &
    \sum_{\psi} \phi' \cdot \psic(A)A & \subseteq 
    \sum_{\psi} \phi \cdot \psic(A)A,
  \end{align*}
  respectively, where the sums are taken over all right integrals
  $\psi$ on $\mathcal{A}$ adapted to $\mu$.
\end{corollary}
\begin{proof}
We only consider the first inclusion.
If  $A\bphi'(A) \subseteq A\bphi(A)$, then  
\begin{align*}
  A\cpsi(A) \cdot \phi' &= (A\bphi'(A)) \cdot \psi \subseteq
  (A\bphi(A)) \cdot \psi = A\cpsi(A) \cdot \phi
\end{align*}
for every right integral $\psi$ by Proposition \ref{proposition:uniqueness-phi-psi}. 
\end{proof}
The preceding result suggests the following non-degeneracy
condition for integrals.
\begin{definition} \label{definition:full}
  Let $\mathcal{A}$ be a regular
  multiplier Hopf algebroid $\mathcal{A}$ with base weight $\mu$.
We call  a left integral $\phi$ (right integral $\psi$) on
$\mathcal{A}$ adapted to $\mu$ 
  \emph{full} if $\bphi$ and $\phib$ ($\cpsi$ and $\psic$,
  respectively) are surjective.  We say $(\mathcal{A},\mu)$ \emph{has
    full and faithful integrals} if there exists a left or right
  integral on $\mathcal{A}$ that is full and faithful.
\end{definition}
\begin{remark} \label{remark:full}
  \begin{enumerate}
  \item If $\omega$ is a full left or right integral, then the right or left
    integrals $\omega \circ S$ and $\omega \circ S^{-1}$ are full
    again. This follows immediately from Lemma \ref{lemma:base} (2).
  \item Proposition \ref{proposition:uniqueness-phi-psi} suggests to
    relax the condition and define a left integral $\phi$ to be full
    if $A\bphi(A)=A=\phib(A)A$. But if the latter condition holds,
    then $\phib(A)=\phib(A\beps(A)) =\phib(A)\beps(A) =
    \beps(\phib(A)A)= \beps(A)$ and likewise $\bphi(A)=\epsb(A) =
    S(\ceps(A))$, and by Proposition \ref{proposition:counits-full},
    we can assume $\beps(A)=B=S(\ceps(A))$ without much loss of
    generality.  
  \end{enumerate}
\end{remark}

The results obtained so far can be summarised as follows.
\begin{theorem} \label{theorem:uniqueness-full}
  Let $\omega$ and $\omega'$ be   left or right integrals 
on a regular multiplier Hopf algebroid $\mathcal{A}$ with base weight
  $\mu$. If $\omega$ is full, then $A\cdot \omega' \subseteq A\cdot \omega$
  and $\omega' \cdot A \subseteq \omega \cdot A$.
\end{theorem}
\begin{proof}
  If $\omega$ is a left and $\omega'$ is a right integral, then the
  assertion follows from Proposition
  \ref{proposition:uniqueness-phi-psi}. If both are left integrals and
  $\omega$ is full, then $\omega \circ S$ is a full right integral by
  Proposition \ref{proposition:integrals-antipode} and Remark
  \ref{remark:full}, and then the assertion follows from Corollary
  \ref{corollary:uniqueness-phi-phi}. The remaining cases are similar.
\end{proof}
\begin{corollary} \label{corollary:full}
For every full left integral  $\phi$ and every full right integral
$\psi$ on a regular multiplier Hopf algebroid $\mathcal{A}$ with base
weight $\mu$, the maps $\cphic$ and $\bpsib$ are surjective.
\end{corollary}
\begin{proof}
   Let $\mathcal{A},\mu$ and $\phi$ be as above. Then $\psi':=\phi
   \circ S$ is a full right integral and  $A \phi = A \psi'$ by Theorem
   \ref{theorem:uniqueness-full}. This relation and idempotence of $A$ imply
   $\cphic(A)=\psi'_{C}(A) = C$. A similar argument shows that
   $\bpsib(A)=B$ for every right integral $\psi$.
\end{proof}

We now can show that
 integrals are unique up to scaling by elements of
$M(B)$ or $M(C)$, respectively:
\begin{theorem} \label{theorem:integrals-uniqueness} For every full
  and faithful left integral $\phi$ and every full and faithful right
  integral $\psi$ on a flat regular multiplier Hopf algebroid
  $\mathcal{A}$ with base weight $\mu$, 
    \begin{gather*}
      M(B) \cdot \phi = \{ \text{left integrals } \phi' \text{ for }
      (\mathcal{A},\mu) \} = \phi \cdot M(B), \\
      M(C) \cdot \psi = \{ \text{right integrals } \psi' \text{ for }
      (\mathcal{A},\mu) \} = \psi \cdot M(C).
    \end{gather*}
\end{theorem}
\begin{proof}
  We only prove the assertion concerning left integrals.

  Let $\mathcal{A}$, $\mu$ and $\phi$ be as above.  Then every element of
  $M(B) \cdot \phi$ and of $\phi \cdot M(B)$ is a left integral by
  Proposition \ref{proposition:integrals-bimodule}.

  Conversely, assume that $\phi'$ is a left integral on
  $\mathcal{A}$. Since $\phi$ is faithful by assumption and $A
  \cdot \phi' \subseteq A\cdot \phi$ and $\phi' \cdot A \subseteq \phi
  \cdot A$ by Theorem \ref{theorem:uniqueness-full}, there exist
  unique linear maps $\alpha,\beta \colon A \to A$ such that $a \cdot
  \phi' = \alpha(a) \cdot \phi$ and $\phi' \cdot a = \phi \cdot
  \beta(a)$ for all $a \in A$. Clearly, $\alpha(ab)=a\alpha(b)$ and
  $\beta(ab)=\beta(a)b$ for all $a,b\in A$, so that we can regard
  $\alpha$ as a right and $\beta$ as a left multiplier of $A$,
  respectively, and write $\alpha \cdot \phi = \phi' = \phi \cdot
  \beta$. Furthermore,
  \begin{align*}
    \alpha \cdot \cphic  = {_{C}(\alpha \cdot \phi)} = \cphic' =   (\phi
    \cdot \beta)_{C} = \cphic \cdot \beta,
  \end{align*}
  in particular, $\alpha$ and $\beta$ commute with $C$.

  Choose a full left integral $\psi$, for example, $\phi \circ S$.  We
will show that for all $a,b\in A$,
  \begin{align} \label{eq:integrals-uniqueness-2}
a \bpsib(b)\alpha &= a\bpsib(b\alpha), &    \beta \bpsib(b) a &= \bpsib(\beta b)a.
  \end{align}
  These equations imply $B\alpha \subseteq B$ and $\beta B \subseteq
  B$. We then conclude
  \begin{align*}
    \mu_{B}(x\alpha\sigma^{\mu}_{B}(\phib(a))) =
    \mu_{B}(\phib(a)x\alpha) &= \phi(ax\alpha) \\ &= \phi(\beta ax) =
    \mu_{B}(\phib(\beta a)x) = \mu_{B}(x\sigma^{\mu}_{B}(\phi_{B}(\beta a)))
  \end{align*}
 for all $a\in A$, $x\in B$. Since $\mu_{B}$ is faithful and $\phi$ is
 full, this relation implies $\alpha B \subseteq B$, that is, $\alpha
 \subseteq M(B)$. A similar argument shows that also $\beta \in M(B)$.

 Therefore, we only need to prove
 \eqref{eq:integrals-uniqueness-2}. We focus on the second equation
 because it is nicer to work with  left multipliers; the first equation follows
 similarly.  Let $a,b\in A$.  Since $\cphic' = \cphic \cdot \beta$ and
 $\cphic$ are left-invariant,
  \begin{align*}
    (\id \otimes S_{B}^{-1} \circ \cphic \circ \beta)(\Delta_{B}(b)(a \otimes 1)) & =
    \cphic(\beta b)a = (\id \otimes
S_{B}^{-1} \circ    \cphic)(\Delta_{C}(\beta b)(a \otimes 1)).
  \end{align*}
  Since $\Tl$ is surjective, we can conclude
  \begin{align*}
    (\id \otimes S_{B}^{-1} \circ \cphic \circ \beta)(\Delta_{B}(b)(a\otimes cd)) & =
    (\id \otimes S_{B}^{-1} \circ \cphic)(\Delta_{B}(\beta b)(a \otimes cd))
  \end{align*}
  for all $a,b,c,d \in A$. Since $\phi$ is faithful and   $\bA$ is flat, maps of the form
  $d \cdot \cphic$ separate the points of $A$ and slice maps of the form $\id \otimes
S_{B}^{-1} \circ  (d\cdot \cphic)$ separate the points of
  $\AlA$. Consequently,
  \begin{align*}
    (\iota \otimes \beta)(\Delta_{B}(b)(1 \otimes c)) =
    \Delta_{B}(\beta b) (1\otimes c) \quad \text{ for all }a,b,d\in A.
  \end{align*}
We  apply $\bpsib \otimes \id$, use right-invariance of $\bpsib$ and
the fact that $\beta$ commutes with $C$, and find $\beta\bpsib(b)c =
\bpsib(\beta b)c$ for all $b,c\in A$.
\end{proof}
Note that the maps $M(A) \to \dA$ given by $a \mapsto a\cdot \omega$
and $a\mapsto \omega \cdot a$ are injective for every faithful $\omega
\in \dA$. Given full and faithful left and right integrals, we thus
obtain bijections between $M(B)$ and the space of left integrals, and
between $M(C)$ and the space of right integrals.  In particular, we
obtain bijections between invertible multipliers and full and faithful
left or right integrals.

\subsection{Convolution operators and the modular automorphism}

\label{subsection:modular-automorphism}

We next show that $A\omega = \omega A$ for every full left and every
full right integral $\omega$, and deduce that every full and faithful
integral has a modular automorphism.  As a tool, we use the following
left and right convolution operators associated to maps $\bupsilon \in
\dbA$, $\comega \in \dcA$, $\upsilonb' \in \dAb$ and $\omegac' \in
\dAc$, respectively:
\begin{align*}
  \begin{aligned}
    \lambda(\bupsilon) &\colon A \to L(A), &
    \lambda(\bupsilon)(a)b &:= (\bupsilon \oo \id)(\Delta_{B}(a)(1
    \oo
    b)),  \\
    \rho(\comega) &\colon A \to L(A),&
    \rho(\comega)(a)b &:= (\id \oo S_{B}^{-1}\circ
    \comega)(\Delta_{B}(a)(b \oo 1)), \\
    \lambda(\upsilonb') &\colon A \to R(A), &
    b\lambda(\upsilonb')(a) &:= (S_{C}^{-1}\circ \upsilonb' \oo
    \id)((1
    \oo b)\Delta_{C}(a)), \\
    \rho(\omegac') &\colon A \to R(A), &
    b\rho(\omegac')(a) &:= (\id \oo \omegac')((b \oo
    1)\Delta_{C}(a)).
  \end{aligned}
 \end{align*}

 This notation is somewhat ambiguous for elements $\bupsilon_{B}\in
 \dbAb$ or $\comega_{C} \in \dcAc$, and we shall always write
 $\rho(\bupsilon)$, $\rho(\upsilonb)$, $\lambda(\comega)$ or
 $\lambda(\omegac)$ to indicate which convolution operator is meant.
This ambiguity will be resolved in Lemma \ref{lemma:convolution} (4)
below.

 Let us collect a few easy observations.  For all
 $\bupsilon,\upsilonb',\comega,\omegac'$ as above and $a,c\in A$,
\begin{align*}
  \lambda(c \cdot \bupsilon)(a) &= (\bupsilon \oo
  \id)(\Delta_{B}(a)(c \oo 1)) \in A, \\
  \rho(c \cdot \comega)(a) &= (\id \oo S_{B}^{-1} \circ
  \comega)(\Delta_{B}(a)(1 \oo c)) \in A
\end{align*}
and likewise $\lambda(\upsilonb' \cdot c)(a), \rho(\omegac'
\cdot c)(a) \in A$. Next,
\begin{align} \label{eq:convolution-right-invariance}
  \lambda(\bupsilon) &= \bupsilon \Leftrightarrow \bupsilon \text{ is
    right-invariant}, & \rho(\comega) &= \comega \Leftrightarrow
  \comega \text{ is left-invariant},
  \\ \label{eq:convolution-left-invariance} \lambda(\upsilonb') &=
  \upsilonb' \Leftrightarrow \upsilonb' \text{ is right-invariant}, &
  \rho(\omegac') &= \omegac' \Leftrightarrow \omegac' \text{ is
    left-invariant},
\end{align}
where we regard $B$ and $C$ as elements of $L(A)$ or $R(A)$, and
finally,
\begin{align} \label{eq:convolution-counit}
  \lambda(\beps) &= \rho(\ceps) =  \lambda(\epsb) =
\rho(\epsc) = \id_{A};
\end{align}
see diagrams \eqref{dg:left-counit}, \eqref{eq:right-counit} and Example
\ref{example:base-counit}.
\begin{lemma} \label{lemma:convolution} Let $\bupsilon \in A\cdot
  \dbA$, $\upsilonb' \in \dAb \cdot A$, $\comega \in A \cdot \dcA$,
  $\omegac' \in \dAc \cdot A$ and 

  $\upsilon:=\mu_{B} \circ \bupsilon$, $\omega:=\mu_{C}\circ \comega$,
  $\upsilon':=\mu_{B}\circ \upsilon'_{B}$, $\omega':=\mu_{C} \circ
  \omega'_{C}$. Then
  \begin{enumerate}
  \item  $\eps \circ \lambda(\bupsilon) =   \upsilon$,
    $\eps \circ \rho(\comega) =  \omega$, $\eps \circ
    \lambda(\upsilonb') = \upsilon'$, 
    $\eps \circ
    \rho(\omegac')  = \omega'$;
  \item $\lambda(\bupsilon)$ and $\lambda(\upsilonb')$ commute
    with both $\rho(\comega)$ and $\rho(\omegac')$;
  \item $\upsilon \circ \rho(\comega) = \omega \circ
    \lambda(\bupsilon)$,  $\upsilon \circ \rho(\omegac') =
    \omega' \circ \lambda(\bupsilon)$, $\upsilon' \circ
    \rho(\comega) = \omega \circ \lambda(\upsilonb')$ and
    $\upsilon' \circ \rho(\omegac') = \omega' \circ
    \lambda(\upsilonb')$;
  \item if $\upsilon=\upsilon'$ and elements of $\dmAm$
        separate the points of $A$, then
    $\lambda(\bupsilon)=\lambda(\upsilonb')$; if  $\omega=\omega'$ and
    elements of $\dmAm$ separate the points of $A$, then
    $\rho(\comega)=\rho(\omegac')$.
  \end{enumerate}
\end{lemma}
\begin{proof}
  (1) We only prove the first equation; the others follow
  similarly. Let $a,c\in A$ and $\bupsilon \in \dbA$. Then the counit
  property implies
  \begin{align*}
    (\eps \circ \lambda(c \cdot \bupsilon))(a) &=
    \eps(\bupsilon \oo \id)(\Delta_{B}(a)(c \oo 1))  \\
    &=\upsilon(\id \oo  \beps)(\Delta_{B}(a)(c \oo 1)) =\upsilon(ac) =
    (c \cdot \upsilon)(a).
  \end{align*}

  (2)  Let $a,b,c\in A$ and $\bupsilon \in \dbA$, $\omega \in \dAc$. We
  show that $\lambda(b\cdot \bupsilon)$ commutes with
  $\rho(\omegac \cdot c)$, and the other commutation relations
  follow similarly. The coassociativity condition relating $\Delta_{B}$ and
$\Delta_{C}$ implies that for all $a\in A$,
  \begin{align*}
    (\lambda(b\cdot \bupsilon)\circ(\rho(\omegac\cdot c))(a) &= (\bupsilon \otimes \id
    \otimes \omegac)((\Delta_{B} \otimes \id)((1 \otimes
    c)\Delta_{C}(a))(b \otimes 1
    \otimes 1)) \\
    &=(\bupsilon \otimes \id \otimes \omegac)((1 \otimes 1 \otimes c)( \id
    \otimes \Delta_{C})(\Delta_{B}(a)(b\otimes 1))) \\ &=
    (\rho(\omegac\cdot c) \circ \lambda(b\cdot \bupsilon))(a).
  \end{align*}

  (3)    Again, we only prove the first equation. By (1) and (2),
  \begin{align*}
    \upsilon \circ \rho(\comega) &= \eps \circ
    \lambda(\bupsilon) \circ \rho(\comega) = \eps \circ
    \rho(\comega) \circ \lambda(\bupsilon) = \omega \circ
    \lambda(\bupsilon).  
  \end{align*}

  (4) Assume that $\upsilon=\upsilon'$ and that elements of $\dmAm$
  separate the points of $A$. Since $A$ is non-degenerate, then also
  functionals of the form like $\omega$ separate the points of $A$,
  and by (3), $\omega \circ \lambda(\bupsilon) = \upsilon \circ
  \rho(\comega) = \upsilon' \circ \rho(\comega) =  \omega \circ
  \lambda(\upsilonb')$, whence $\lambda(\bupsilon) =
  \lambda(\upsilonb')$. A similar argument proves the assertion
  concerning $\rho(\comega)$ and $\rho(\omegac')$.
\end{proof}
We come back to the study of convolution operators associated to
general elements of $\dbA,\dAb,\dcA$ and $\dAc$, respectively, when we
have proved the existence of a modular automorphism for full and
faithful integrals. The next step towards this aim is to rewrite the
strong invariance relations  in diagrams
\eqref{dg:strong-invariance-left} and
\eqref{dg:strong-invariance-right} as follows.
 \begin{corollary} \label{corollary:dual-strong-invariance}
  Let $\phi$ be a left and $\psi$ a right integral on a regular
  multiplier Hopf algebroid $\mathcal{A}$ with base weight $\mu$. Then
  \begin{align*}
    \rho(\phic \cdot a)(b) &=    S(\rho(b\cdot \cphi)(a)), &
    \lambda(a\cdot \bpsi)(b) &= S(\lambda(\psib \cdot b)(a)) &&\text{for
    all } a,b\in A.
  \end{align*}
\end{corollary}
\begin{proof}
  Combine Proposition \ref{proposition:invariant-elements-hopf} and
Lemma \ref{lemma:convolution}.
\end{proof}
Iterated applications of the relations above and Theorem
\ref{theorem:uniqueness-full} yield the following result.
\begin{theorem} \label{theorem:modular}
  Let $\phi$ be a full left and $\psi$ a full right integral on a regular
  multiplier Hopf algebroid $\mathcal{A}$ with base weight
  $\mu$. Then $A\cdot \phi = \phi \cdot A$ and $A \cdot \psi = \psi
  \cdot A$.
\end{theorem}
\begin{proof}
  By Theorem \ref{theorem:uniqueness-full}, it suffices to prove the
  assertion for a full left integral $\phi$. Let $\psi:=\phi \circ
  S$. Then $\psi$ is a full right integral and $A \cdot \phi = A \cdot
  \psi$ by Theorem \ref{theorem:uniqueness-full}. We show that $\phi
  \cdot A \subseteq A \cdot \psi$, and a similar argument
  proves the reverse inclusion.

  Let $a,b,c \in A$.  Then
  \begin{align*}
    ((a \cdot \phi) \circ \lambda(\psib \cdot b))(c) &=
    ((\psi \cdot b) \circ \rho(a \cdot \cphi))(c)  =
    ((S(b) \cdot \phi) \circ S  \circ \rho(a \cdot \cphi))(c).
  \end{align*}
  Choose $b' \in A$ with $S(b) \cdot \phi = b' \cdot \psi$ and use Corollary \ref{corollary:dual-strong-invariance} to rewrite the expression above in the form
\begin{align*}
  ((b' \cdot \psi) \circ \rho(\phic \cdot c))(a) &=
  ((\phi \cdot c) \circ \lambda(b'\cdot \bpsi))(a) \\ &=
  ((\phi \cdot c) \circ S \circ \lambda(\psib \cdot a))(b') =
  ((S^{-1}(c) \cdot \psi) \circ \lambda(\psib \cdot a))(b').
\end{align*}
Again, choose $c' \in A$
with $S^{-1}(c) \cdot \psi = c' \cdot \phi$ and rewrite  the expression above
in the form
\begin{align*}
  ((c' \cdot \phi)  \circ \lambda(\psib \cdot a))(b') &=
  ((\psi \cdot a) \circ \rho(c' \cdot \cphi))(b').
\end{align*}
Summarising, 
\begin{align*}
  \phi(\lambda(\psib \cdot b)(c) a) &=
      ((a \cdot \phi) \circ \lambda(\bpsi \cdot b))(c) \\ &=
  ((\psi \cdot a) \circ \rho(c' \cdot \cphi))(b') =
  \psi(a \rho(c' \cdot \cphi)(b')).
\end{align*}
Here, $b$ and $c \in A$ were arbitrary, and the linear span of all
elements of the form
\begin{align*}
  \lambda(\psib \cdot b)(c) &= (S_{C}^{-1} \circ \psib \oo \id)((b
  \oo 1)\Delta_{C}(c))
\end{align*}
is equal to $AS_{C}^{-1}(\psib(A))=AC=A$ because $\lT$ and $\psib$ are
surjective. Thus,  $\phi \cdot A \subseteq A
\cdot \psi=A \cdot \phi$. 
\end{proof}
\begin{theorem} \label{theorem:modular-automorphism}
  Let $\phi$ be a full and faithful left integral and  let $\psi$ be a
  full and faithful right integral on a regular multiplier Hopf
  algebroid $\mathcal{A}$ with base weight $\mu$. Then $\phi$ and
  $\psi$ possess modular automorphisms $\sigma^{\phi}$ and
  $\sigma^{\psi}$, respectively, whose extensions to $M(A)$
  satisfy
  \begin{align*}
  \sigma^{\phi}(y) &=\sigma^{\mu}_{C}(y)=S^{2}(y) \text{ for all }y
  \in C, &
  \sigma^{\psi}(x) &= \sigma^{\mu}_{B}(x)=S^{-2}(x)   \text{ for all }
  x\in B.
  \end{align*}
  Furthermore,
  \begin{align*}
    \Delta_{B} \circ \sigma^{\phi} &= (S^{2} \otimes \sigma^{\phi})
    \circ \Delta_{B}, &
    \Delta_{C} \circ \sigma^{\phi} &= (S^{2} \otimes \sigma^{\phi})
    \circ \Delta_{C}, \\
    \Delta_{B} \circ \sigma^{\psi} &= (\sigma^{\psi} \otimes S^{-2})
    \circ \Delta_{B}, &
    \Delta_{C} \circ \sigma^{\psi} &= (\sigma^{\psi} \otimes S^{-2})
    \circ \Delta_{C}.    
  \end{align*}
  If $\mathcal{A}$ is flat, then
  $\sigma^{\phi}(M(B))=M(B)$ and $\sigma^{\psi}(M(C))=M(C)$.
\end{theorem}
\begin{proof}
  We prove the assertion for a full and faithful left integral $\phi$;
  the case of right integrals is similar.  

  Since $\phi$ is faithful and $A\cdot \phi = \phi \cdot A$ by
Theorem \ref{theorem:modular}, there exists a unique
  bijection $\sigma^{\phi}\colon A \to A$ such that $\phi \cdot a =
  \sigma^{\phi}(a) \cdot \phi$ for all $a \in A$, and this map is
  easily seen to be an algebra automorphism.

  For all $y\in C$, we have $\sigma^{\phi}(y)=\sigma^{\mu}_{C}(y)$
  because for all $a\in A$,
  \begin{align*}
    \phi(ya) = \mu_{C}(y\cphic(a)) =
    \mu_{C}(\cphic(a)\sigma^{\mu}_{C}(y)) = \phi(a\sigma^{\mu}_{C}(y)).
  \end{align*}

  The tensor product $S^{2} \otimes \sigma^{\phi}$ is well-defined on
  $\AlA$ and $\ArA$ because
  $\sigma^{\phi}(y)=\sigma^{\mu}_{C}(y)=S^{2}(y)$ for all $y\in C$.
  Two applications of Corollary \ref{corollary:dual-strong-invariance}
  show that for all $a,b \in A$,
  \begin{align*}
    \rho(\phi \cdot a)(\sigma^{\phi}(b)) = S(\rho(\sigma^{\phi}(b)
    \cdot \phi)(a)) &=S(\rho(\phi \cdot b)(a)) \\
    &=S^{2}(\rho(a \cdot
    \phi)(b)) = S^{2}(\rho((\phi \cdot a) \circ \sigma^{\phi})(b)).
\end{align*}
Since $a \in A$ was arbitrary, we can conclude   $\Delta_{B} \circ \sigma^{\phi} = (S^{2} \otimes \sigma^{\phi})
    \circ \Delta_{B}$ and $
    \Delta_{C} \circ \sigma^{\phi} = (S^{2} \otimes \sigma^{\phi})
    \circ \Delta_{C}$.

    Finally, the relation $\sigma^{\phi}(M(B))=M(B)$ follows
    immediately from the relation $M(B) \cdot \phi = \phi \cdot M(B)$
    obtained in Theorem \ref{theorem:integrals-uniqueness}.
\end{proof}
The following additional relations will be useful later.
\begin{lemma} \label{lemma:bphi-phib} Let $\phi$ be a full and
  faithful left integral on a flat regular multiplier Hopf algebroid
  $\mathcal{A}$ with base weight $\mu$. Then
  \begin{enumerate}
  \item $ \phib(xa) = (S^{2} \circ \sigma^{\phi})(x)\phib(a)$ and
    $\bphi(ax) = (\sigma^{\phi} \circ S^{2})^{-1}(x)\bphi(a)$ for all
    $x\in M(B)$ and $a\in A$;
  \item $\bphi(A)$ and $\phib(A)$ are two-sided ideals in $B$ and
    $\ker \bphi= \{ a\in A : \phi(BaB) = 0\} = \ker \phib$;
  \item the map $\theta \colon \omegab(A) \to \bomega(A),
    \omegab(a) \mapsto \bomega(a)$, is well-defined and 
    $\mu_{B} \circ \theta =\mu_{B}$;
  \item
    $\mu_{B}(x\sigma^{\phi}(x'))=\mu_{B}(x'\theta(x))=\mu_{B}(S^{2}(\theta(x))x')$
    for all $x\in B$ and $x' \in M(B)$.
  \end{enumerate}
\end{lemma}
\begin{proof}
(1)  We only prove the first relation; the second one follows
  similarly. Let $x,x'\in B$ and $a\in A$. By the corollary above,
  $\sigma^{\phi}(x) \in M(B)$, whence
  \begin{align*}
    \mu_{B}(\phib(xa)x') = \phi(xax') &= \phi(ax'\sigma^{\phi}(x)) \\ &=
    \mu_{B}(\phib(a)x'\sigma^{\phi}(x)) =
    \mu_{B}(S^{2}(\sigma^{\phi}(x))\phib(a)x').
  \end{align*}
  Here, we used the relation $(\sigma^{\mu}_{B})^{-1}(x)=S^{2}(x)$; see
  Proposition \ref{proposition:counit-kms}.  Since $x' \in B$ was
  arbitrary and $\mu_{B}$ is faithful, the first relation follows.

  (2), (3)  It is easy to see that  (1) implies (2), which in turn
  implies (3).

  (4) Let $a\in A$, $x' \in M(B)$ and $x=\phib(a)$. Then
  \begin{align*}
    \mu_{B}(x\sigma^{\phi}(x')) = \phi(a\sigma^{\phi}(x')) &=
    \phi(x'a) \\ &= \mu_{B}(x'\bphi(a)) = \mu_{B}(x'\theta(x))
    =\mu_{B}(S^{2}(\theta(x))x'). \qedhere
  \end{align*}
\end{proof}

\subsection{More on convolution operators, and the modular element}
\label{subsection:modular-element}
The results obtained so far imply that left and right integrals on
regular multiplier Hopf algebroids are related by modular elements,
very much like in the theory of multiplier Hopf algebras.
\begin{theorem} \label{theorem:modular-element} Let $\phi$ be a full
  and faithful left integral and let $\psi$ be a right integral on a
  regular multiplier Hopf algebroid $\mathcal{A}$ with base weight
  $\mu$. Then there exists a unique multiplier $\delta_{\psi} \in
  M(A)$ such that $\psi = \delta_{\psi} \cdot \phi$, and this
  multiplier is invertible if and only if $\psi$ is full and faithful.
\end{theorem}
\begin{proof}
  Since $\phi$ is faithful and $A \cdot \psi \subseteq A \cdot \phi$
  and $\psi \cdot A \subseteq \phi \cdot A$ by Theorem
  \ref{theorem:uniqueness-full}, there exist unique linear maps
  $\alpha,\beta \colon A \to A$ such that $a \cdot \psi =
\beta(a) \cdot \phi$ and $\psi \cdot a = \phi \cdot
  \alpha(a)$ for all $a \in A$.  For all $a,b\in A$,
  \begin{align*}
    a\sigma^{\phi}(\alpha(b)) \cdot \phi = a \cdot \phi \cdot
    \alpha(b) = 
    a \cdot \psi \cdot b = \beta(a) \cdot \phi \cdot b =
    \beta(a)\sigma^{\phi}(b) \cdot \phi,
  \end{align*}
  that is, $a\sigma^{\phi}(\alpha(b)) =
  \beta(a)\sigma^{\phi}(b)$. Thus,
  $\delta:=(\beta,\sigma^{\phi} \circ \alpha \circ
  (\sigma^{\phi})^{-1})$ is a multiplier of $A$ and $\delta \cdot \phi
  = \psi$. 
  If $\psi$ is full and faithful, the inclusions $A \cdot \psi
  \subseteq A \cdot \phi$ and $\psi \cdot A \subseteq \phi \cdot A$
  are equalities and the maps $\beta,\alpha$ are
  invertible. Conversely, if $\delta$ is invertible, then evidently
  $\psi$ is faithful, and full because then $\cpsi(A)=\cphic(A
  \delta)= C$ and likewise $\psic (A)=C$ by Corollary \ref{corollary:full}.
\end{proof}

In the special case where $\psi$ equals $\phi \circ S^{-1}$ or $\phi
\circ S$, we can determine the behaviour of the comultiplication,
counits and antipode on the modular elements.  For that, we need a few
more results on the convolution operators introduced above.

Given a regular multiplier Hopf algebroid $\mathcal{A}$ with base
weight $\mu$ and a full and faithful left or right integral $\phi$ or
$\psi$, we have seen that the space
\begin{align} \label{eq:hata}
  \hat A:=A\cdot \phi = \phi \cdot A = A \cdot \psi = \psi \cdot A
  \subseteq \dmAm
\end{align}
does not depend on the choice of $\phi$ or $\psi$.  Elements of $\hat
A$ naturally extend to functionals on $M(A)$:
\begin{lemma} \label{lemma:extension-multipliers} Let $\mathcal{A}$ be
  a regular multiplier Hopf algebroid with a base weight $\mu$ and full and
  faithful integrals. Then there exists a unique embedding
  $j\colon\hat A \to \dual{M(A)}$ such that
  \begin{align*}
    j(a \cdot \phi)(T) &= \phi(Ta), & j(\phi \cdot
      a)(T) &= \phi(aT), & j(a \cdot \psi)(T) &= \psi(Ta),
    & j(\psi \cdot a)(T) &= \psi(aT)
  \end{align*}
  for every   $T\in M(A)$, $a\in A$ and every left integral $\phi$ and
  right integral $\psi$ on $\mathcal{A}$ adapted to $\mu$.
\end{lemma}
\begin{proof}
  The point is, of course, to show that the formulas given above are
  compatible in the sense that for each $\upsilon \in \hat A$, the
  extension $j(\upsilon)$ is well-defined. This can be done using
Theorem \ref{theorem:modular-automorphism} and Theorem
  \ref{theorem:modular-element}, or as follows.  Assume, for example,
  that $a \cdot \omega = b \cdot \psi$, where $\phi$ and $\psi$ are a
  left and a right integral on $\mathcal{A}$, respectively, and
  $a,b \in A$. We have to show that  $\phi(Ta)=\psi(Tb)$ for every
  $T\in M(A)$. Choose a full and faithful left integral $\phi'$ and
  write $a=\sum_{i} c_{i}d_{i}$ and $b=\sum_{j} e_{j}f_{j}$ with
  $a_{i},b_{i},e_{j},f_{j} \in A$.  By Theorem
  \ref{theorem:uniqueness-full}, we find $d'_{i},f'_{j} \in A$ such
  that $d_{i} \cdot \phi = d'_{i} \cdot \phi'$ and $f_{j} \cdot \psi =
  f'_{j} \cdot \phi'$ for all $i,j$. Then
  \begin{align*}
    \sum_{i} c_{i}d'_{i} \cdot \phi' = a \cdot \phi = b\cdot \psi  =
    \sum_{j} e_{j}f'_{j} \cdot \phi'.
  \end{align*}
  Since $\phi'$ is faithful, we can conclude  $\sum_{i} c_{i}d'_{i} =
  \sum_{j}e_{j}f'_{j}$ and hence
  \begin{align*}
    \phi(Ta) &= \sum_{i} \phi(Tc_{i}d_{i}) =
    \sum_{i}\phi'(Tc_{i}d'_{i}) = \sum_{j}\phi'(Te_{j}f'_{j}) =
    \sum_{j}  \psi(Te_{j}f_{j}) =   \psi(Tb). \qedhere
  \end{align*}
\end{proof}
We henceforth regard elements of $\hat A$ as functionals on $M(A)$
without mentioning the embedding $j$  explicitly.
\begin{proposition} \label{proposition:convolution} Let $\mathcal{A}$
  be a regular multiplier Hopf algebroid with base weight $\mu$ and
  full and faithful integrals. Then the space $\hat A$ in
  \eqref{eq:hata} separates the points of $A$, and for all $\upsilon
  \in \dmAm$,$\omega\in \hat A$ and $b\in A$, the following relations
  hold:
  \begin{enumerate}
  \item $(\rho(\upsilonb)(b),\rho(\bupsilon)(b))$ and
    $(\lambda(\upsilonc)(b),\lambda(\cupsilon)(b))$ form two-sided
    multipliers $\rho(\upsilon)(b)$ and $\lambda(\upsilon)(b)$ of $A$,
    respectively, and $\lambda(\bomega)(b) =
    \lambda(\omegab)(b)$, $ \rho(\comega)(b) = \rho(\omegac)(b)$;
  \item $\upsilon \circ \rho(\omega)=\omega \circ \lambda(\upsilon)$
    and $\upsilon \circ \lambda(\omega) = \omega \circ
    \rho(\upsilon)$, and both compositions lie in $\dmAm$;
  \item $\rho(\upsilon \circ \rho(\omega)) =
    \rho(\upsilon)\rho(\omega)$ and $\lambda(\upsilon \circ
    \lambda(\omega)) = \lambda(\upsilon)\lambda(\omega)$;
  \item $\rho(\upsilon) \circ S = S \circ \lambda(\upsilon \circ S)$
    and $\lambda(\upsilon) \circ S = S \circ \rho(\upsilon \circ S)$;
  \item for all $x,x',x'' \in B$, $y,y',y'' \in C$,
    \begin{align*}
      \rho(S_{B}(x'')x \cdot\upsilon \cdot y''x') (yby')&= S_{C}(y'')y
      \rho(\upsilon)(xbx) y'x'', \\
      \lambda(x'' y \cdot\upsilon \cdot S_{C}(y'') y')(xbx') &=
      y''x\lambda(\upsilon)(y'by)x'S_{B}(x'').
    \end{align*}
  \end{enumerate}
\end{proposition}
\begin{proof}
  (1), (2) First, observe that Lemma \ref{lemma:convolution} (4)
  implies $\lambda(\bomega)(b) = \lambda(\omegab)(b)$ and $
  \rho(\comega)(b) = \rho(\omegac)(b)$. We can therefore drop the
  subscripts $B$ and $C$ and write $\lambda(\omega)$ and
  $\rho(\omega)$ from now on.

 Let $\omega = \phi \cdot a$, where $\phi$ is a full and
  faithful left integral and $a\in A$. Then
  \begin{align*}
    \omega(\rho(\upsilonc)(b)) = \phi(a\rho(\upsilonc)(b)) &= (\phi
    \oor \upsilon)((a \oo 1)\Delta_{C}(b)) = \upsilon(\lambda( \phi
    \cdot a)(b)) = \upsilon(\lambda(\omega)(b)).
  \end{align*}
  A similar calculation shows that $\omega(\rho_{C}(\upsilon)(b)) =
  \upsilon(\lambda(\omega)(b))$. For $\omega\in \hat A$ of the
  form $\omega=c\cdot \phi \cdot a$ with $a,c\in A$, we obtain
  \begin{align*}
    \phi((a\rho(\upsilonc)(b))c) = \upsilon(\lambda(\omega)(b)) =
    \phi(a(\rho(\cupsilon)(b)c)).
  \end{align*}
  Since $a,c\in A$ were arbitrary and $\phi$ is faithful, we can
  conclude that $(a\rho(\upsilonc)(b))c =
  a(\rho(\cupsilon)(b)c)$ for all $a,c\in A$ so that
  $\rho(\cupsilon)(b)$ and $\rho_{C}(\upsilon)(b)$ form a two-sided
  multiplier $\rho(\upsilon)(b)$ as claimed.

  Along the way, we just showed that $\omega \circ \rho(\upsilon) =
  \upsilon \circ \lambda(\omega)$. One easily verifies that this
  composition belongs to $\dmAm$, for example, $_{B}(\upsilon \circ
  \lambda(\omega)) =\bupsilon \circ \lambda(\omega)$ and $(\omega
  \circ \rho(\upsilon))_{C} = \omegac \circ \rho(\upsilon)$.

  (3) Let $\omega' \in \hat A$. Then $\lambda(\omega')$ commutes with
  $\rho(\omega)$ by Lemma \ref{lemma:convolution} and hence
  \begin{align*}
    \omega' \circ \rho(\upsilon) \circ \rho(\omega) = \upsilon \circ
    \lambda(\omega')\circ \rho(\omega) = \upsilon \circ \rho(\omega)
    \circ \lambda(\omega') = \omega' \circ \rho(\upsilon \circ
    \rho(\omega)).
  \end{align*}
  Since $\omega' \in \hat A$ was arbitrary and $\hat A$ separates the
  points of $A$, we can conclude the first relation in (3).  A similar argument proves the second
  one.

(4)  These relations follow from the fact that $S$ reverses the
  comultiplication, see
  \eqref{dg:left-counit}, \eqref{eq:right-counit} and
  \eqref{dg:galois-aux2}.

  (5) This follows from \eqref{eq:left-comult-module} and
  \eqref{eq:right-comult-module}; for example, the first equation
  without $y''$ holds because for all $a \in A$,
  \begin{align*}
    (\rho(S_{B}(x'')x \cdot \upsilon \cdot x')(yby'))a &= (\id \otimes
    S_{B}^{-1}\circ\cupsilon)((1 \otimes
    x')\Delta_{B}(yby')(a \otimes S_{B}(x'')x)) \\
    &= (\id \otimes S_{B}^{-1}\circ\cupsilon) ((y \otimes
    1)\Delta_{B}(x'bx)(y'x''a \otimes 1))  \\
    &= y \rho(\upsilon)(x'bx) y'x''a.  \qedhere
  \end{align*}
\end{proof}
We leave the further study of the convolution operators till
Subsection \ref{subsection:dual-graphs}, where
we construct the dual algebra of a regular multiplier Hopf
algebroid.

Now, we return to the modular elements defined above and determine the
behaviour of the comultiplication, counit and antipode on the modular
elements relating a left integral $\phi$ to the right integrals $\phi
\circ S^{-1}$ and $\phi\circ S$. For a multiplier Hopf $*$-algebroid,
of particular interest is the case that $\phi$ is self-adjoint
in the sense that $\phi(a^{*})=\overline{\phi(a)}$ for all $a\in A$.
\begin{theorem} \label{theorem:modular-element-second} Let $\phi$ be
  a full and faithful left integral on a regular multiplier Hopf
  algebroid $\mathcal{A}$ with base weight $\mu$. Then there exist
  unique invertible multipliers $\delta, \delta^{\dag} \in M(A)$ such
  that
  $\phi \circ S = \delta^{\dag} \cdot \phi$ and $
  \phi \circ S^{-1} =
  \phi \cdot \delta$. These elements satisfy
  \begin{gather*}
    \begin{aligned}
      \bphi(a)\delta^{\dag} &= \lambda(\phi)(a) = \delta \phib(a)
      \text{ for all } a \in A,
    \end{aligned} \\
    \begin{aligned}
       S(\delta^{\dag}) &= \delta^{-1}, &
      \eps \cdot \delta &= \eps = \delta^{\dag} \cdot \eps, &
      \Delta_{B}(\delta) &= \delta^{\dag} \oo \delta, &
      \Delta_{C}(\delta^{\dag}) &= \delta \oo \delta^{\dag}.
    \end{aligned}
  \end{gather*}
  If $\mathcal{A}$ is flat, then $\Delta_{C}(\delta) = \delta \oo
  \delta$ and $\Delta_{B}(\delta^{\dag}) = \delta^{\dag} \oo
  \delta^{\dag}$.  If $\mathcal{A}$ is a multiplier Hopf $*$-algebroid
  and $\phi$ is self-adjoint, then $\delta^{\dag}=\delta^{*}$.
\end{theorem}
\begin{proof}
  The compositions $\psi^{\dag} := \phi \circ S$ and $\psi := \phi
  \circ S^{-1}$ are full and faithful right integrals by Proposition
  \ref{proposition:integrals-antipode} and Remark \ref{remark:full},
  and of the form $\psi^{\dag} = \delta^{\dag} \cdot \phi$ and $\psi =
  \phi \cdot \delta$ with unique invertible multipliers
  $\delta,\delta^{\dag} \in M(A)$ by Theorem
  \ref{theorem:modular-element} and Theorem
  \ref{theorem:modular-automorphism}.  By Proposition
  \ref{proposition:convolution} and Corollary
  \ref{corollary:dual-strong-invariance},
  \begin{align*}
    \phi(\lambda(\phi)(a)b) &= ((b \cdot \phi) \circ \lambda(\phi))(a)
    \\
    &= (\phi \circ \rho(b\cdot \phi))(a) \\ &=
    ((\phi \circ S^{-1}) \circ \rho(\phi \cdot a))(b)
    \\ &= ((\phi \cdot a) \circ \lambda(\psi))(b) \\ &= (\phi \cdot
    a)(\bpsib(b)) = \phi(a\bpsib(b)) = \psi(\phib(a)b) =  \phi(\delta\phib(a)b)
  \end{align*}
  for all $a,b\in A$, and hence $\lambda(\phi)(a) = \delta\phib(a)$ for
  all $a\in A$. A similar calculation shows that $\lambda(\phi)(a) =
  \bphi(a)\delta^{\dag}$ for all $a\in A$.

  The relation $\delta^{-1} = S(\delta^{\dag})$ holds because
  \begin{align*}
    \phi = \phi \circ S \circ S^{-1} = (\delta^{\dag} \cdot \phi)
    \circ S^{-1} = (\phi \circ S^{-1}) \cdot S(\delta^{\dag}) = \phi
    \cdot \delta S(\delta^{\dag}).
  \end{align*}

 The properties of the counit imply that
  \begin{align*}
    \eps(\delta\phib(a)b) &= \eps(\lambda(\phi)(a)b) = (\phi \oob
    \eps)(\Delta(a)(1 \oo b)) = \phi(a\beps(b)) = \eps(\phib(a)b)
  \end{align*}
  for all $a,b \in A$, whence $\eps \cdot \delta = \eps$. A similar
  calculation shows that $\delta^{\dag}
  \cdot \eps = \eps$.

  The relations for the comultiplication require a bit more work. For all
  $a,b\in A$,
  \begin{align*}
    (\phi \ool \phi)(\Delta_{B}(\delta a)(1 \oo b)) &=
    \phi(\lambda(\phi)(\delta a)b) \\ &= \phi(\delta\phib(\delta a) b) =
    \psi(\bpsib(a)b) = (\psi \ool \psi)(\Delta_{B}(a)(1\oo b)).
  \end{align*}
  Since the map $\Tr$ is surjective, we can conclude  that for all
  $c,d \in A$,
  \begin{align*}
    (\phi \ool \phi)(\Delta_{B}(\delta)(c \oo d)) &= (\psi \ool
    \psi)(c \oo d) 
    = \psi(S_{B}(\bpsib(c))d) 
    = \phi(\delta S_{B}(\phib(\delta c)) d).
  \end{align*}
  Inserting the relation $\phib(\delta c)S^{-1}(\delta) = \phib(\delta
  c)(\delta^{\dag})^{-1} = \delta^{-1}\bphi(\delta c)$, the
  expression above becomes
  \begin{align*}
         \phi(S(\phib(\delta c)S^{-1}(\delta))d) &=
         \phi(S(\delta^{-1}\bphi(\delta c))d) \\ &=\phi(S_{B}(\bphi(\delta
         c))S(\delta^{-1})d) = (\phi \ool \phi)(\delta c \oo S(\delta^{-1})d).
  \end{align*}
  Consequently, $\Delta_{B}(\delta) = \delta \oo S(\delta)^{-1}$.
  Since the antipode reverses the comultiplication (see \eqref{dg:galois-aux2}), we can conclude
  \begin{align*}
    \Delta_{C}(\delta^{\dag}) = \Delta_{C}(S^{-1}(\delta)^{-1}) =
    \delta \oo S^{-1}(\delta^{-1}) =  \delta \oo \delta^{\dag}.
  \end{align*}
  A similar argument shows that $\Delta_{B}(\delta) =\delta^{\dag} \oo
  \delta$. 

  If $\mathcal{A}$ is flat, then  the module $(\ArA)_{(1 \oo B)}$ is
  non-degenerate by a similar argument as in Lemma 2.6 in
  \cite{timmermann:regular}, and the relation
  \begin{align*}
    \Delta_{C}(\delta) (1 \oo \phib(a)) = \Delta_{C}(\delta \phib(a))
    = \Delta_{C}(\bphi(a)\tilde \delta) = \delta \oo
    \bphi(a)\delta^{\dag} = (\delta \oo \delta)(1 \oo \phib(a)),
  \end{align*}
  valid for all $a \in A$, implies $\Delta_{C}(\delta) = \delta \oo
  \delta$. A similar reasoning shows that in this case,
  $\Delta_{B}(\delta^{\dag})=\delta^{\dag} \oo \delta^{\dag}$.

  Finally, if $\mathcal{A}$ is a multiplier Hopf $*$-algebroid and
  $\phi$ is self-adjoint, then
  \begin{align*}
    \phi(a^{*}\delta^{*}) &= \overline{\phi(\delta a)} = \overline{\phi
    \circ S^{-1}(a)} = \phi(S^{-1}(a)^{*}) = \phi(S(a^{*})) = \phi(a^{*}\delta^{\dag})
  \end{align*}
  for all $a \in A$, where we used \eqref{eq:involutions}, and hence $\delta^{*}=\delta^{\dag}$. 
\end{proof}
% \begin{corollary}
%   Let $\phi$ be a full and faithful left integral and let $\psi$ be a
%   right integral on a flat regular multiplier Hopf algebroid
%   $\mathcal{A}$ with base weight $\mu$ and define
%   $\delta_{\psi},\delta\delta,\delta^{\dag} \in M(A)$ by $\delta_{\psi} \cdot \phi =
%   \psi$, $\delta^{\dag}\cdot \phi = \phi \circ S$ and $\phi \circ
%   S^{-1} = \phi \circ S^{-1}$; see Theorem
%   \ref{theorem:modular-element} and Theorem
%   \ref{theorem:modular-element-second}. Then
%   $\Delta_{B}(\delta_{\psi}) = \delta_{\psi} \oo \delta^{\dag}$ and
%   $\Delta_{C}(\delta_{\psi}) = \delta \oo \delta_{\psi}$.
% \end{corollary}
% \begin{proof}
%   By Theorem \ref{theorem:integrals-uniqueness}, $\psi = y\cdot (\phi
%   \circ S)$ for some $y\in M(C)$, and hence $\delta_{\psi} =
%   y\delta^{\dag}$. Now, the assertion follows from
%   \eqref{eq:left-comult-module}, \eqref{eq:right-comult-module} and
%   Theorem \ref{theorem:modular-element-second}.
% \end{proof}
The proof of Theorem \ref{theorem:modular-element} shows that
the tensor products $\delta^{\dag} \oo \delta$, $\delta^{\dag}\oo
\delta^{\dag}$ and $\delta \oo \delta$, $\delta\oo \delta^{\dag}$ are
well-defined as operators on $\AlA$ or $\ArA$, respectively. This
also follows from the following intertwining relations.
\begin{lemma}
  Let $\phi$ be a full and faithful left integral on a flat regular
  multiplier Hopf algebroid $\mathcal{A}$ with base weight $\mu$, and
  denote by $\delta, \delta^{\dag} \in M(A)$ the unique invertible
  multipliers such that $\phi \circ S = \delta^{\dag} \cdot \phi$ and
  $\phi \circ S^{-1} = \phi \cdot \delta$.  Then for all $ x\in M(B)$,
\begin{align*}
  \delta x &= \delta S^{2}(\sigma^{\phi}(x)), & x \delta^{\dag} &=
  \delta^{\dag} \sigma^{\phi}(S^{2}(x)), \\
S(x)  \delta   &=  \delta S(\sigma^{\phi}(S^{2}(x))), &
\delta^{\dag}S^{-1}(x) &= S(\sigma^{\phi}(x)) \delta^{\dag}.
\end{align*}
\end{lemma}
\begin{proof}
    By Lemma \ref{lemma:bphi-phib},
  \begin{align*}
    \delta S^{2}(\sigma^{\phi}(x)) \phib(a) = \delta \phib(xa) =
    \lambda(\phi)(xa) = x\lambda(\phi)(a) = x \delta \phib(a),
  \end{align*}
  and a similar calculation shows that $x \delta^{\dag} =
  \delta^{\dag} \sigma^{\phi}(S^{2}(x))$. To deduce the remaining
  assertions, use the relation $S(\tilde \delta) = \delta^{-1}$.
\end{proof}

\subsection{Faithfulness}
\label{subsection:faithful}
Non-trivial integrals on multiplier Hopf algebras are always faithful,
as shown by Van Daele in \cite{daele:1}.  We now prove a corresponding
statement for integrals on regular multiplier Hopf algebroids,
where the former need to be full and the latter projective. Note that both assumptions are trivially satisfied in the case of multiplier
Hopf algebras. 
\begin{theorem} \label{theorem:integrals-faithful}
Every full left and every full  right
  integral on a projective  regular multiplier Hopf algebroid
  with a base weight is faithful.
\end{theorem}
\begin{proof}
  We only prove the assertion for left integrals, closely following
  the argument in \cite{daele:1}. A similar reasoning applies to right
  integrals.

  Let $\phi$ be a full left integral on a projective regular
  multiplier Hopf algebroid $\mathcal{A}$ with base weight $\mu$, let
  $a\in A$ and assume $ a\cdot \phi=0$.  Since $\mu_{C}$ is faithful,
  we can conclude that then also $a\cdot \cphi=0$.  Let $b\in A,\upsilonb
  \in \dAb$ and $c := \lambda(\upsilonb \cdot b)(a)$.  Then
  \begin{align} \label{eq:faithful-1}
    \begin{aligned}
      \rho(\phic \cdot d)(c) &= (\rho(\phic \cdot d) \circ
      \lambda(\upsilonb
      \cdot b))(a) \\
      &= (\lambda(\upsilonb \cdot b) \circ \rho(\phic \cdot d) )(a) =
      (\lambda(\upsilonb \cdot b) \circ S \circ \rho(a \cdot \cphi))(d)
      = 0
    \end{aligned}
  \end{align}
  for all $d\in A$, where we used Lemma \ref{lemma:convolution} and
  Corollary \ref{corollary:dual-strong-invariance}.

  We show that the relation above implies $\eps(c)=0$, and then we can
  conclude, using Lemma \ref{lemma:convolution} again, that
  \begin{align*}
    (\mu_{B} \circ \upsilonb)(ba) = (\eps \circ \lambda(\upsilonb
    \cdot b))(a) = \eps(c)=0. 
  \end{align*}
  Using the facts that $\mu_{B}$ is faithful, $b\in A$ is arbitrary
  and $A$ is nondegenerate, and that $\upsilonb \in \dAb$ is arbitrary
  and $\dAb$ separates the points of $A$ because the module $\Ab$ is
  projective, we can deduce that $a=0$.

  So, let us prove that \eqref{eq:faithful-1} implies $\eps(c)=0$.
  Since $\Ab$ is projective, there exist elements $e_{i} \in A$ and
  $\omega^{i}_{B} \in \dAb$ such that $\sum_{i}e_{i}\omega^{i}_{B}(d)$
  for all $d\in A$, the sum always being finite, and since $A \cdot
  \phi= \phi \cdot A$ by Theorem \ref{theorem:modular}, we find
  elements $e'_{i} \in A$ such that $S^{-1}(e_{i}) \cdot \phi = \phi
  \cdot e'_{i}$ for all $i$. Let $f \in A$ and write $(f \oo
  1)\Delta_{C}(c) = \sum_{j} f_{j} \oo c_{j}$ with $f_{j},c_{j} \in
  A$.  We take $d=e'_{i}$ in \eqref{eq:faithful-1}, multiply on the
  left by $f$, apply $\omega^{i}:=\mu_{B} \circ \omega_{B}^{i}$, sum
  over $i$, and find
  \begin{align*}
    0 = \sum_{i} ((\omega^{i} \cdot f) \circ \rho(\phic \cdot
    e'_{i}))(c)  &
    = \sum_{i} ((\phi \cdot e'_{i}) \circ \lambda(\omega_{B}^{i}
    \cdot f))(c) 
\\ &= \sum_{i,j}\phi(c_{j}S_{B}^{-1}(\omega_{B}^{i}(f_{j}))S^{-1}(e_{i})) = \sum_{j} \phi(c_{j}S^{-1}(f_{j})),
  \end{align*}
  where we used Lemma \ref{lemma:convolution} again.  The second
  diagram in \eqref{dg:antipode} shows that
 \begin{align*}
   \sum_{j} c_{j}S^{-1}(f_{j}) = \sum_{j} S^{-1}(f_{j}S(c_{j})) =
   S^{-1}(f S_{B}(\beps(c))) =\beps(c)S^{-1}(f),
 \end{align*}
and hence
\begin{align*}
   0 = \phi(\beps(c)S^{-1}(f)) = \mu_{B}(\beps(c)\bphi(S^{-1}(f))).
 \end{align*}
Since $f\in A$ was arbitrary, $\bphi$ is surjective and $\mu_{B}$ is
faithful, we can conclude $\beps(c)=0$ and hence $\eps (c)=0$.
\end{proof}

 \section{Duality}
\label{section:duality}

In this section, we establish a generalized Pontrjagin duality for regular
multiplier Hopf algebroids with full and faithful integrals.  In
contrast to the duality results on Hopf algebroids found in the
literature \cite{boehm:integrals}, \cite{kadison:inclusions},
\cite{schauenburg}, we need no finiteness or Frobenius condition
assumption. As
outlined in the introduction,  we follow the
approach taken by Van Daele in the case of multiplier Hopf algebras
\cite{daele:1}, but face a number of conceptual and technical
difficulties. 

\subsection{The dual algebra and dual  quantum graphs}
\label{subsection:dual-graphs}

We start with a flat regular multiplier Hopf algebroid with full and
faithful integrals. Let us call a functional $\omega$ on a $*$-algebra
$D$ \emph{positive} if $\omega(d^{*}d) \geq 0$ for all $d\in D$.
\begin{definition}
  A \emph{measured regular multiplier Hopf algebroid} is a flat
  regular multiplier Hopf algebroid $\mathcal{A}$ with a base weight
  $\mu$ and left and right integrals $\phi$ and $\psi$, respectively,
  which are full and faithful. If $\mathcal{A}$ is a multiplier Hopf
  $*$-algebroid and the functionals $\mu_{B},\mu_{C},\phi,\psi$ are
   positive, we call $(\mathcal{A},\mu,\phi,\psi)$ a
  \emph{measured multiplier Hopf $*$-algebroid}.
\end{definition}
Given a measured regular multiplier Hopf algebroid
$(\mathcal{A},\mu,\phi,\psi)$, we have seen that
\begin{align*}
    A\cdot \phi = \phi \cdot A = A \cdot \psi = \psi \cdot A
  \subseteq \dmAm.
\end{align*}
We denote this space by $\hat A$  as before, and equip it with a
convolution product as follows.
\begin{proposition} \label{proposition:dual-algebra} Let
  $(\mathcal{A},\mu,\phi,\psi)$ be a measured regular multiplier Hopf
  algebroid. Then the subspace $\hat A =A\cdot\phi \subseteq \dmAm$ is
  an idempotent, non-degenerate algebra, $A$ is an idempotent,
  non-degenerate and faithful $\hat A$-bimodule and $\dmAm$ is a
  non-degenerate and faithful $\hat A$-bimodule with respect to the
  products given by
  \begin{gather*}
    \omega \omega' := \omega \circ \rho(\omega') = \omega' \circ
    \lambda(\omega) \quad \text{for all } \omega,\omega' \in \hat A,
    \\ \begin{aligned} \omega \ast a &= \rho(\omega)(a), & a \ast
      \omega &= \lambda(\omega)(a) && \text{for all } a\in A, \omega
      \in \hat
      A, \\
      \omega\upsilon &= \upsilon\circ \lambda(\omega), &
      \upsilon\omega &= \upsilon \circ \rho(\omega) && \text{for all }
      \upsilon \in \dmAm, \omega \in \hat A.
    \end{aligned}
  \end{gather*}
\end{proposition}
We call $\hat A$ the \emph{dual algebra} of
$(\mathcal{A},\mu,\phi,\psi)$.
\begin{proof}
  We first show that the product defined on $\hat A$ takes values in
  $\hat A$ again. So, let $\omega,\omega' \in \hat A$. Then $\omega
  \circ \rho(\omega')=\omega' \circ \lambda(\omega)$ by Proposition
  \ref{proposition:convolution}. To see that this functional lies in
  $\hat A$, write $\omega=a\cdot \phi$ and $\omega' = b\cdot \psi$
  with $a,b\in A$ and $a \otimes b = \sum_{i} \Delta_{B}(d_{i})(e_{i}
  \otimes 1)$ with $d_{i},e_{i} \in A$. Then
  \begin{align*}
    (\omega \circ \rho(\omega'))(c) &=
    (\phi \ool \phi)(\Delta_{B}(c)(a \oo b)) \\
    &=\sum_{i} (\phi \ool \phi)(\Delta_{B}(cd_{i})(e_{i} \otimes 1))
    \\ &      = \sum_{i} \phi(\cphic(cd_{i})e_{i}) =\sum_{i}\phi(cd_{i}\cphic(e_{i}))
  \end{align*}
  for all $c \in A$,
  that is
  \begin{align} \label{eq:dual-product} \omega \circ \rho(\omega') =
    f\cdot \phi \ \text{ if } \ \omega=a\cdot \phi, \ \omega'=b\cdot
    \phi, \ f= (\cphic \otimes \id)(\Tl^{-1}(a\otimes b)).
  \end{align}
  
  The products defined above turn $\hat A$ into an algebra and $A$ and
  $\dmAm$ into $\hat A$-bimodules because
  $\rho(\omega)\rho(\omega')=\rho(\omega \circ \rho(\omega'))$ and
  $\lambda(\omega')\lambda(\omega) = \lambda(\omega' \circ
  \lambda(\omega))$ by Proposition \ref{proposition:convolution} (3).

  Equation \eqref{eq:dual-product} and bijectivity of the canonical
  maps $\Tl,\Tr$ imply that $\hat A$ is idempotent as an algebra and
  $A$ is idempotent as an $\hat A$-bimodule. These facts and
  non-degeneracy of the pairing $\hat A \times A \to A$, $(\upsilon,a)
  \mapsto \upsilon(a)$ imply that $A$ is non-degenerate and faithful
  as an $\hat A$-bimodule. But $A$ being faithful and idempotent as an
  $\hat A$-bimodule implies that the algebra $\hat A$ is
  non-degenerate, and that  $\dmAm$ is non-degenerate and
  faithful as an $\hat A$-bimodule.
\end{proof}

The algebras  $B,C$ and $\hat A$ can be assembled into left and
right quantum graphs as follows.
\begin{proposition} \label{proposition:dual-quantum-graphs} Let
  $(\mathcal{A},\mu,\phi,\psi)$ be a measured regular multiplier Hopf
  algebroid with dual algebra $\hat A$ and let
  \begin{align*}
    \hat B &:= C,  & \hat C &:= B, & \hat S_{ \hat B} &:= S_{B}^{-1}
    \colon \hat B \to \hat C, & \hat S_{\hat C} &:= S_{C}^{-1} \colon
    \hat C \to \hat B.
  \end{align*}
  Then there exist embeddings $\iota_{\hat B} \colon \hat B \to M(\hat
  A)$ and $\iota_{\hat C} \colon \hat C \to M(\hat A)$ such that
    \begin{align*}
      y \omega &= y \cdot \omega, & \omega y &= S_{B}^{-1}(y) \cdot
      \omega, & x\omega &= \omega \cdot S_{C}^{-1}(x), & \omega x &=
      \omega \cdot x
    \end{align*}
    for all $y\in \hat B$, $x \in \hat C$, $\omega \in \hat A$, and
    the tuples
    \begin{align*}
  \hat{\mathcal{A}}_{\hat B}&:=(\hat A,\hat B,
    \iota_{\hat B},\hat S_{\hat B}) &&\text{and}  &
\hat{\mathcal{A}}_{\hat
      C}&:=(\hat A,\hat C,\iota_{\hat C},\hat S_{\hat C})    
    \end{align*}
    form a left and a right quantum graph, respectively, which are
    flat and compatible. 
\end{proposition}
We call $\hat{\mathcal{A}}_{\hat B}$ and $\hat{\mathcal{A}}_{\hat C}$  the \emph{dual left} and \emph{dual right quantum graph} of
$(\mathcal{A},\mu,\phi,\psi)$, respectively.
\begin{proof}
 For each $x\in B=\hat C$, the formulas above define a multiplier
  $\iota_{\hat C}(x)$ of $\hat A$ because 
  \begin{align*}
    \upsilon ( x\omega) = \upsilon \circ \rho(\omega \cdot S^{-1}(x)) =
    (\upsilon \cdot x) \circ \rho(\omega) = (\upsilon x)
    \omega
  \end{align*}
  for all $\upsilon,\omega \in \hat A$ by Proposition
  \ref{proposition:convolution} (5). The
  assignment $x\mapsto \iota_{\hat C}(x)$ evidently is a
  homomorphism. Similar arguments show that the formulas above define
  a homomorphism $\iota_{\hat C} \colon \hat C\to M(\hat
  A)$. Evidently, $\iota_{\hat B}(\hat B)$ and $\iota_{\hat C}(\hat
  C)$ commute. 

 The module $\hbA$ is idempotent, non-degenerate, faithful and
  flat because the map $A \to \hat A$ given by $a \mapsto a \cdot
  \phi$ is an isomorphism from $\cA$ to $\hbA$ and the module $\cA$
  has the desired properties. Note that faithfulness implies
  injectivity of $\iota_{\hat B}$. Similar arguments apply to the
  modules $\hAb,\hcA,\hAc$ and yield injectivity of $\iota_{\hat C}$.
\end{proof}

In the case of a multiplier Hopf $*$-algebroid, we obtain a natural
involution on the dual algebra $\hat A$ which is admissible with the
dual quantum graphs in the sense explained in Subsection
\ref{subsection:hopf-algebroids}. The proof uses the following easy
observation.
\begin{lemma} \label{lemma:dual-space}
  \begin{enumerate}
  \item Let $(\mathcal{A},\mu,\phi,\psi)$ be a measured regular
    multiplier Hopf algebroid. Then $\upsilon \circ S \in \hat A$ for
    all $\upsilon \in \hat A$.
  \item Let $(\mathcal{A},\mu,\phi,\psi)$ be a
    measured multiplier Hopf $*$-algebroid. Then $\ast \circ \upsilon
    \circ \ast \in \hat A$ for all $\upsilon \in \hat A$.
  \end{enumerate}
\end{lemma}
\begin{proof}
 Let $\upsilon \in \hat A$ and write
$\upsilon = a \cdot \phi$ with $a\in A$. In (1), $\upsilon \circ S = (\phi
\circ S) \cdot S^{-1}(a)$ lies in $\hat A$ because $\phi \circ S$ is
right-invariant, and in (2), $\ast \circ \upsilon
  \circ \ast = \phi \cdot a^{*} \in \hat A$.
\end{proof}
\begin{proposition} \label{proposition:dual-involution}
  Let $(\mathcal{A},\mu,\phi,\psi)$ be a measured multiplier Hopf
  $*$-algebroid. Then the dual algebra $\hat A$ is a $*$-algebra with
  respect to the involution given by $\omega \mapsto \omega^{*}:=\ast
  \circ \omega \circ \ast \circ S$, and this involution is admissible
  with respect to the dual quantum graphs $\hat { \mathcal{A}}_{\hat
    B}$ and $\hat{\mathcal{A}}_{\hat C}$.
\end{proposition}
\begin{proof}
  The map $\omega \mapsto \omega^{*}$ sends $\hat A$ to $\hat A$ by
  Lemma \ref{lemma:dual-space} and is involutive because $\ast \circ S
  \circ \ast \circ S = \id_{A}$ by \eqref{eq:involutions}. To see that it reverses the multiplication,
  note first that
  $\rho(\omega) \circ \ast = \ast \circ \rho(\ast \circ \omega
    \circ \ast)$
  because $\Delta_{B}$ and $\Delta_{C}$ are $*$-homomorphisms. Using
  Proposition \ref{proposition:dual-algebra}, we can conclude  \begin{align*}
    (\upsilon \ast \omega)^{*} &= \ast \circ \upsilon \circ \rho(\omega) \circ
    \ast \circ S =  \ast\circ \upsilon \circ \ast \circ S \circ \lambda(\ast
    \circ \omega
    \circ \ast \circ S) = \omega^{*} \ast \upsilon^{*}.
  \end{align*}
   The embedding $\iota_{\hat B}$  is a $*$-homomorphism because for all $y\in \hat B=C$ and
  $\omega \in \hat A$,
  \begin{align*}
    (y\omega)^{*} &= \ast \circ (y \cdot \omega) \circ \ast \circ S =
     S^{-1}(y^{*}) \cdot (\ast \circ \omega \circ \ast\circ S) = \omega^{*} y^{*}.
  \end{align*}
  A similar calculation shows that $\iota_{\hat C}$ is a
  $*$-homomorphism as well.
\end{proof}

Given a measured regular multiplier Hopf algebroid with dual algebra
$\hat A$, let
\begin{align}
  \label{eq:1}
\hat \mu &= (\hat \mu_{\hat B},\hat \mu_{\hat C}), && \text{where} &
\hat \mu_{\hat B}&:=\mu_{C} \in \dual{\hat B}, & \hat
\mu_{\hat C}&:=\mu_{B} \in \dual{\hat C}.
\end{align}
Dual to $\dmAm$, we denote by
  \begin{align} \label{eq:hdmam}
    \hdmAm &:= (\hat\mu_{\hat B})_{*}(_{\hat B}\hat A) \cap
    (\hat\mu_{\hat B})_{*}(\hat A _{\hat B}) \cap(\hat\mu_{\hat
      C})_{*}(_{\hat C}\hat A) \cap(\hat\mu_{C})_{*}(\hat A_{\hat C})
    \subseteq \dual{(\hat A)}
  \end{align}
  the space of all functionals $\omega$ on $\hat A$ that can be
  written in the form
\begin{align*}
  \omega = \hat\mu_{\hat B} \circ {_{\hat B}\omega} = 
\hat\mu_{\hat B} \circ {\omega_{\hat B}} = \hat\mu_{\hat C} \circ {_{\hat C}\omega} = \hat\mu_{\hat C} \circ {\omega_{\hat C}}
\end{align*}
with $_{\hat B}\omega \in {_{\hat B}\hat A}$, $\omega_{\hat B} \in
\hat A _{\hat B}$, ${_{\hat C}\omega}\in {_{\hat C}\hat A} $,
$\omega_{\hat C}\in\hat A_{\hat C}$.  
By Lemma \ref{lemma:extension-multipliers}, functionals in $\hat A$
naturally extend to  $M(A)$. Dually, evaluation on a multiplier  gives
a functional in $\hdmAm$.
\begin{lemma} \label{lemma:dual-evaluation}
  Let $(\mathcal{A},\mu,\phi,\psi)$ be a measured regular multiplier
  Hopf algebroid with dual algebra $\hat A$ and dual base weight
  $\hat\mu$ and let $T \in M(A)$. Then the functional  $\dual{T}\colon
  \hat A \to \C$ given by $\omega \to \omega(T)$ lies in $\hdmAm$, and     for all $a \in A$,
    \begin{align*}
      {_{\hat B}(\dual{T})}(\phi \cdot a) &=  S^{2}(\cphic(aT)), &
      (\dual{T}_{\hat B})(\psi \cdot a) &= S^{-1}(\bpsib(aT)), \\
      {_{\hat C}(\dual{T})}(a \cdot \phi) &= S^{-1}(\cphic(Ta)), &
      (\dual{T}_{\hat C})(a \cdot \psi) &= S^{2}(\bpsib(Ta)).
    \end{align*}
\end{lemma}
\begin{proof}
  The functional $\dual{T}$ lies in $(\hat\mu_{\hat B})_{*}(_{\hat
    B}\dual{\hat A}) \cap (\hat\mu_{\hat B})_{*}(\dual{\hat A}_{\hat
    B})$ and the equations concerning $_{\hat B}(\dual{T})$ and
  $\dual{T}_{\hat B}$ hold because for all $y\in C=\hat B$,
  \begin{align*}
    \dual{T}(y(\phi \cdot a))&= \dual{T}(y\cdot \phi \cdot a) =
    \phi(aTy) =\mu_{C}(\cphic(aT)y) = 
    \mu_{C}(yS^{2}(\cphic(aT)))
\end{align*}
and likewise
\begin{align*}
    \dual{T}((\psi \cdot a)y) =  \dual{T}(S^{-1}(y) \cdot \psi  \cdot
    a)  &=
 \mu_{B}(\bpsib(aT)S^{-1}(y))
\\ &    = \mu_{B}(S(y)\bpsib(aT)) = \mu_{C}(S^{-1}(\bpsib(aT))y).
  \end{align*}
  Similar calculations show that $\dual{T}$ lies in $(\hat\mu_{\hat
    C})_{*}(_{\hat C}\dual{\hat A}) \cap (\hat\mu_{\hat
    C})_{*}(\dual{\hat A}_{\hat C})$ and that the equations concerning
  $_{\hat C}(\dual{T})$ and $\dual{T}_{\hat C}$ hold.
\end{proof}

\subsection{The dual regular multiplier Hopf algebroid}
\label{subsection:dual-pairs}

We now equip the dual left and right quantum graphs
$\hat{\mathcal{A}}_{\hat B}$ and $\hat{\mathcal{A}}_{\hat C}$ of a
measured regular multiplier Hopf algebroid
$(\mathcal{A},\mu,\phi,\psi)$ with a left and a right comultiplication
so that we obtain a flat regular multiplier Hopf algebroid
$\hat{\mathcal{A}}$.  Rather than defining these comultiplications
directly, we construct the associated canonical maps.  We start with
the canonical maps of $\mathcal{A}$, which induce dual maps
\begin{align} \label{eq:dual-canonical-1}
  \dual{(\Tr)} &\colon \dual{(\AlA)} \to \dual{(\AbA)}, &
  \dual{(\lT)} &\colon \dual{(\ArA)}\to \dual{(\AcA)}, \\ \label{eq:dual-canonical-2}
  \dual{(\Tl\Sigma)} &\colon \dual{(\AlA)} \to
  \dual{(\AcA)}, &
  \dual{(\rT\Sigma)} &\colon \dual{(\ArA)} \to
  \dual{(\AbA)}.
\end{align}
Next, we  identifiy the relative tensor products
\begin{align*}
  \hAcA, && \hAbA, && \hAlA = {_{\hat B}\hat A} \oo {_{\hat S_{\hat B}(\hat
      B)}\hat A}, &&
  \hArA = \hat A_{\hat S_{\hat C}(\hat C)} \oo \hat A_{\hat C}
\end{align*}
with certain subspaces of the domains and ranges of the maps above as
follows. As explained in Subsection \ref{subsection:uniqueness}, we
can form for all $\upsilon,\omega \in \hat A\subseteq \dmAm$ certain
tensor products
 \begin{align} \label{eq:dual-products}
\upsilon \ool \omega &\in \dual{(\AlA)}, &
   \upsilon \oor \omega &\in \dual{(\ArA)}, &
   \upsilon \ooc \omega &\in \dual{(\AcA)}, &
   \upsilon \oob \omega &\in \dual{(\AbA)}.
 \end{align}
 Using the embedding $A \to \hdmAm$ defined in
Lemma \ref{lemma:dual-evaluation}, we can also
 form for all $a,b\in A$ tensor products
 \begin{align*}
   \dual{a} \oohc \dual{b} &\in \dual{(\hAcA)}, &
   \dual{a} \oohb \dual{b} &\in \dual{(\hAbA)}, &
   \dual{a} \oohl \dual{b} &\in \dual{(\hAlA)}, &
   \dual{a} \oohr \dual{b} &\in \dual{(\hArA)}.   
 \end{align*}
 \begin{lemma} \label{lemma:dual-pairings}
   For all $a,b\in A$ and $\upsilon,\omega \in \hat A$, the following
   relations hold:
   \begin{align*}
     (\upsilon \ool \omega)(a \otimes b) &= (\dual{a} \oohc
     \dual{b})(\upsilon \oo \omega), &
     (\upsilon \oor \omega)(a \otimes b) &= (\dual{a} \oohb
     \dual{b})(\upsilon \oo \omega),  \\
     (\upsilon \ooc \omega)(a \otimes b) &= (\dual{a} \oohl
     \dual{b})(\upsilon \oo \omega), &
     (\upsilon \oob \omega)(a \otimes b) &= (\dual{a} \oohr
     \dual{b})(\upsilon \oo \omega).     
   \end{align*}
 \end{lemma}
 \begin{proof}
   We only  prove the first equation. By Lemma
   \ref{lemma:dual-evaluation},
   \begin{align*}
     (\upsilon \ool \omega)(a \oo b) =
     \mu_{B}(S^{-1}(\comega(b))\bupsilon(a)) 
     &=\mu_{B}(S^{2}(\bupsilon(a))S^{-1}(\comega(b))) \\
     &= \hat\mu_{\hat C}(\dual{a}_{\hat C}(\upsilon){_{\hat
         C}\dual{b}}(\omega)) = (\dual{a} \oohc \dual{b})(\upsilon \oo
     \omega). \qedhere
\end{align*}
 \end{proof}
\begin{lemma} \label{lemma:dual-embeddings}
There exist well-defined embeddings
    \begin{align*}
      \hAcA &\hookrightarrow \dual{(\AlA)}, & \hAbA &\hookrightarrow
      \dual{(\ArA)}, & \hAlA &\hookrightarrow \dual{(\AcA)}, &\hArA
      &\hookrightarrow \dual{(\AbA)}
    \end{align*}
    that map each elementary tensor $\upsilon \oo \omega$ to the
    respective products in \eqref{eq:dual-products}.
\end{lemma}
\begin{proof}
  Lemma \ref{lemma:dual-pairings} implies that these maps are
  well-defined. We show that the first map is injective, and similar
  arguments apply to the other maps.  For all $a,b\in A$ and
  $\upsilon,\omega \in \hat A$,
  \begin{align*}
    (\upsilon \ool \omega)(a \oo b) = (\dual{a} \oohc
    \dual{b})(\upsilon \oo \omega) = \dual{a}((\id \oo {_{\hat
        C}\dual{b}})(\upsilon \oo \omega)).
  \end{align*}
Since  maps of the form
  $\dual{a}$ and $_{\hat C}\dual{b}$ with $a,b\in A$ separate the
  points of $\hat A$ and the module $\hAc$ is flat,  products of the
  form $\dual{a} \oohc \dual{b}$ separate the points of
  $\hAcA$. Therefore, the map 
  $\hAcA \hookrightarrow \dual{(\AlA)}$ is injective. 
\end{proof}
We henceforth identify the respective domains and ranges of the
embeddings above without further notice.
\begin{lemma} \label{lemma:dual-bijections}
  The maps $\dual{(\Tr)}$, $\dual{(\lT)}$, $\dual{(\Tl\Sigma)}$ and
  $\dual{(\rT\Sigma)}$ induce bijections
  \begin{align*}
    \hlT  &\colon \hAcA \to \hArA, & {\hTr} &\colon \hAbA \to \hAlA,  \\
    {\hTl\Sigma} &\colon\hAcA \to \hAlA, & \hrT\Sigma &\colon \hAbA
    \to \hArA.
  \end{align*}
\end{lemma}
\begin{proof}
  Again, we only treat the first map, and similar arguments apply to
  the others. We only need to show that $\dual{(\Tr)}(\hAcA) = \hArA$
  or, equivalently, that the two subspaces of $\dual{(\AbA)}$ spanned
  by functionals of the form $(\upsilon \ool \omega) \circ \Tr$ or
  $\upsilon \oob \omega$, respectively, where $\upsilon,\omega \in
  \hat A$, coincide.  So, let $\upsilon,\omega \in \hat A$, $a,b\in A$
  and write $\upsilon=\psi\cdot c$, $\omega= \psi \cdot d$ with $c,d\in A$. Then
  \begin{align*}
    (\upsilon \ool \omega)(\Tr(a \oo b)) &=     (\upsilon \ool \omega)(\Delta_{B}(a)(1 \oo b))) = \psi(d\lambda(\psi \cdot c)(a)b), \\
    (\upsilon \oob \omega)(a \oo b) &= \omega(\bpsib(ca)b) =
    \psi(d\bpsib(ca)b).
  \end{align*}
  Thus, it suffices to show that the two subspaces of $\End(A)$
  spanned by maps of the form $a \mapsto d\lambda(\psi \cdot c)(a)$ or
  $a \mapsto d\bpsib(ca)$, where $c,d\in A$, coincide.  But
  \begin{align*}
    d\lambda(\psi \cdot c)(a) &= (S_{C}^{-1}\circ \bpsib \oo \id)((c
    \oo d)\Delta_{C}(a)),
  \end{align*}
and using bijectivity of $\rT$ and right-invariance of $\bpsib$, we
can conclude  that maps of the form $a\mapsto d\lambda(\psi \cdot
c)(a)$, where $c,d\in A$, and maps of the form
\begin{align*}
  a \mapsto (S_{C}^{-1}\circ \bpsib \oo \id)((1 \oo
    d')\Delta_{C}(c'a)) = d'\bpsib(c'a),
\end{align*}
where  $c',d'\in A$, span the same subspace.
\end{proof}

To prove that the bijections $(\hTl,\hTr)$ and $(\hlT,\hrT)$ obtained
above form multiplicative pairs, we will use the functional on $\hat
A$ given by evaluation at $1_{M(A)}$, see Lemma \ref{lemma:dual-evaluation}, which later will turn out to be
the counit of the dual multiplier Hopf algebroid $\hat{\mathcal{A}}$.
 \begin{lemma} \label{lemma:dual-counit} Let $\heps:=\dual{1_{M(A)}}
   \in \hdmAm$. Then for all $a\in A, y\in C,x\in B,\omega \in \hat
   A$, \begin{align*} \hbeps(a\cdot \phi) &= \cphic(a), & \hepsb(a
     \cdot \psi) &= S_{B}(\bpsib(a)), & \hbeps(\hat S_{\hat B}(y)\omega)
     &= \hbeps(\omega)y, \\ \hceps(\phi \cdot a) &= S_{C}(\cphic(a)), &
     \hepsc(\psi \cdot a) &= \bpsib(a), & \hepsc(\omega \hat S_{\hat
       C}(x)) &= x\hepsc(\omega),
\end{align*}
  and the following
diagrams commute, where $\hat m$ denotes the multiplication maps:
  \begin{align*}
    \xymatrix@R=18pt{
\hArA \ar[rd]_{\hepsb \otimes \id} &
      \hAbA \ar[l]_{\hrT\Sigma} \ar[r]^{\hTr} \ar[d]_{\hat m} & \hAlA
      \ar[ld]^{\hat S_{\hat C}^{-1} \circ \hbeps \otimes
        \id}  &
\hArA \ar[rd]_{\id \otimes \hepsc} &
      \hAcA \ar[l]_{\hlT} \ar[r]^{\hTl \Sigma} \ar[d]_{\hat m} & \hAlA \ar[ld]^{\id \otimes
\hat S_{\hat B}^{-1} \circ      \hceps} &
      \\
      & \hat A , & & & \hat A,
    } 
  \end{align*}
\end{lemma}
\begin{proof}
We first prove the formulas involving
  $\hbeps$ and $\hepsb$, which follow from the fact that
  \begin{align*}
    \hat \eps(y(a \cdot \phi)) &= \hat\eps(ya\cdot\phi) = \phi(ya)=
    \mc(y\cphic(a)), \\
    \hat\eps((a\cdot \psi)y) &= \hat\eps(S^{-1}(y)a\cdot \psi) =
    \mu_{B}(S^{-1}(y)\bpsib(a)) = \mu_{C}(S(\bpsib(a))y), \\
    \hbeps(\hat S(y)(a \cdot \phi)) &= \hbeps(a \cdot \phi \cdot
    S^{-2}(y)) = \hbeps(ay\cdot \phi) = \cphic(a)y
  \end{align*}
for all $x\in B$, $y\in C$, $a \in A$. The formulas involving $\hceps$
and $\hepsc$ follow similarly.

Next, we show that $(\id \oo \hepsc) \circ \hlT =\hat m$;
commutativity of the other triangles follows similarly. Let
$\upsilon,\omega \in \hat A$ and $a \in A$.  Then
  \begin{align*}
    ((\id \oo \hepsc)(\hlT(\upsilon \oo \omega)))(a) &=
    (\dual{a} \oohr \heps)((\upsilon \ool \omega) \circ \Tr) 
\\ & =\heps((\upsilon \ool \omega)(\Tr(a \oo -)))\\
&= \heps(\omega\cdot \lambda(\upsilon)(a))
=\omega(\lambda(\upsilon)(a)) = (\upsilon \omega)(a). \qedhere
  \end{align*}
\end{proof}
We can now prove the main result of this subsection.
\begin{theorem} \label{theorem:dual-hopf-algebroid} Let
  $(\mathcal{A},\mu,\phi,\psi)$ be a measured regular multiplier Hopf
  algebroid. Then there exist a unique left and a unique right
  comultiplication $\hat \Delta_{\hat B}$ and $\hat\Delta_{\hat C}$
  for the dual quantum graphs $\hat{\mathcal{A}}_{\hat B}$ and
  $\hat{\mathcal{A}}_{\hat C}$, respectively, such that the associated
  canonical maps $(\hTl,\hTr)$ and $(\hlT,\hrT)$ make the following diagrams commute:
  \begin{align*}
    \xymatrix@C=18pt@R=8pt{
      \hAcA  \ar[r]^{\hTl \Sigma} \ar@{^(->}[d] & \hAlA
      \ar@{^(->}[d]  &
      \hAbA \ar[r]^{\hTr} \ar@{^(->}[d] & \hAlA  \ar@{^(->}[d] 
      \\
      \dual{(\AlA)} \ar[r]_{\dual{(\Tl\Sigma)}} & \dual{(\AcA)} &
      \dual{(\ArA)} \ar[r]_{\dual{(\lT)}} & \dual{(\AcA)}  \\
      \hAcA \ar[r]^{\hlT} \ar@{^(->}[d] & \hArA \ar@{^(->}[d] &
      \hAbA \ar[r]^{\hrT\Sigma} \ar@{^(->}[d] & \hArA \ar@{^(->}[d]
      \\
      \dual{(\AlA)} \ar[r]_{\dual{(\Tr)}} & \dual{(\AbA)} &
      \dual{(\ArA)} \ar[r]_{\dual{(\rT\Sigma)}} & \dual{(\AbA)} 
    } 
  \end{align*}
  The pair $\hat{\mathcal{A}}=((\hat{\mathcal{A}}_{\hat B},\hat
  \Delta_{\hat B}),(\hat{\mathcal{A}}_{\hat C},\hat\Delta_{\hat C}))$
  is a flat, regular multiplier Hopf algebroid with antipode
  $\hat S$ and left and right counits $\hbeps$ and $\hepsc$ given by
  \begin{align*}
    \hat S(\omega) &= \omega \circ S, & \hbeps(a \cdot \phi) &= \cphic(a),
    & \hepsc(\psi \cdot a) &= \bpsib(a) &&\text{for all } \omega \in
    \hat A, a\in A.
  \end{align*}
  The pair $\hat\mu=(\hat\mu_{\hat B},\hat\mu_{\hat
    C}):=(\mu_{C},\mu_{B})$ is a base weight for $\hat{\mathcal{A}}$. Finally,
  if $(\mathcal{A},\mu,\phi,\psi)$ is a measured multiplier Hopf
  $*$-algebroid, then $\hat{\mathcal{A}}$ is a multiplier Hopf
  $*$-algebroid with respect to the involution $\omega \mapsto \ast
  \circ \omega \circ \ast \circ S$ on $\hat A$.
\end{theorem}
In the situation above, we call $\hat{\mathcal{A}}$ the \emph{dual regular
multiplier Hopf algebroid} or \emph{dual multiplier Hopf $*$-algebroid} of $(\mathcal{A},\mu,\phi,\psi)$.
\begin{proof}
  We first show that $(\hTl,\hTr)$ make the diagrams
  \eqref{dg:left-galois-1} and \eqref{dg:left-galois-2} commute.  Consider
  the following diagram:
  \begin{align*}
    \xymatrix@R=15pt@C=35pt{ \sbA \oo {_{(B \otimes 1)}(\ArA)} \ar@{<-}[r]^{\id \oo \lT}
      \ar@{<-}[d]_(0.55){\Sigma\Tl\Sigma \oo
        \id} & \sbA \oo {_{(B \otimes 1)}(\AcA)} \ar@{<-}[d]^(0.55){\Sigma\Tl\Sigma \oo \id}     \\
      \AcArA \ar@{<-}[r]_{\id \oo \lT} & \AcAcA. }
  \end{align*}
  It commutes because  both compositions are given by
  \begin{align*}
    a \otimes b\otimes c \mapsto (\Sigma(\Delta_{B}(a)) \otimes 1)(1
    \otimes b\otimes 1)(1 \otimes \Delta_{C}(c)).
  \end{align*}
  Taking transposes, we find that the square in the following diagram
  commutes:
  \begin{gather*}
    \xymatrix@R=10pt{ \hat A \oohC \hat A \oohb \hat A \ar[rr]^{\id
        \oo \hTr} \ar[dd]_{\hTl \oo \id} & & \hat A \oohC \hat A \oohl
      \hat A \ar[dd]^{\hTl \oo \id}
      \ar[ld]_{\hat m\Sigma \otimes \id} \\
      & \hAlA& \\
      \hAlA \oohb \hat A \ar[rr]_{\id \oo \hTr} \ar[ru]^{\id \otimes
        \hat m} && \hat A \oohl \hat A \oohl \hat A. \ar[lu]^(0.55){\id
        \otimes \hbeps \otimes \id}}
  \end{gather*}
  By Lemma \ref{lemma:dual-counit}, the left triangle and the lower
  triangle commute as well. Hence, the upper left triangle commutes as well.

  Taking the transpose of the pentagon diagram for $\lT$ given in
  Proposition 5.11 of \cite{timmermann:regular}, we find that the
  outer rectangle in the diagram below commutes:
  \begin{align*}
    \xymatrix@C=10pt@R=15pt{
      \hAbA \oohb \hat A \ar[rd]_{\hat m \otimes \id} \ar[dd]_(0.6){(\hTr)_{23}} \ar[r]^{(\hTr)_{12}} &
      \hAlA \oohb \hat A \ar[rr]^{(\hTr)_{23}} \ar[d]^{\hbeps \otimes
        \id \otimes \id} && \hAlA \oohl \hat A
      \ar[ld]_(0.6){\hbeps \otimes \id \otimes \id} \\  
      &  \hAbA \ar[r]_{\hTr}&  \hAlA & 
      \\
      \hat A_{\hat B} \oo  {_{(1 \otimes \hat B)}(\hAlA)} \ar[rrr]_{(\hTr)_{13}} && &
      {_{(\hat B\otimes 1)}(\hAbA)} \oo {_{\hat S(\hat B)}\hat A}
      \ar[uu]_(0.4){(\hTr)_{12}}   \ar[lu]^(0.55){\hat m \otimes \id}  }     
  \end{align*}
  The triangles commute by Lemma \ref{lemma:dual-counit} again, and
  hence the lower left cell commutes.

  Therefore, the diagrams (3.2) and (3.3) in Definition 3.1 in
  \cite{timmermann:regular} commute. The verification of the
  intertwining relations (3.1) in this definition is
  straightforward. Summarizing, we find that $(\hTl,\hTr)$ is a left
  multiplicative pair. By Proposition 3.3 in
  \cite{timmermann:regular}, there exists a unique associated left
  comultiplication $\hat\Delta_{\hat B}$ for $\hat{\mathcal{A}}_{\hat
    B}$.

  Similar arguments show that $(\hlT,\hrT)$ is a right multiplicative
  pair, and Proposition 5.8 in \cite{timmermann:regular} then yields
  a unique associated right comultiplication $\hat \Delta_{\hat C}$
  for $\hat{\mathcal{A}}_{\hat C}$.

  We claim that $(\hlT,\hrT)$ is compatible with $(\hTl,\hTr)$ in the
  sense of Definition 6.4 in \cite{timmermann:regular}.  Since
  $(\Tl,\Tr)$ and $(\lT,\rT)$ are compatible, the following diagrams
  commute:
    \begin{align*}
      \xymatrix@R=15pt{ \AcAbA \ar[r]^{\id \oo \Tr} \ar[d]_{\lT \oo
          \id} & \AcAlA \ar[d]^{\lT \oo \id}
&&        \ACABA \ar[r]^{\id \oo \rT} \ar[d]_{\Tl \oo \id} & \ar[d]^{\Tl
          \oo \id} \ACArA
        \\
        \ArAbA \ar[r]_{\id \oo \Tr} & \ArAlA,  && 
        \AlABA \ar[r]_{\id \oo \rT} & \AlArA. }
\end{align*}
Taking the transposes, one obtains the corresponding diagrams for
$(\hTl,\hTr)$ and $(\hTl,\hTr)$.

Therefore, the pair $\hat{\mathcal{A}}$ is a multiplier
bialgebroid. It is a regular multiplier Hopf algebroid because the maps
$\hTl,\hTr,\hlT,\hrT$ are bijective by Lemma \ref{lemma:dual-bijections}.

The formulas for the counits follow immediately from Lemma
\ref{lemma:dual-counit}. In particular, these formulas show that the
counits are surjective.

Let us show that the transpose of the antipode $S$ of $\mathcal{A}$
coincides with the antipode $\hat S$ of $\mathcal{\hat A}$.  By
\eqref{dg:galois-aux}, the
 following diagrams commute:
  \begin{align*}
    \xymatrix@R=18pt{
      \AbA \ar[r]^{\rT\Sigma} \ar[d]_{\Tr} & \ArA  & 
&
      \hAbA \ar[r]^{\hrT\Sigma} \ar[d]_{\hTr} & \hArA\\
      \AlA  \ar[r]^{S\otimes \id}& \AcA \ar[u]_{\lT}& &
      \hAlA \ar[r]^{\hat S\otimes \id}& \hAcA \ar[u]_{\hlT} }
  \end{align*}
 Taking the transpose of the first diagram gives the second
  diagram  with $\hat S \otimes \id$ replaced by 
  $\dual{S}\otimes \id$, whence $\hat S=\dual{S}$.

  The pair $\hat \mu$ is a base weight for $\hat{\mathcal{A}}$ because
  $\hat\mu_{\hat B} \circ {_{\hat B}\eps} = \eps = \hat\mu_{\hat C}
  \circ \eps_{\hat C}$ by Lemma \ref{lemma:dual-counit}.

  Finally, assume that $(\mathcal{A},\mu,\phi,\psi)$ is a measured
  multiplier Hopf $*$-algebroid. To prove that $\hat{\mathcal{A}}$ is
  a multiplier Hopf $*$-algebroid, it suffices to show that $(\ast \oo
  \ast) \circ \hTr \circ (\ast \oo \ast) = \hrT$. Inserting the
  definition of the involution and of the maps $\hTr$ and $\hrT$, we find
  \begin{align*}
    ((\ast \oo \ast)\circ \hTr \circ (\ast \oo \ast))(\upsilon \oohB
    \omega) &= (\upsilon \oohB \omega) \circ (* \oo *) \circ (S \oo
    S) \circ \lT
    \circ (S^{-1} \oo S^{-1}) \circ (* \oo *) \\
    &=(\upsilon \oohB \omega) \circ (* \oo *) \circ \Sigma \circ \rT
      \circ  \Sigma  \circ  \circ (* \oo *)  \\
      &= (\upsilon \oohB \omega) \circ \Sigma \circ \Tr \circ \Sigma
      \\
      &=  \hrT  (\upsilon \oohB \omega)
  \end{align*}
  for all $\upsilon,\omega \in \hat A$, where we used
  \eqref{dg:galois-aux2} 
  and Lemma 7.2 of \cite{timmermann:regular}.
\end{proof}
\subsection{The full duality and biduality}
\label{subsection:dual-integrals}

To conclude the main results of this article, the duality and
biduality of measured regular multiplier Hopf algebroids, we need to
construct invariant weights on the dual. This can again be done in a
similar way as in the case of multiplier Hopf algebras, see \cite{daele:1}.

\begin{theorem} \label{theorem:dual-measured} Let
  $(\mathcal{A},\mu,\phi,\psi)$ be a measured regular multiplier Hopf
  algebroid.  Then the dual regular multiplier Hopf algebroid
  $\hat{\mathcal{A}}$ constructed in Theorem
  \ref{theorem:dual-hopf-algebroid} and the functionals
  \begin{align} \label{eq:dual-measured-functionals}
    \hat \mu_{\hat B}&=\mu_{C}, & \hat \mu_{\hat C} &=\mu_{B}, &
    \hat \phi &\colon \hat A \to \C, \ \psi \cdot a \mapsto \eps(a),
& \hat
    \psi &\colon \hat A \to \C, \ a\cdot \phi \mapsto \eps(a),
  \end{align}
  form a measured regular multiplier Hopf algebroid again. Moreover,
  \begin{align} \label{eq:dual-left-integral} \hat \phi((\psi \cdot
    S(a))\omega) &= \omega(a), & \sigma^{\hat\phi}(\psi \cdot a) &=
    (\psi \circ S) \cdot S^{-2}(a), & \sigma^{\hat\phi}(y) &=
    (\sigma^{\psi})^{-1}(y), \\ \label{eq:dual-right-integral}
    \hat\psi (\omega (S(b) \cdot \phi)) &=
    \omega(b),   &
    \sigma^{\hat\psi}(a\cdot \phi)&= S^{2}(a) \cdot (\phi \circ S), &
    \sigma^{\hat\psi}(x) &= (\sigma^{\phi})^{-1}(x)
  \end{align}
  for all $a,b\in A$, $x\in B$, $y\in C$ and $\omega \in \hat
  A$. Finally, if   $(\mathcal{A},\mu,\phi,\psi)$ is a measured multiplier Hopf
  $*$-algebroid, then $(\hat{\mathcal{A}},\hat\mu,\hat\phi,\hat\psi)$
  is a  measured multiplier Hopf
  $*$-algebroid and
  \begin{align*}
    \hat \psi((a\cdot \phi)^{*}(a \cdot \phi)) &= \phi(a^{*}a)
    &&\text{and} &
    \hat\phi((\psi \cdot a)(\psi \cdot a)^{*}) &= \psi(aa^{*}) &&
    \text{for all } a\in A.
  \end{align*} 
\end{theorem}
In the situation above, we call
$(\hat{\mathcal{A}},\hat\mu,\hat\phi,\hat\psi)$ the \emph{dual} of
$(\mathcal{A},\mu,\phi,\psi)$.
\begin{proof} 
  To prove that $(\hat{\mathcal{A}},\hat\mu,\hat\phi,\hat\psi)$ is a
  measured regular multiplier Hopf algebroid, we only need to show
  that $\hat \phi$ and $\hat \psi$ are full and faithful left and
  right integrals on $\hat{\mathcal{A}}$ adapted to $\hat\mu$,
  respectively. We focus on $\hat \psi$; the arguments concerning
  $\hat \phi$ are similar.

  We first show that the functional $\hat \psi$ belongs to $\hdmAm$
  and that
  \begin{align} \label{eq:dual-right-integral-fac}
    \begin{gathered}
      \hbpsi(a\cdot \phi)=\ceps(a)=\hpsib(a\cdot \phi), \\ \hcpsi(a
      \cdot \phi) =S^{-1}(\epsc(a)), \quad \hpsic(a\cdot \phi)=
      S^{2}(\theta(\epsb(a)))
    \end{gathered}
  \end{align}
  for all $a\in A$, where $\theta$ is given by $\phib(a) \mapsto
  \bphi(a)$; see Lemma \ref{lemma:bphi-phib}.  Let $y,y'\in C$ and
  $x\in B$. Then
  \begin{align*}
    \hat\psi(y (a\cdot \phi) y') = \hat \psi(yS^{-1}(y')a \cdot \phi)
        &= \eps(yS^{-1}(y')a) \\
    &= \mb(S^{-1}(y')\beps(a)S^{-1}(y)) 
    = \mc(yS(\beps(a))y'),
  \end{align*}
  and $S\circ \beps=\ceps$ by Example
  \ref{example:base-counit}. Furthermore, 
  \begin{align*}
  \hat\psi(x (a \cdot \phi)) = \hat\psi(a \cdot \phi \cdot S^{-1}(x)) 
&= \hat\psi(aS(x) \cdot \phi) \\
&= \eps(aS(x))  = \mc(\epsc(a)S(x)) = \mb(xS^{-1}(\epsc(a))),
\end{align*}
where we used the relation $\sigma^{\phi}(S^{-1}(x))=S(x)$ (see
Theorem \ref{theorem:modular-automorphism}), and
\begin{align*}
  \hat\psi((a\cdot \phi)x) = \hat\psi(a \cdot \phi \cdot x)
  &= \hat\psi(a\sigma^{\phi}_{B}(x) \cdot \phi) \\
  &= \eps(a\sigma^{\phi}_{B}(x)) = \mb(\epsb(a)\sigma^{\phi}_{B}(x)) =
  \mb(S^{2}(\theta(\epsb(a))x)
\end{align*}
by Lemma \ref{lemma:bphi-phib}.  Since $a\in A$ was arbitrary, these
calculations show that $\hat\psi $ lies in $\hdmAm$ and that the
relations \eqref{eq:dual-right-integral-fac} hold. Using  Proposition
\ref{proposition:counits-full} (1) and the fact that $\phi$ is full, 
we see that $\epsb$ and $\epsc$ and hence also $\hcpsi$ and $\hpsic$ are surjective.

We next show that $\hbpsib$ is right-invariant.  Let
  $\upsilon=a\cdot \phi$ and $\omega = b\cdot \psi$. Then 
  \begin{align*}
    (\hat \psi \oohl \dual{c})(\hTr(\upsilon \oo \omega)) &=
    \hat \psi((\upsilon \oor \omega)(\lT(- \oo c))) \\
    &= \hat \psi(\rho(\omega)(c) \cdot \upsilon) \\
    &= \eps(\rho(\omega)(c)a) \\
    &= \mu_{C}(\epsc(\rho(\omega)(c))\ceps(a)) \\
    &= \mu_{C}(\omega_{C}(c)\ceps(a))  \\
    &= \mu_{C}(\ceps(a)S^{2}(\omega_{C}(c)))
    =(\hat\psi \oohb
    \dual{c})(\upsilon \oo \omega),
  \end{align*}
  where we used \ref{eq:base-counit-mult}, Lemma
  \ref{lemma:dual-evaluation}, Proposition
  \ref{proposition:convolution} (5) and 
  \eqref{eq:dual-right-integral-fac}. Since $\hat{\mathcal{A}}$ is flat
  and elements of the form $\dual{c}$ with $c\in A$ separate the
  points of $\hat A$, we can conclude that $\hat \psi$ is
  right-invariant.

Let us prove the first equation in  \eqref{eq:dual-right-integral}.  By
  \eqref{eq:dual-product}, $(a\cdot \phi) (b\cdot \phi) = c\cdot
  \phi$, where $c = (\cphic \oo \id)(\Tl^{-1}(a \oo b))$. Using the
  relation $\Tl^{-1} = (S^{-1} \oo \id)\lT(S\oo \id)$, see
  \eqref{dg:galois-inverse}, and the counit property, we find
\begin{align*}   
  \begin{aligned}
    \hat \psi((a \cdot \phi)(b\cdot \phi)) &= (\phi \ooC
    \eps)(\Tl^{-1}(a \oo b)) \\
    &= (\phi \circ S^{-1} \oor \eps)(\lT(S(a) \oo b)) =
    \phi(S^{-1}(S(a)b)) = \phi(S^{-1}(b)a).
  \end{aligned}
\end{align*}
Since $\phi$ is faithful, we can conclude that $\hat\psi$ is faithful
as well.  Moreover, with  $\omega := S^{-2}(b) \cdot (\phi \circ
    S)$,  we can conclude that
  \begin{align*}
    \hat \psi((a\cdot \phi)(b\cdot \phi)) &= \phi(S^{-1}(b)a) = (\phi
    \cdot S^{-1}(b)) (a) = \omega(S(a)) = \hat\psi(\omega(a \cdot
    \phi))
  \end{align*}
  and hence $(\sigma^{\hat\psi})^{-1}(b\cdot \phi) = S^{-2}(b) \cdot
  (\phi \circ S) = S^{-2}(b)\delta^{\dag} \cdot \phi$ in the notation
  of Theorem \ref{theorem:modular-element-second}. The formula for
  $\sigma^{\hat\psi}$ in \eqref{eq:dual-right-integral} follows.

  The relation $\sigma^{\hat\psi}(x) = (\sigma^{\phi})^{-1}(x)$ holds
  for all $x\in B$ because for all $a,b\in A$,
  \begin{align*}
    \hat \psi(x(a \cdot \phi)(b\cdot \phi)) &= \hat\psi((a \cdot \phi
    \cdot S^{-1}(x))(b \cdot \phi)) \\ &= \phi(S^{-1}(x)S^{-1}(b)a) \\
    & = \hat\psi((a\cdot \phi)(bx\cdot \phi)) = \hat\psi((a\cdot
    \phi)(b\cdot \phi)(\sigma^{\phi}_{B})^{-1}(x)). \qedhere
  \end{align*}  

  Finally, assume that $(\mathcal{A},\mu,\phi,\psi)$ is a measured
  multiplier Hopf $*$-algebroid. Then $\hat\mu_{\hat B}=\mu_{C}$ and
  $\hat\mu_{\hat C}=\mu_{B}$ are positive by assumption, and
  \begin{align*}
    \hat\psi((a\cdot \phi)^{*}(a\cdot \phi)) &= (a \cdot
    \phi)^{*}(S^{-1}(a)) =  \phi(S(S^{-1}(a))^{*}a)^{*} = \phi(a^{*}a)
  \end{align*}
  for all $a\in A$ by \eqref{eq:dual-right-integral}. A similar
  calculation shows that $ \hat\phi((\psi \cdot a)(\psi \cdot a)^{*})
  = \psi(aa^{*})$, and consequently, $\hat\psi$ and $\hat \phi$ are
  positive.\end{proof}

% Passage to the dual commutes with the passage to the bi-opposite:
% \begin{proposition} \label{proposition:dual-measured-opposite}
%   The dual of the bi-opposite of a measured regular multiplier Hopf
%   algebroid is the bi-opposite of its dual.
% \end{proposition}
% \begin{proof}
%   Let $(\mathcal{A},\mu,\phi,\psi)$ be a measured regular multiplier
%   Hopf algebroid. Then by Proposition
%   \ref{proposition:dual-hopf-algebroid-opposite} and construction, the
%   dual of the bi-opposite $(\mathcal{A}^{\co},\mu^{\co},\psi,\phi)$
%   consists of the regular multiplier Hopf algebroid
%   $\hat{\mathcal{A}}^{\co,\op}$, the base weight
%   $(\mu_{B},\mu_{C})=\hat\mu^{\co}$ and the functionals
%   $\hat\psi\colon a \cdot \phi \mapsto \eps(a)$ and $\hat\phi\colon\psi
%   \cdot a \mapsto \eps(a)$, because the roles of $\phi$ and $\psi$ are
%   flipped and the multiplication is reversed when passing from
%   $(\mathcal{A},\mu,\phi,\psi)$ to its bi-opposite.  But this is
%   precisely the bi-opposite of the dual
%   $(\hat{\mathcal{A}},\hat\mu,\hat\phi,\hat\psi)$.
% \end{proof}
% Unfortunately, we were not able to relate the dual of the
% opposite/co-opposite to the co-opposite/opposite of the dual directly.

\begin{theorem} \label{theorem:biduality} Let
  $(\mathcal{A},\mu,\phi,\psi)$ be a measured regular multiplier Hopf
  algebroid or a measured multiplier Hopf $*$-algebroid and let
  $(\hat{\mathcal{A}},\hat{\mu},\hat{\phi},\hat{\psi})$ be its
  dual. Then the map $A \to \dual{(\hat A)}$ given by $a\mapsto
  \dual{a}$ is an isomorphism from $(\mathcal{A},\mu,\phi,\psi)$ to
  the bi-dual
  $(\hhat{\mathcal{A}},\hhat{\mu},\hhat{\phi},\hhat{\psi})$.
\end{theorem}
\begin{proof}
  We first show that the map above identifies the left and right
  quantum graphs $(B,A,\iota_{B},S_{B})$ and $(C,A,\iota_{C},S_{C})$
  with the left and right quantum graphs
  $(\hhat{B},\hhat{A},\iota_{\hhat{B}},\hhat{S}_{\hhat{B}})$ and
  $(\hhat{C},\hhat{A},\iota_{\hhat{C}},\hhat{S}_{\hhat{C}})$,
  respectively.  By definition, $ \hhat{B}= \hat{C} = B$ and
  $\hhat{C}=\hat{B}= C$. Equation
  \eqref{eq:dual-right-integral} implies that for all $b\in
  A$ and $\upsilon \in \hat A$,
  \begin{align*}
    ((S(b) \cdot \phi) \cdot \hat \psi)(\upsilon) =
    \hat\psi(\upsilon(S(b) \cdot \phi)) = \upsilon(b) = \dual{b}(\upsilon).
  \end{align*}
  Since $\hhat{A} = \hat{A} \cdot \hat\psi$, we can conclude that the
  map $b \mapsto \dual{b} = (S(b) \cdot \phi) \cdot \hat \psi$ is a
  linear isomorphism from $A$ to $\hhat{A}$.  This isomorphism is
  $B$-linear because for all $x\in B$,
  \begin{align*}
    x((S(b) \cdot \phi) \cdot \hat \psi) &= x(S(b) \cdot \phi) \cdot
    \hat \psi  \\ &= (S(b) \cdot \phi \cdot S^{-1}(x)) \cdot \hat \psi =
    (S(b)S(x) \cdot \phi) \cdot \psi = S(xb) \cdot \psi
\end{align*}
and
\begin{align*}
    ((S(b) \cdot \phi) \cdot \hat\psi)x &= \hat S^{-1}(x)(S(b) \cdot
    \phi) \cdot \hat \psi = (S(x)S(b) \cdot \phi) \cdot \hat \psi =
    S(bx) \cdot \hat \phi.
  \end{align*}
  Similar calculations show that the isomorphism $b\mapsto \dual{b}$
  is $C$-linear.  By definition, $\hhat{S}_{\hhat{B}} =
  \hat{S}_{\hat{B}}^{-1} = S_{B}$, $ \hhat{S}_{\hhat{C}} =
  \hat{S}_{\hat{C}}^{-1} = S_{C}$, and
  the base weight $\hhat{\mu}$ is equal to $(\hat\mu_{\hat
    C},\hat\mu_{\hat B})=(\mu_{B},\mu_{C})$. 

  Let us next consider the canonical maps of $\hhat{\mathcal{A}}$.  By
  construction and Lemma \ref{lemma:dual-pairings}, the map
  $\hhat{\Tr}$ satisfies
  \begin{align*}
    (\hhat{\Tr}(\dual{a} \oo \dual{b}))(\upsilon \oo \omega) =
    (\dual{a} \oohr \dual{b})(\hlT(\upsilon \oo \omega)) =
    (\upsilon \ool \omega)(\Tr(a \oo b))
  \end{align*}
  for all $a,b\in A$ and $\upsilon,\omega \in \hat{A}$. Therefore, the
  isomorphism $a\mapsto \dual{a}$ identifies $\hhat{\Tr}$ with $\Tr$,
  and similar arguments show that it identifies the maps $\hhat{\Tl},
  {_{\lambda}\hhat{T}},{_{\rho}\hhat{T}}$ with $\Tl,\lT,\rT$, respectively. 

 The left integral $\hhat{\phi}$ is given by
  \begin{align*}
    \hhat{\phi}(\dual{a}) = \hhat{\phi}((S(a) \cdot \phi)\cdot \hat
    \psi) &= \hat\eps((\sigma^{\hat\psi})^{-1}(S(a) \cdot \phi)) \\ &=
    \hat \eps(S^{-1}(a) \cdot (\phi \circ S)) =
    (\phi \circ S)(S^{-1}(a)) =\phi(a)
  \end{align*}
  for all $a\in A$, where we used the formula for $\sigma^{\hat\psi}$
obtained in the proof of Theorem \ref{theorem:dual-measured}. A similar argument shows
  that $\hhat{\psi}(\dual{a})=\psi(a)$.

  Finally, assume that $(\mathcal{A},\mu,\phi,\psi)$ is a measured
  multiplier Hopf $*$-algebroid and let $b\in A$, $\upsilon \in \hat
  A$.  Then $(\dual{b})^{*}(\upsilon) = \dual{b}(\hat
  S(\upsilon)^{*})^{*}$ and $\hat S(\upsilon)^{*} = \ast \circ
  (\upsilon \circ S) \circ \ast \circ S = \ast \circ \upsilon \circ
  \ast$ because $S \ast \circ \ast S = \ast$ by
  \eqref{eq:involutions}, whence $(\dual{b})^{*}(\upsilon) =
  \upsilon(b^{*})$. Therefore, the map $A \to \hat{\hat{A}}$ given by
  $b\mapsto \dual{b}$ is a $*$-isomorphism.
\end{proof}

\subsection{The duality on the level of multipliers}
\label{subsection:multipliers}
Let $(\mathcal{A},\mu,\phi,\psi)$ be a measured regular multiplier
Hopf algebroid as before.  Then the canonical maps of
$\hat{\mathcal{A}}$ dualise the canonical maps of $\mathcal{A}$, and
intuitively, the multiplication and the comultiplication of $\hat
{\mathcal{A}}$ dualise the comultiplication and the multiplication of
$\mathcal{A}$, respectively. We now make this idea precise and study
the duality on the level of multiplier algebras.  

First, recall that the multiplier algebra $M(A)$ embeds into $\hdmAm$
by Lemma \ref{lemma:dual-evaluation}. By biduality, $M(\hat A)$ embeds
into $\dmAm$. In the case of a multiplier Hopf algebra $(H,\Delta)$,
Kustermans showed in \cite{kustermans:analytic-2} that $M(\hat H)$
identifies with the subspace of all functionals $\upsilon \in
\dual{H}$ satisfying $(\id \oo \upsilon)(\Delta(a))\in H$ and $
(\upsilon \oo \id)(\Delta(a)) \in H$ for all $a\in H$.
We extend this result to the present context,  the following result being
the first step.
\begin{lemma} \label{lemma:dual-multiplier-characterization}
  Let $\upsilon \in \dmAm$. 
  \begin{enumerate}
  \item $\upsilon (a \cdot \phi) =  (\lambda(\upsilon \circ
    S^{-1})(a)) \cdot \phi=S(\rho(\upsilon)(S^{-1}(a))) \cdot \phi$ for all $a\in A$.
  \item  $\rho(\upsilon)(a)\in A$ and $\lambda(\upsilon)(a) \in A$ for
    all $a \in A$ if and only if $\upsilon \omega\in \hat A$ and $\omega \upsilon \in \hat A$ for all $\omega
    \in \hat A$.
  \end{enumerate}
\end{lemma}
\begin{proof}
  (1) Let $a,b \in A$.  By Corollary
  \ref{corollary:dual-strong-invariance} and Proposition
  \ref{proposition:convolution} (4),
  \begin{align*}
    (\upsilon (a \cdot \phi))(b) = \upsilon(\rho(a \cdot \phi)(b)) &=
    \upsilon(S^{-1}(\rho(\phi \cdot b)(a)))  \\ &= (
    \phi \cdot b)(\lambda(\upsilon \circ S^{-1})(a))
\\ & = (\lambda(\upsilon \circ S^{-1})(a) \cdot \phi)(b)
= (S(\rho(\upsilon)(S^{-1}(a))) \cdot \phi)(b).
  \end{align*}

  (2) Assertion (1) implies that $\upsilon (a \cdot \phi)$
  belongs to $\hat A$ if and only if $\rho(\upsilon)(S^{-1}(a))$ lies in
  $A$. A similar argument completes the proof.
\end{proof}
The next result will not be used but may prove useful elsewhere.
\begin{lemma}
  Let $\upsilon \in \dmAm$ such that $\rho(\upsilon)(a) \in A$ and
  $\lambda(\upsilon)(a) \in A$ for all $a\in A$. Then $\eps \circ
  \rho(\upsilon) =\upsilon = \eps \circ \lambda(\upsilon)$.
\end{lemma}
\begin{proof}
  We only prove the first equation.  Equation
  \eqref{eq:base-counit-mult}, the counit properties \eqref{eq:right-counit} and Proposition
  \ref{proposition:convolution} (5) imply
\begin{align*}
  \eps(\rho(\upsilon)(\epsc(b)a)) &=
  \eps(\epsc(b)\rho(\upsilon)(a)) \\ 
&= \eps(b\rho(\upsilon)(a)) \\
&= (\eps \oo \upsilon)((b \oo 1)\Delta_{C}(a)) \\
&= \upsilon((S_{B}^{-1}\circ \epsc \oo \id)((b \oo 1)\Delta_{C}(a)))
= \upsilon(\epsc(b)a)
\end{align*}
for all $a,b\in A$. Since $\epsc(A)A=A$, we can conclude $\eps \circ
\rho(\upsilon) = \upsilon$.
\end{proof}
We can now identify the multiplier algebra $M(\hat A)$  with a
subspace of $\dmAm$ as follows.
\begin{theorem} \label{theorem:dual-multipliers} Let
  $(\mathcal{A},\mu,\phi,\psi)$ be a measured regular multiplier Hopf
  algebroid with dual algebra $\hat A$. Denote by $\tilde A \subseteq
  \dmAm$ the subspace of all $\upsilon$ satisfying that
  $\rho(\upsilon)(a) \in A$ and $\lambda(\upsilon)(a) \in A$ for all $a \in
  A$. Then each $\upsilon \in \tilde A$ defines a multiplier of $\hat
  A$ such that $\upsilon \omega = \upsilon \circ \rho(\omega)$ and
  $\omega \upsilon = \upsilon \circ \lambda(\omega)$ for all $\omega
  \in \hat A$, and the resulting map $\tilde A \to
  M(\hat A)$ is a linear isomorphism. With respect to this
  isomorphism, $\eps$ corresponds to the unit $1_{M(\hat A)}$, and the
  elements $x\in B=\hat C$ and $y\in C=\hat B$ correspond to the
  functionals $\eps\cdot x$ and $y\cdot \eps$, respectively.
\end{theorem}
\begin{proof}
  The map $\tilde A \to M(\hat A)$ is well-defined and injective by
  Lemma \ref{lemma:dual-multiplier-characterization} and Proposition
  \ref{proposition:dual-algebra}. We only need to show that this map
  is surjective. So, let $T \in M(\hat A)$ and define $\upsilon \in
  \dual{A}$ by
  \begin{align*}
   \upsilon(a) :=  \hat \psi(T(S(a) \cdot \phi)),
  \end{align*}
  where $\hat \psi$ denotes the dual right integral defined in Theorem
  \ref{theorem:dual-measured}.  We claim that $\upsilon$ belongs to
  $\dmAm$. Indeed, using Proposition
  \ref{proposition:dual-quantum-graphs} and the relation $y \cdot \phi
  = \phi \cdot S^{-2}(y)$ from Theorem
  \ref{theorem:modular-automorphism}, we find that for all $x\in B$,
  \begin{align*}
    \upsilon(xa) = \hat\psi(T(S(a)S(x) \cdot \phi)) &= \hat\psi(T(S(a)
    \cdot \phi \cdot S^{-1}(x))) \\ & = \hat\psi(Tx(S(a) \cdot \phi)) =
     \mu_{B}(x {_{\hat
      C}\hat\psi}((S(a) \cdot \phi)\sigma^{\hat\psi}(T))), \\
  \upsilon(ax) =\hat\psi(T(S(x)S(a) \cdot \phi)) &=
  \hat\psi(TS(x)(S(a) \cdot \phi)) \\ &= \mu_{C}(S(x){_{\hat
      B}\hat\psi}((S(a) \cdot \phi)\sigma^{\hat\psi}(T))) 
\\ &= \mu_{B}(S^{-1}({_{\hat
      B}\hat\psi}((S(a) \cdot \phi)\sigma^{\hat\psi}(T))) x).
  \end{align*}
 Similar calculations show that $\upsilon(ya)$ and $\upsilon(ay)$
  can be written in the form $\mu_{C}(y\cupsilon(a))$ and
  $\upsilon(ay)=\mu_{C}(\upsilonc(a)y)$ with suitable maps
  $\cupsilon,\upsilonc$.  

We finally show that $T=\upsilon$.  Let $a\in A$. Then equation
  \eqref{eq:dual-right-integral} and an application of  Lemma
  \ref{lemma:dual-multiplier-characterization} (1)  to
  $\omega$  show that
  \begin{align*}
    (T\omega)(a) &= \hat\psi(T\omega(S(a) \cdot \phi)) =
    \psi\hat(T(S(\rho(\omega)(a)) \cdot \phi)) =
    \upsilon(\rho(\omega)(a)) = (\upsilon\omega)(a). 
\end{align*}

In the special case where $T=x \in B=\hat C$ or $T=y\in C=\hat C$,
the assertion follows easily from the fact that $\eps$, regarded as an
element of $M(\hat A)$, acts as the identity. Alternatively, in the
first case, the
corresponding functional $\upsilon$ is given in by
\begin{align*}
  \upsilon(a) = \hat \psi(x(S(a)\cdot \phi)) = \hat\psi((S(a) \cdot
  \phi \cdot S^{-1}(x))) = \eps(S(a)S(x)) = \eps(xa),
\end{align*}
 where we used Example \ref{example:base-counit}, and a similar
 calculation applies to the second case.
\end{proof}

The dual algebra $\hat A = A\cdot \phi = \phi \cdot A$ evidently is a
bimodule over $M(A)$, and $\hat{\hat A}$ is a bimodule over $M(\hat
A)$.  The corresponding $M(\hat A)$-bimodule structure on
$A$, which is obtained via the identification $A \to \hat{\hat A}$,
$a\mapsto \dual{a}$, in Lemma \ref{lemma:dual-evaluation}, takes the
following form.
\begin{lemma} \label{lemma:dual-bimodule-explicit}
Let $a \in A$ and $\upsilon \in \tilde A \cong M(\hat A)$. Then
  $\upsilon \cdot \dual{a} = \dual{(\rho(\upsilon)(a))}$ and $\dual{a}
  \cdot \upsilon = \dual{(\lambda(\upsilon)(a))}$.
\end{lemma}
\begin{proof}
  We only prove the first equation. For all $\omega \in \hat A$,
  \begin{align*}
    (\upsilon \cdot \dual{a})(\omega) &= \dual{a}(\omega\upsilon) =
    (\omega\upsilon)(a) = (\omega \circ \rho(\upsilon))(a) =
    \dual{(\rho(\upsilon)(a))}(\omega). \qedhere
  \end{align*}
\end{proof}

The multiplication on $\hat A$ dualizes the comultiplication on $A$,
and the same is true when we pass to $M(A)$. To make this assertion
precise, we first extend the comultiplications of $\mathcal{A}$ to
multipliers.

Let us write $(\AlA)_{(A\oo A)}$ and $_{(A \oo A)}(\ArA)$ when we
regard $\AlA$ and $\ArA$ as a right or as a left module over $A
\otimes A$ via right or left multiplication, respectively.  
\begin{lemma} \label{lemma:comult-extension}
  The maps $\Delta_{B}$ and $\Delta_{C}$ extend uniquely to
  homomorphisms
  \begin{align*}
    \Delta_{B} &\colon L(A) \to \End((\AlA)_{(A \oo A)}), & \Delta_{C}
    &\colon R(A) \to \End({_{(A \oo A)}(\ArA)})
  \end{align*}
  such that the following diagrams commute for each $a \in L(A)$ and
  $b \in R(A)$:
  \begin{align*}
    \xymatrix@R=5pt{ \ACA \ar[r]^{\Tl} \ar[d]_{\id \otimes a} & \AlA
      \ar[d]^{\Delta_{B}(a)} & \AbA \ar[l]_{\Tr} \ar[d]^{a\oo\id} &&
      \AcA \ar[r]^{\lT} \ar[d]_{\id \otimes b} & \ArA
      \ar[d]^{\Delta_{C}(b)} & \ABA \ar[l]_{\rT} \ar[d]^{b\oo \id} \\
      \ACA \ar[r]_{\Tl} & \AlA & \AbA \ar[l]^{\Tr} && \AcA
      \ar[r]_{\lT} & \ArA & \ABA \ar[l]^{\rT} }
  \end{align*}
\end{lemma}
\begin{proof}
 Straightforward.
\end{proof}
Note that the algebras $\End((\AlA)_{(A \oo A)})$ and $\End({_{(A \oo
    A)}(\ArA)})$ naturally embed into $L((\AlA)_{(A \oo A)})$ and
$R({_{(A \oo A)}(\ArA)})$, respectively.
\begin{lemma}  \label{lemma:dual-embeddings-mult}
  There exist embeddings
  \begin{align*}
L((\AlA)_{(A \oo A)}) &\to    \dual{(\hAcA)}, &
R({_{(A\oo A)}(\ArA)}) &\to \dual{(\hAbA)}
  \end{align*}
given by
  \begin{align*}
e((a
    \cdot \phi) \oo (b \cdot \phi))&:= (\phi \ool
   \phi)(e(a \oo
    b)),    &  f((\phi
    \cdot a) \oo (\phi \cdot b))&:= (\phi \oor
    \phi)((a \oo
    b)f)
  \end{align*}
for all $a,b\in A$, $e \in L((\AlA)_{A\oo A})$, $f
  \in R({_{(A \oo A)}(\ArA)})$. 
\end{lemma}
\begin{proof}
  Similar arguments as in the case of Lemma \ref{lemma:dual-pairings}
  show that the maps above are well-defined.  Assume that $e \in
  L((\AlA)_{(A \oo A)})$ and $e(\upsilon \oo \omega)=0$ for all
  $\upsilon,\omega \in \hat A$. Then $e(a\cdot \upsilon \oo b\cdot
  \omega) = (\upsilon \ool \omega)(e(a\oo b)) = 0$ for all
  $\upsilon,\omega\in \hat A$ and $a,b\in A$, and by the dual version
  of Lemma \ref{lemma:dual-embeddings}, we can conclude that $e(a \oo
  b) = 0$ for all $a,b\in A$, whence $e=0$.  A similar argument shows
  that the second map is injective.
\end{proof}

We now can evaluate the images $\Delta_{B}(c)$ and $\Delta_{C}(c)$ of
a multiplier $c\in M(A)$ using the embeddings above, and make precise
the idea that the comultiplication dualises the multiplication.
\begin{proposition} \label{proposition:mult-comult}
  Let $(\mathcal{A},\mu,\phi,\psi)$ be a measured regular multiplier
  Hopf algebroid with dual algebra $\hat A$. Then
  for all $\upsilon,\omega \in
\hat A$ and $c \in M(A)$,
  \begin{align*}
    (\Delta_{B}(c))(\upsilon \oo \omega) = (\upsilon \omega)(c) =
    (\Delta_{C}(c))(\upsilon \oo \omega).
  \end{align*}
\end{proposition}
\begin{proof}
  Write $\upsilon =a \cdot \phi$, $\omega = b\cdot \psi$ and $a \oo b
  = \sum_{i} \Delta_{B}(d_{i})(e_{i} \oo 1)$ with $a,b,d_{i},e_{i} \in
  A$. Then $\upsilon \omega = \sum_{i} d_{i}\cphic(e_{i}) \cdot \phi$
  by \eqref{eq:dual-product} and
\begin{align*}
(\Delta_{B}(c)) (\upsilon \oo \omega) &=(\phi \ool
  \phi)(\Delta_{B}(c)(a \oo b)) \\ & = \sum_{i} (\phi \ool
  \phi)(\Delta_{B}(cd_{i})(e_{i} \oo 1))
  \\ &= \phi(\cphic(cd_{i})e_{i}) = \phi(cd_{i}\cphic(e_{i})) =
  (\upsilon\omega)(c). \qedhere
\end{align*}
\end{proof}

By construction, the canonical maps of $\hat{\mathcal{A}}$ dualize the
canonical maps of $\mathcal{A}$.  We next show that the same is true
on the level of multipliers. To make this assertion precise, we first
 extend the canonical maps
$\Tl,\Tr,\lT,\rT$ to maps
\begin{align*}
\Tl &\colon L(A) \oomC L(A) \to L((\AlA)_{(A\oo A)}), & 
\Tr &\colon L(A) \oomb L(A) \to L((\AlA)_{(A\oo A)}), \\
\lT &\colon R(A) \oomc R(A) \to R({_{(A \oo A)}(\ArA)}), &
\rT &\colon R(A) \oomB R(A) \to R({_{(A \oo A)}(\ArA)})
\end{align*}
using  formulas
\eqref{eq:left-galois-definition} and
\eqref{eq:right-galois-definition} and Lemma \ref{lemma:comult-extension}.
Next, we  identify $M(\hat A)$ with a subspace of
$\dmAm$  as in Theorem \ref{theorem:dual-multipliers} and obtain maps
\begin{align*}
  M(\hat A) \oomhb M(\hat A) &\to \dual{(\ArA)}, \quad \upsilon \oo
  \omega \mapsto \upsilon \oor \omega, &
  \\
  M(\hat A) \oomhc M(\hat A) &\to \dual{(\AlA)}, \quad \upsilon \oo
  \omega \mapsto \upsilon \ool \omega;
\end{align*}
 see also \eqref{eq:base-oo-b} and \eqref{eq:base-oo-c}.
\begin{proposition} \label{proposition:dual-galois-multipliers}
  Let $(\mathcal{A},\mu,\phi,\psi)$ be a measured regular multiplier
  Hopf algebroid. Then the following diagrams commute:
  \begin{align*}
    \xymatrix@R=15pt@C=15pt{
      M(\hat A) \oomhb M(\hat A) \ar[r]^{\hTr} \ar[d] & L((\hAlA)_{\hat A\oo
        \hat A}) \ar@{^(->}[d] &     M(\hat A) \oomhc M(\hat A) \ar[r]^{\hlT} \ar[d] & R(_{(\hat A\oo
        \hat A)}(\hArA)) \ar@{^(->}[d] \\
    \dual{(\ArA)} \ar[r]_{\dual{(\lT)}} & \dual{(\AcA)} & 
      \dual{(\AlA)} \ar[r]_{\dual{(\Tr)}} & \dual{(\AbA)} 
    } 
  \end{align*}
\end{proposition}
\begin{proof}
  Let $\upsilon,\omega \in M(\hat A)$ and $a,b\in A$. Then by
  definition, Proposition \ref{proposition:mult-comult} and Lemma \ref{lemma:dual-bimodule-explicit},
  \begin{gather} \label{eq:dual-galois-multipliers}
    \begin{aligned}
      (\hat\Delta_{\hat B}(\upsilon)(1 \oo \omega))(\dual{a} \oo
      \dual{b}) &=
      (\hat \Delta_{\hat B}(\upsilon))(\dual{a} \oo \omega \cdot \dual{b}) \\
      &= \upsilon(a\rho(\omega)(b)) = (\upsilon \ool \omega)((a \oo
      1)\Delta_{C}(b)).
    \end{aligned}
  \end{gather}
  Therefore, the first diagram commutes. A similar argument applies to
  the second diagram.
\end{proof}

\section{Examples}
\label{section:examples}

We now consider two examples of regular multiplier Hopf algebroids
presented in \cite{timmermann:regular}, the function algebra of an
\'etale, locally compact, Hausdorff groupoid and a two-sided crossed
product which appeared already in \cite{connes:rankin}. Given a
quasi-invariant measure on the unit space of the groupoid or invariant
functionals for the actions underlying the crossed product, we
construct base weights and integrals for these examples and obtain
measured regular multiplier Hopf algebroids. Finally, we give a
detailed description of the respective dual objects, which in the
groupoid case turns out to be the groupoid algebra of compactly
supported functions equipped with the convolution product.

\subsection{\'Etale groupoids}

% For every \'etale, locally compact, Hausdorff groupoid, the associated
% algebra of compactly supported functions carries the structure of a
% multiplier Hopf $*$-algebroid \cite{timmermann:regular}. We now show that
% each continuously quasi-invariant Radon measure on the unit space 
% yields a base weight and integrals so that one obtains a measured
% multiplier Hopf $*$-algebroid. The dual turns out to be the groupoid
% algebra of compactly supported functions equipped with the convolution
% product. 

We assume basic terminology for topological groupoids, see
 \cite{paterson} or \cite{renault},
and recall that a locally compact, Hausdorff groupoid is
\emph{\'etale} if its source and the target maps locally are 
homeomorphisms.   Given a locally
compact, Hausdorff space $X$, we denote by $C(X)$ the $*$-algebra of
all continuous functions on $X$ and by $C_{c}(X) \subseteq C(X)$ the
subalgebra of all functions with compact support, and identify $C(X)$
with the multiplier algebra $M(C_{c}(X))$ in the natural way. Given a
continuous map $f\colon X\to Y$ of locally compact, Hausdorff spaces,
we denote by $f^{*} \colon C(Y) \to C(X)$ the pull-back.
Given spaces $X,Y,Z$ and maps $f\colon X\to Z$, $g\colon Y \to Z$, we
denote by $X {_{f}\times_{g}} Y:=\{ (x,y) \in X\times Y: f(x)=g(y)\}$
the pull-back.
\begin{proposition} \label{proposition:groupoid-hopf} Let $G$ be an
  \'etale, locally compact, second countable, Hausdorff groupoid with
  unit space $X$, target and source maps $t,s \colon G \to X$,
  multiplication map $m \colon \GstG \to G$ and inversion map $j\colon
  G\to G$, and let
    \begin{align*}
      A&:=C_{c}(G), & B&:=s^{*}(C_{c}(X)), & C &:= t^{*}(C_{c}(X)), &
      S_{B} &:= j^{*}|_{B}, & S_{C} &:= j^{*}|_{C}.
    \end{align*}
    Then $\mathcal{A}_{B}:=(B,A,\id,S_{B})$ and
    $\mathcal{A}_{C}:=(C,A,\id,S_{C})$ are compatible left and right
    quantum graphs.  The algebras $\Lreg(\AltkA)$ and $\Rreg(\ArtkA)$
    can be identified with a subalgebra of $C(\GstG)$ that contains
    the image of the map 
    \begin{align*}
      \Delta_{B}:=\Delta_{C}:=m^{*} \colon
    C_{c}(G) \to C(\GstG),
    \end{align*}
    and
    $\mathcal{A}:=((\mathcal{A}_{B},\Delta_{B}),(\mathcal{A}_{C},\Delta_{C}))$
    is a projective multiplier Hopf $*$-algebroid. Its antipode and
    left and right counit are given by
    \begin{align*}
      S(f) &= j^{*}(f), &   \beps(f) &= s^{*}(f|_{X}), & \epsc(f) &=
      t^{*}(f|_{X}) &&\text{for all } f\in C_{c}(G).
    \end{align*}  
\end{proposition}
\begin{proof}
  This is a summary of Section 8 in \cite{timmermann:regular}, apart
  from the projectivity assertion. We show that $C_{c}(G)$ is projective
  as a module over $s^{*}(C_{c}(X))$.

  Using the assumptions on $G$, we find a locally finite partition of
  unity $(\chi_{i})_{i}$ on $G$ and a family of functions
  $(\theta_{i})_{i}$ in $C_{c}(G)$ such that for each $i$, the
  function $\theta_{i}$ equals $1$ on $\supp \chi_{i}$ and the map $s$
  restricts to a local homeomorphism $s_{i}$ on some neighbourhood of
  $\supp \theta_{i}$. Then the maps
  \begin{align*}
    j\colon C_{c}(G) \to \prod_{i} C_{c}(X), \ f \mapsto
    ((s^{-1}_{i})^{*}(\chi_{i} f))_{i}
  \end{align*}
  and
  \begin{align*}
    p \colon \prod_{i} C_{c}(X) \to C_{c}(G), \ (g_{i})_{i}
    \mapsto \sum_{i} \theta_{i}s_{i}^{*}(g_{i})
  \end{align*}
  satisfy $p \circ j = \id_{C_{c}(G)}$ and are morphisms of $C_{c}(X)
  \cong s^{*}(C_{c}(X))$-modules. Therefore, $C_{c}(G)$ embeds as a
  direct summand into a free $C_{c}(X)$-module.
\end{proof}

Let $G$ be an \'etale, locally compact, second countable, Hausdorff
groupoid as above.  Then a Radon measure $\mu_{X}$ on $X$ is
\emph{quasi-invariant} if the Radon measures $\nu$ and
$\nu^{-1}$ on $X$ defined by
\begin{align} \label{eq:groupoid-phi}
  \phi(f) &= \int_{G} f \intd\nu := \int_{X} \sum_{\gamma \in
    t^{-1}(x)} f(\gamma) \intd\mu_{X}(x),  \\ \label{eq:groupoid-psi}
  \psi(f) &= \int_{G} f \intd\nu^{-1} := \int_{X} \sum_{\gamma \in
    s^{-1}(x)} f(\gamma) \intd\mu_{X}(x)
\end{align}
are equivalent. We call $\mu_{X}$ \emph{continuously quasi-invariant}
if the Radon-Nikodym derivative $\delta:=\intd\nu^{-1}/\intd\nu$ and
its inverse are continuous. Given a function $f\in C(G)$ such that $\supp f
\cap X$ is compact, we define
\begin{align} \label{eq:groupoid-eps}
  \tilde \eps(f) &:= \int_{X} f(x) \intd\mu_{X}(x).
\end{align}
\begin{theorem} \label{theorem:groupoid-measured} Let $G$ be an
  \'etale, locally compact, second countable, Hausdorff groupoid with
  unit space $X$, target and source maps $t,s \colon G \to X$,
  multiplication $m \colon \GstG \to G$, inversion $j\colon G\to G$,
  and let $\mu_{X}$ be a Radon measure on $X$ which is continuously
  quasi-invariant and has full support.  Define 
  $\phi,\psi,\tilde\eps$ as in \eqref{eq:groupoid-phi},
  \eqref{eq:groupoid-psi} and \eqref{eq:groupoid-eps}, respectively. Then the multiplier Hopf
  $*$-algebroid $\mathcal{A}$ in Proposition
  \ref{proposition:groupoid-hopf} and the functionals $\mu_{B}:=\tilde
  \eps|_{B}$, $\mu_{C}:=\tilde \eps|_{C}$, $\phi$ and $\psi$ form a
  measured multiplier Hopf $*$-algebroid.
\end{theorem}
\begin{proof}
  The functionals $\mu_{B}, \mu_{C}, \phi$ and $\psi$ are evidently
  positive, and faithful because $\mu_{X}$ has full support by
  assumption.  By construction, $\mu_{B} \circ S_{C} = \mu_{C}$,
  $\mu_{C} \circ S_{B} = \mu_{B}$, $\mu_{B} \circ \beps = \mu_{C}
  \circ \ceps$. One easily verifies that
  $\phi$ and $\psi$ can be written in the form $\phi = \mu_{C} \circ
  \cphic$ and $\psi = \mu_{B}\circ \bpsib$, where
  \begin{align*}
    (\cphic(f))(\gamma') &= \sum_{\substack{\gamma \in G\\ t(\gamma')=t(\gamma)}}
    f(\gamma), & (\bpsib(f))(\gamma') &= \sum_{\substack{\gamma \in G\\
      s(\gamma')=s(\gamma)}} f(\gamma),
  \end{align*}
  These relations and the fact that $\psi = \delta \cdot \phi$, where
  $\delta=\intd\nu^{-1}/\intd \nu$, imply that $\phi$ and $\psi$ lie
  in $\dmAm$.  Straightforward calculations show that $\cphic$ and
  $\bpsib$ are left- or right-invariant, respectively.
\end{proof}

Let us describe the dual measured Hopf $*$-algebroid.  Recall that the space
$C_{c}(G)$ is a $*$-algebra with respect to the convolution product
and involution given by
\begin{align*}
  (f \ast g)(\gamma) &=\sum_{\gamma'\gamma''=\gamma}
  f(\gamma')g(\gamma''), & f^{*}(\gamma) = \overline{f(\gamma^{-1})}
  \delta(\gamma)
\end{align*}
for all $f,g\in C_{c}(G)$ and $\gamma \in G$. We denote this
$*$-algebra by $\C_{c}(G)$. Note that we can regard $C_{c}(X)$ as a
subspace of $C_{c}(G)$ because $X \subseteq G$ is closed and open, $G$
being Hausdorff and \'etale, and that then $C_{c}(X)$ is a subalgebra
of $\C_{c}(G)$.

Observe furthermore that the subspace $G^{[2]}:=(\GssG) \cap (\GttG)$
of $G\times G$ is a groupoid with unit space $X$, target and source
maps $t\circ \pi_{1} = t\circ \pi_{2}$ and $s\circ \pi_{1} = s\circ
\pi_{2}$, respectively, where $\pi_{i} \colon G^{[2]} \to G$ denotes
the canonical projection, and component-wise multiplication.
\begin{theorem} \label{theorem:groupoid-dual} Let $G$ be an \'etale,
  locally compact, second countable, Hausdorff groupoid with a
  continuously quasi-invariant Radon measure $\mu_{X}$ of full
  support, and denote by $(\mathcal{A},\mu,\phi,\psi)$ the associated
  measured multiplier Hopf $*$-algebroid defined in Theorem
  \ref{theorem:groupoid-measured}.  Then the dual $*$-algebra $\hat A$
  can be identified with $\C_{c}(G)$, the dual left and right quantum
  graphs $(\hat B,\hat A,\iota_{\hat B},\hat S_{\hat B})$ and $(\hat
  C,\hat A,\iota_{\hat C},\hat S_{\hat C})$ coincide and are given by
  $\hat B=\hat C=C_{c}(X) \hookrightarrow \C_{c}(G)$ and $\hat S_{\hat
    B} = \hat S_{\hat C}=\id$, and the dual comultiplications $\hat
  \Delta_{\hat B}$ and $\hat\Delta_{\hat C}$ take values in the
  algebras $\hAltkA$ and $\hArtkA$, which can be identified with the
  space $\C_{c}(G^{[2]})$. These dual comultiplications and the dual
  counit, antipode, integrals and their modular automorphism are given
  by
  \begin{gather*}
    (\hat \Delta_{\hat B}(f))(\gamma,\gamma') = (\hat \Delta_{\hat
      C}(f))(\gamma,\gamma') = \delta_{\gamma,\gamma'} f(\gamma), \\
  \begin{aligned}
    \hat \eps(f) &= \phi(f), & \hat
    S(f)(\gamma)&= f(\gamma^{-1})\delta(\gamma), &
 & \hat \phi(f) = \hat \psi(f) =\eps(f), &
 (\sigma^{\hat \phi}(f))(\gamma) &= f(\gamma)\delta^{-1}(\gamma)
  \end{aligned}
\end{gather*}
for all $f\in C_{c}(G)$ and $\gamma,\gamma' \in G$, where
$\delta^{-1}=\intd\nu/\intd\nu^{-1}$.
\end{theorem}
\begin{proof}
By definition,  $\hat A \subseteq
  \dual{C_{c}(G)}$  is the space of all functionals of the form
  \begin{align*}
    \hat f:= f\cdot \phi \colon g \mapsto \phi(gf) = \int_{G} gf \intd
    \nu,
  \end{align*}
  where $f\in C_{c}(G)$. By \eqref{eq:dual-product}, the
  multiplication on $\hat A$ is given by 
  \begin{align*} 
    \hat f \ast \hat g = \hat h \quad \text{with} \quad h(\gamma) &= ((\cphic \oo \id)(\Tl^{-1}(f
    \otimes
    g)))(\gamma) \\
    &= \sum_{\substack{\gamma' \in G\\ t(\gamma')=t(\gamma)}}
    f(\gamma')g(\gamma'{}^{-1}\gamma) 
    = \sum_{\gamma'\gamma''=\gamma} f(\gamma')g(\gamma'')
  \end{align*}
  for all $f,g\in C_{c}(G)$. The involution on $\hat A$ is defined by
  \begin{align*}
    (\hat f)^{*}(g) &= \left(\int_{G} fS(g)^{*} \intd \nu\right)^{*} =
    \int_{G} \overline{f(\gamma)} g(\gamma^{-1}) \intd\nu(\gamma) =
    \int_{G} \overline{f(\gamma^{-1})}g(\gamma) \delta(\gamma)
    \intd\nu(\gamma).
  \end{align*}

  The embedding $\hat B = C \to M(\hat A)$ is given by
  \begin{align*}
    t^{*}(h) \hat f &= t^{*}(h)f\cdot \phi = \widehat{t^{*}(h)f} =
    \hat h \ast \hat f, &
    \hat f t^{*}(h) &= S^{-1}(t^{*}(h)) f\cdot \phi =
    \widehat{s^{*}(h)f} = \hat f \ast \hat h
   \end{align*}
   for all $f\in C_{c}(G)$ and  $h\in C_{c}(X)$. We can therefore
   identify $\hat B$ with $\widehat{C_{c}(X)}$. A similar caclulation
   shows that the same is true for $\hat C$ and that then $\hat
   S_{\hat B}$ and $\hat S_{\hat C}$ reduce to the identity.

   There exist natural isomorphisms $\hat A \oohl \hat A \cong
   C_{c}(\GttG)$ and $\hat A \oohr \hat A \cong C_{c}(\GssG)$, which
   map a tensor $\hat f \otimes \hat g$ to the function
   $(\gamma,\gamma') \mapsto f(\gamma) g(\gamma')$, and these maps
   induce isomorphisms $\hAltkA \cong \C_{c}(G^{[2]}) \cong \hArtkA$.
   To compute the dual comultiplication, let $f\in C_{c}(G)$ and
   choose $h \in C_{c}(X)$ such that $f s^{*}(h) = f$. Then for all
   $g,g' \in C_{c}(G)$,
   \begin{align*}
     (\hat \Delta_{\hat B}(\hat f)(1 \otimes \hat h))(g \otimes g') &=
     (\hat f \oor \hat h)((g \otimes 1)\Delta_{C}(g')) \\ &= \int_{G}
     \sum_{\substack{\gamma' \in G\\ r(\gamma')=s(\gamma)}}
     f(\gamma)h(\gamma')    g(\gamma)g'(\gamma\gamma') \intd\nu(\gamma) \\
     &= \int_{G} f(\gamma)g(\gamma)g'(\gamma) \intd \nu(\gamma).
   \end{align*}
   The formula for $\hat\Delta_{\hat B}$ follows, and a similar
   calculation applies to $\hat \Delta_{\hat C}$.

Let us next prove the formula for the dual antipode $\hat S$. For all
$f,g\in C_{c}(G)$,
\begin{align*}
  (\hat S(\hat f))(g) = \phi(fS(g)) &= \int_{G} f(\gamma)g(\gamma^{-1})
  \intd\nu(\gamma) \\ &= \int_{G} f(\gamma^{-1})g(\gamma)
  \delta(\gamma)\intd \nu (\gamma) = \widehat{(j(f)\delta)}(g).
\end{align*}
The formulas for $\hat \eps,\hat\phi,\hat\psi$ follow more or less
directly from the definitions. The formula for $\sigma^{\hat\psi}$ can
be deduced from Theorem \ref{theorem:dual-measured}. More directly,
the definitions imply that
\begin{align*}
  \hat\phi(f\ast g) &= \int_{X}\sum_{\gamma\gamma'=x}
  f(\gamma')g(\gamma'') \intd\mu_{X}(x) \\
  &= \int_{G} f(\gamma^{-1}) g(\gamma) \intd\nu^{-1}(\gamma) \\
  &= \int_{G} g(\gamma) (f\delta^{-1})(\gamma^{-1}) \intd\nu(\gamma)
  \\
  &=\int_{X}\sum_{\gamma\gamma'=x}
  g(\gamma)(f\delta^{-1})(\gamma^{-1}) \intd\mu_{X}(x) = \hat\phi(g
  \ast (f\delta^{-1}))
\end{align*}
 for all $f,g\in \C_{c}(G)$, where $f\delta^{-1}$ denotes the
 pointwise product.
\end{proof}

\subsection{A two-sided crossed product}
\label{subsection:two-sided}

As a second example, we consider a two-sided crossed product that
arises as a two-sided crossed product $C \rtimes H \ltimes B$ for
compatible actions of a multiplier Hopf algebra $H$ on two
anti-isomorphic algebras $B,C$. This example appeared already in the
work of Connes and Moscovici on transverse geometry in
\cite{connes:rankin} and was put into the framework of multiplier Hopf
algebroids in \cite{timmermann:regular}. We show that invariant
functionals on $B$ and $C$ and integrals on $H$ yield integrals on $C
\rtimes H \ltimes B$, and describe the dual multiplier Hopf algebroid.

Recall from \cite{daele:actions} that a \emph{right action} of a
multiplier Hopf algebra $\mathcal{H}=(H,\Delta_{H})$ on an algebra $B$
is a non-degenerate right $H$-module structure on $B$, written $x
\actb h$, such that $(xx') \actb h = \sum (x \actb h_{(1)})(x' \actb
h_{(2)})$ for all $h\in H$ and $x,x'\in B$, where the right hand side
makes sense due to the assumption that $B \actb H = B$.  Left actions
of $\mathcal{H}$ on an algebra $C$ are defined accordingly and written
in the form $ h \actc y$, where $y\in C$ and $h\in H$.

We define a \emph{coupled pair of algebras}, briefly denoted $(S_{B}
\colon B \rightleftarrows C : S_{C})$, to be two non-degenerate,
idempotent algebras $B$ and $C$ with anti-isomorphisms $S_{B} \colon
B\to C$ and $S_{C} \colon C\to B$.  An \emph{action} of a multiplier
Hopf algebra $\mathcal{H}$ as above on a coupled pair of
algebras $(S_{B} \colon B \rightleftarrows C : S_{C})$ is a right
action of $\mathcal{H}$ on $B$ and a left action of $\mathcal{H}$ on
$C$ such that
\begin{align} \label{eq:coupled-action}
  S_{B}(x \actb h) &= S_{H}(h) \actc S_{B}(x), &
  S_{C}(h \actc y) &= S_{C}(y) \actb S_{H}(h)
\end{align}
for all $x\in B$, $y\in C$ and $h\in H$.
Given such an action,  the vector space $C \otimes H \otimes B$
becomes a non-degenerate, idempotent algebra with respect to the
multiplication given by
\begin{align*}
  (y \otimes h \otimes x)(y' \otimes h' \otimes x') &= \sum y(h_{(1)}
  \actc y') \otimes h_{(2)}h'_{(1)} \otimes (x \actb
  h'_{(2)})x' 
\end{align*}
for all $x,x'\in B$, $y,y'\in C$, $h,h'\in H$. We denote this algebra by $C
\rtimes H \ltimes B$ and call it the \emph{crossed product} associated
to the action. One easily verifies that this crossed product can be
characterized as  the
universal algebra $A$ with non-degenerate embeddings $B,C,H
\hookrightarrow M(A)$ such that $A=CHB \subseteq M(A)$ and
\begin{align*}
  xy &= yx, & xh &= \sum h_{(1)}(x \actb h_{(2)}), & hy &= \sum (h_{(1)}
  \actc y)h_{(2)}      
\end{align*}
for all $x\in B$, $y\in C$, $h \in H$. Note that the last two relations imply 
\begin{align} \label{eq:reverse-crossed}
  hx &= \sum (x \actb S_{H}^{-1}(h_{(2)})) h_{(1)}, & yh &=
\sum h_{(2)}(S_{H}^{-1}(h_{(1)}) \actc y).
\end{align}
 \begin{proposition} \label{proposition:coupled-hopf}
   Let $(S_{B} \colon B \rightleftarrows C \colon S_{C})$ be a coupled
   pair of algebras with an action of a regular multiplier Hopf algebra
   $\mathcal{H}=(H,\Delta_{H})$, and let $A=C \rtimes H \ltimes
   B$. Then $\mathcal{A}_{B}:=(B,A,\iota_{B},S_{B})$ and
   $\mathcal{A}_{C}:=(C,A,\iota_{C},S_{C})$ are compatible left and
   right quantum graphs, there exist a left and a right
   comultiplication $\Delta_{B}$ for $\mathcal{A}_{B}$ and
   $\Delta_{C}$ for $\mathcal{A}_{C}$ such that
   \begin{align*}
     \Delta_{B}(yhx)(a \otimes b) &= \sum yh_{(1)}a \otimes h_{(2)}xb,
     & (a \otimes b) \Delta_{C}(yhx) &= \sum ayh_{(1)} \otimes
     bh_{(2)}x
   \end{align*}
   for all $a,b\in A$, $x\in B$, $y\in C$, $h\in H$, and
   $\mathcal{A}=((\mathcal{A}_{B},\Delta_{B}),(\mathcal{A}_{C},\Delta_{C}))$
   is a projective regular multiplier Hopf algebroid with antipode and
   counits  given by
    \begin{gather} \label{eq:coupled-counits}
      \beps(yhx) = (x \actb S_{H}^{-1}(h))S_{B}^{-1}(y), \quad
      \epsc(yhx) = S_{C}^{-1}(x)(S_{H}^{-1}(h) \actc y), \\ \label{eq:coupled-antipode}
      S(yhx) =  S_{B}(x)S_{H}(h) S_{C}(y) 
    \end{gather}
    for all $x\in B$, $y\in C$, $h\in H$.
  \end{proposition}
  \begin{proof}
    Most parts of the statement summarise a construction in Section 8
    of \cite{timmermann:regular} up to a switch between left and right
    actions.  The formulas for the left and right counit follow from
    the relations $\beps(xyh) = xS_{B}^{-1}(y)\eps_{H}(h)$ and
    $\epsc(hxy) = S_{C}^{-1}(x)y\eps_{H}(h)$, respectively, which are
    contained in \cite{timmermann:regular}, and
    \eqref{eq:reverse-crossed}.  Note that the algebra $A$ is free as
    a $B$- or $C$-module and hence projective. 
  \end{proof}

We define  a \emph{base weight} for  a coupled pair of algebras $(S_{B}
  \colon B \rightleftarrows C \colon S_{C})$ to be a pair of faithful
  functionals $\mu=(\mu_{B},\mu_{C})$, where $\mu_{B} \in \dB$ and
  $\mu_{C} \in \dC$, such that $\mu_{C} \circ S_{B} =\mu_{B}$,
  $\mu_{B} \circ S_{C} =\mu_{C}$ and
  $\sigma^{\mu}_{B}:=S_{B}^{-1}\circ S_{C}^{-1}$ and
  $\sigma^{\mu}_{C}:=S_{B}\circ S_{C}$ are modular automorphisms for
  $\mu_{B}$ and $\mu_{C}$, respectively.  We call such a base weight
  \emph{invariant} with respect to an action of a multiplier
  Hopf algebra $\mathcal{H}=(H,\Delta_{H})$ if
  \begin{align} \label{eq:coupled-invariance}
    \mu_{B}(h \actc x) &= \eps_{H}(h)\mu_{B}(x), &
    \mu_{C}(y \actb h) &= \mu_{C}(y) \eps_{H}(y)
  \end{align}
for all $x\in B$, $y\in C$, $h\in H$. Note that this relation implies
  \begin{align} \label{eq:coupled-invariance-adjoint}
    \begin{aligned}
      \mu_{C}((h \actc y)y') &= \sum \mu_{C}((h_{(1)}\actc
      y)(h_{(2)}S_{H}(h_{(3)}) \actc y')) \\
      &= \sum \mu_{C}(h_{(1)} \actc(y(S_{H}(h_{(2)}) \actc y'))) =
      \mu_{C}(y(S_{H}(h) \actc y'))
    \end{aligned}
  \end{align} for all $y,y'\in C$, and
  likewise $\mu_{B}(x(x' \actb h)) = \mu_{B}((x \actb S_{H}(h))x')$
  for all $x,x' \in B$.

  \begin{theorem} \label{theorem:coupled-measured} Let $(S_{B} \colon
    B \rightleftarrows C \colon S_{C})$ be a coupled pair of algebras
    with an action of a regular multiplier Hopf algebra
    $\mathcal{H}=(H,\Delta_{H})$ and an invariant base weight
    $\mu=(\mu_{B},\mu_{C})$, and let $\phi_{H}$ be a left and
    $\psi_{H}$ a right integral on $\mathcal{H}$.  Then the regular
    multiplier Hopf algebroid $\mathcal{A}$ defined in Proposition
    \ref{proposition:coupled-hopf} and the functionals
    $\mu_{B},\mu_{C}$ and
    \begin{align*}
      \phi \colon yhx &\mapsto\mu_{C}(y) \phi_{H}(h)\mu_{B}(x), &
      \psi \colon yhx &\mapsto\mu_{C}(y) \psi_{H}(h)\mu_{B}(x)
    \end{align*}
    form a measured regular multiplier Hopf algebroid.  Denote by
    $\sigma^{\phi}_{H},\sigma^{\psi}_{H}$ and $\delta_{H}$ the modular
    automorphisms of $\phi_{H}$ and $\psi_{H}$ and the modular element
    of $\mathcal{H}$, respectively. Then the modular automorphisms of
    $\phi$ and $\psi$ are given by
    \begin{align*}
      \sigma^{\phi}(yhx) &=
      \sigma^{\mu}_{C}(y)\sigma^{\phi}_{H}(h)(\sigma^{\mu}_{B}(x)
      \actb \delta_{H}), & \sigma^{\psi}(yhx) &= (\delta_{H} \actc
      \sigma^{\mu}_{C}(y)) \sigma^{\psi}_{H}(h) \sigma^{\mu}_{B}(x),
    \end{align*}
    and $\psi=\delta\cdot \phi$, where $\delta\in M(H)$ denotes the
    unique multiple of $\delta_{H}$ satisfying $\delta \cdot
    \phi_{H}=\psi_{H}$.
\end{theorem}
\begin{proof}
  Throughout this proof, $x,x',y,y'$ and $h,h'$ will always denote
  arbitrary elements of $B,C$ or $H$, respectively.

  To see that $\mu$ is a base weight for $\mathcal{A}$, we use
  equations \eqref{eq:coupled-counits}, \eqref{eq:coupled-action},
  \eqref{eq:coupled-invariance-adjoint} and the modular automorphism
  for $\mu_{C}$, and find
  \begin{align} \label{eq:coupled-counit-functional}
    \begin{aligned}
      \mu_{B}(\beps(yhx)) &= \mu_{B}((x \actb
      S_{H}^{-1}(h))S_{B}^{-1}(y)) \\
      &=\mu_{C}(y(h \actc S_{B}(x))) \\ &= \mu_{C}((S_{H}^{-1}(h)
      \actc y)S_{B}(x)) \\ & = \mu_{C}(S_{C}^{-1}(x)(S_{H}^{-1}(h) \actc
      y)) = \mu_{C}(\epsc(yhx)).
    \end{aligned}
  \end{align}

Next, we show that the formula for $\phi$ defines a full and faithful
left integral.  First, we check that $\phi$ belongs to $\dmAm$.  By
\cite{daele:1}, $\phi_{H}$ has a modular automorphism
$\sigma^{\phi}_{H}$ and there exists an invertible multiplier
$\delta_{H} \in M(H)$ such that $\sum \phi_{H}(h_{(1)})h_{(2)} =
\delta_{H}\phi_{H}(h)$ for all $h \in H$. Using this relation, left
invariance of $\phi_{H}$ and \eqref{eq:coupled-invariance}, we find that
  \begin{align}
    \phi(x'yhx) &= \sum \phi(yh_{(1)}(x' \actb h_{(2)})x) \nonumber \\ &=
\sum    \mu_{C}(y)\phi_{H}(h_{(1)}) \mu_{B}((x' \actb h_{(2)})x) =
    \mu_{C}(y)\phi_{H}(h)\mu_{B}((x'\actb \delta_{H})x), \label{eq:coupled-phi-b} \\
    \phi(yhxy') &= \sum \phi(y(h_{(1)} \actc y')h_{(2)}x) \nonumber \\
    &= \sum
    \mu_{C}(y(h_{(1)}\actc y'))\phi_{H}(h_{(2)})\mu_{B}(x) =
    \mu_{C}(yy')\phi_{H}(h)\mu_{B}(x). \nonumber
  \end{align}
  These relations imply
  that $\phi \in \dmAm$ and
  \begin{gather}
    \cphi(yhx) = y\phi_{H}(h)\mu_{B}(x) = \phic(yhx),  \label{eq:coupled-cphic}\\
    \phib(x'yh) = \mu_{C}(y)\phi_{H}(h)(x' \actb \delta_{H}), \quad
    \bphi(yhx) = \mu_{C}(y)\phi_{H}(h)(x \actb
    \delta_{H}^{-1}), \label{eq:coupled-bphib}
  \end{gather}
where we used the relation
  $\delta_{H}^{-1}=S_{H}(\delta_{H})$ and the analogue of
  \eqref{eq:coupled-invariance-adjoint} for $\mu_{B}$. Since $\delta_{H}$ is invertible,
  $\phib$ and $\bphi$  are surjective. 
  
  The map $\cphic$ is left-invariant because   for all $a\in A$,
  \begin{align*}
    (\id \otimes \cphic)((a \otimes 1)\Delta_{C}(yhx)) &= \sum (\id
    \otimes \cphic)(ayh_{(1)} \otimes h_{(2)}x) \\ &=\sum ay h_{(1)}
    \phi_{H}(h_{(2)})\mu_{B}(x) \\ &= \sum a y\phi_{H}(h)\mu_{B}(x) =
    a\cphic(yhx).
  \end{align*}

  Thus, $\phi$ is a full left integral. By Theorem
  \ref{theorem:integrals-faithful}, it is  faithful.  

  Let us prove the formula for the modular automorphism
  $\sigma^{\phi}$.  By Theorem \ref{theorem:modular-automorphism},
  $\sigma^{\phi}(y) = \sigma^{\mu}_{C}(y)$ for all $y\in C$. Equation
  \eqref{eq:coupled-phi-b} shows that
  \begin{align*}
    \phi(yhx\sigma^{\mu}_{B}(x' \actb \delta_{H})) &=
    \mu_{C}\phi_{H}(h)\mu_{B}(x\sigma^{\mu}_{B}(x' \actb \delta_{H}))
    \\
    &= \mu_{C}\phi_{H}(h)\mu_{B}((x' \actb \delta_{H})x) = \phi(x'yhx),
  \end{align*}
  and hence $\sigma^{\phi}(x')=\sigma^{\mu}_{B}(x' \actb
  \delta_{H})$. Next,
  \begin{align*}
    \phi(h'yhx) &= \sum \phi((h'_{(1)} \actc y)h'_{(2)}hx) =
    \mu_{C}(y)\phi_{H}(h'h)\mu_{B}(x),
  \end{align*}
  and a similar calculation shows that 
  \begin{align*} 
    \phi(yhx\sigma^{\phi}_{H}(h')) &=
    \mu_{C}(y)\phi_{H}(h\sigma^{\phi}_{H}(h')) \mu_{B}(x) = 
    \mu_{C}(y)\phi_{H}(h'h) \mu_{B}(x).
  \end{align*}
  Therefore, $\sigma^{\phi}(h')=\sigma^{\phi}_{H}(h')$.  Replacing
  $\sigma^{\phi}_{H}(h')$ by $\delta \in M(H)$ above, we find that the
  relation $\delta \cdot \phi_{H} = \psi_{H}$ implies $\delta \cdot
  \phi = \psi$.

  Similar arguments show that $\psi$ is a full and faithful right
  integral and that its modular automorphism has the claimed form,
  where one uses the relations $\sum h_{(1)}\psi_{H}(h_{(2)}) =
  \delta_{H}^{-1}\psi_{H}(h)$ and
  \begin{align*}
    \psi(yhxy') &= \sum \psi(y(h_{(1)} \actc y')h_{(2)}x) \\ &= \sum
    \mu_{C}(y(h_{(1)}\actc y'))\psi_{H}(h_{(2)})\mu_{B}(x) =
    \mu_{C}(y(\delta_{H}^{-1} \actc y'))\psi_{H}(h)\mu_{B}(x).
    \qedhere
  \end{align*}
\end{proof}

We finally describe the dual measured regular multiplier Hopf
algebroid.  In the situation of Theorem
\ref{theorem:coupled-measured}, denote by $K^{\mu}_{B}
\subseteq \End(B)$ the subspace spanned by all linear maps of the form
$|x\rangle\langle x'| \colon x'' \mapsto x\mu_{B}(x'x'')$, where
$x,x'\in B$, and denote by $\hat{\mathcal{H}}=(\hat H,\hat\Delta_{\hat
  H})$ the dual of $\mathcal{H}=(H,\Delta_{H})$.
\begin{theorem} \label{theorem:coupled-dual}
Let $(S_{B} \colon
    B \rightleftarrows C \colon S_{C})$ be a coupled pair of algebras
    with an action of a regular multiplier Hopf algebra
    $\mathcal{H}=(H,\Delta_{H})$ and an invariant base weight
    $\mu=(\mu_{B},\mu_{C})$, and let $\phi_{H}$ be a left and
    $\psi_{H}$ a right integral on $\mathcal{H}$.    
    Then there exists an isomorphism
    \begin{align*}
      \hat A &\to K^{\mu}_{B} \otimes \hat H, \quad h S_{B}(x) \cdot
      \phi \cdot x' \mapsto |x\rangle\langle x'| \otimes h \cdot
      \phi_{H}.
    \end{align*}
    With respect to this isomorphism, the embeddings $\hat B,\hat C \to M(\hat
    A)$ are given by
    \begin{align}
      x'' ( |x\rangle\langle x'| \oo h\cdot \phi_{H}) &=
      |x''x\rangle\langle x'| \oo h\cdot \phi_{H}, \\
      S_{C}^{-1}(x'')(|x\rangle\langle x'| \oo h\cdot \phi_{H}) &=
      \sum |x \sigma^{\mu}_{B}(x'' \actb h_{(1)})\rangle\langle x'|
      \oo h_{(2)}\cdot \phi_{H} \label{eq:coupled-c}
  \end{align}
  for all $x,x',x''\in B$, $h \in H$. The comultiplications, counit,
  antipode and
  integrals on $(\hat{\mathcal{A}},\hat\mu,\hat\phi,\hat\psi)$ are
  given by
  \begin{gather*}
    \hat\Delta_{\hat B}(a)(1_{M(\hat A)} \oo
    (|x''\rangle\langle x'''| \oo \omega')) = \sum (|x\rangle\langle x'|
    \oo \omega_{(1)}) S(x'') \oo (|1\rangle\langle x'''| \oo
    \omega_{(2)}\omega'), \\
    \begin{aligned}
      \heps(a) &= \mu_{B}(x) \mu_{B}(x')\hat\eps_{\hat H}(\omega), &
      \hat S(a) &= S^{-1}_{C}(x')( |1\rangle\langle 1| \oo \hat
      S_{\hat H}(\omega)) S^{-1}_{C}(x), 
 \\
      \hat \phi(a) &= \mu_{B}(xx') \hat\phi_{\hat H}(\omega), & \hat
      \psi(a) &= \mu_{B}((x \actb \delta_{H}^{-1})x') \hat\psi_{\hat
        H}(\omega)
    \end{aligned}
    \end{gather*} 
    for $a=|x\rangle\langle x'| \oo \omega \in K^{\mu}_{B} \oo \hat
    H$, where $\hat\eps_{\hat H}$, $\hat S_{\hat H}$, $\hat \phi_{\hat
      H}$ and $\hat \psi_{\hat H}$ denote the counit, antipode and the
    left and right integrals on $\hat{\mathcal{H}}$, respectively, and
    where the formula for the comultiplication and antipode should be interpreted using
    \eqref{eq:coupled-c}.
\end{theorem}
\begin{proof}
  Throughout the proof, $x,x',x''$, $y,y',y''$ and $h,h',h''$ denote
  arbitrary elements of $B,C$ and $H$, respectively.

  We first show that $\hat A$ is isomorphic to $K^{\mu}_{B} \oo \hat
  H$. By definition
  \begin{gather} \label{eq:coupled-rho}
    \begin{aligned}
      \rho(hS_{B}(x) \cdot \phi \cdot x')(x''h''y'') &= \sum (\id
      \otimes S_{B}^{-1}\circ \cphic)(h''_{(1)}y'' \otimes
      x'x''h''_{(2)}h S_{B}(x)) \\ &= \sum \mu_{B}(x'x'') x
      \phi_{H}(h''_{(2)}h)h''_{(1)} y''.
    \end{aligned}
  \end{gather}
  On the other hand, since $A \cong B\oo  H\oo C$ via
  $x''h''y'' \equiv x'' \oo h'' \oo y''$ as a vector space, we can  define
  an embedding $\pi \colon K^{\mu}_{B} \oo \hat H \to \End(A)$ via the
  formula
  \begin{align*}
    \pi(|x\rangle\langle x'| \oo h \cdot\phi_{H}) (x''h'' y'') :=
    \sum x
    \mu_{B}(x' x'') \phi_{H}(h''_{(2)}h)h''_{(1)}  y''.
  \end{align*}
  The composition $\pi^{-1} \circ \rho$ is the desired
  isomorphism $\hat A \to K_{B} \oo \hat H$. 

  With respect to this isomorphism,
  \begin{align*}
x''    ( |x\rangle\langle x'|  \oo  h\cdot \phi_{H})  &= hS_{B}(x)
    \cdot \phi \cdot x' S_{C}^{-1}(x'') \\ &=hS_{B}(x)S_{B}(x'') \cdot \phi \cdot x'x'' =
    |x''x\rangle\langle x'|  \oo    h\cdot \phi_{H}
  \end{align*}
  and
  \begin{align*}
    S_{C}^{-1}(x'')(|x\rangle\langle x'| \oo h\cdot \phi_{H}) &=
    S_{C}^{-1}(x'')hS_{B}(x) \cdot \phi \cdot x' \\ & = \sum
    h_{(2)}(S^{-1}_{H}(h_{(1)}) \actc
    S_{C}^{-1}(x''))S_{B}(x) \cdot \phi \cdot x' \\
    &= \sum h_{(2)} S_{B}(x\sigma^{\mu}_{B}(x'' \actb h_{(1)})) \cdot
    \phi \cdot x' \\ & =   \sum |x \sigma^{\mu}_{B}(x'' \actb
    h_{(1)})\rangle\langle x'| \oo h_{(2)}\cdot \phi_{H}.
  \end{align*}

  Let us  compute the dual counit $\hat\eps$ and the dual integrals
  $\hat\psi,\hat \phi$. First,
  \begin{align*}
         \heps(|x\rangle\langle x'| \otimes  h \cdot \phi) &=
         \heps(hS_{B}(x) \cdot \phi \cdot x') =
         \phi(x'hS_{B}(x))  =\mu_{B}(x')\phi_{H}(h)\mu_{B}(x).
  \end{align*}
  The definition of $\hat \psi$ and relation
  \eqref{eq:coupled-counit-functional} imply
  \begin{align*}
    \hat \psi(|x\rangle\langle x'| \otimes h\cdot \phi_{H}) &= \hat \psi(h
    S_{B}(x) \cdot \phi \cdot x') \\
    &=  \eps(hS_{B}(x)\sigma^{\phi}(x')) \\&=
    \mu_{C}(((S_{H}^{-1}(h_{(2)})h_{(1)}) \actc
    S_{B}(x))S_{B}(\sigma^{\phi}(x'))) =
    \mu_{B}(\sigma^{\phi}(x')x)\eps_{H}(h).
  \end{align*}
  We insert the formula for $\sigma^{\phi}(x')$, use
  \eqref{eq:coupled-invariance-adjoint}, and find
  \begin{align*}
    \hat \psi(|x\rangle\langle x'| \otimes h\cdot \phi_{H}) = 
    \mu_{B}((\sigma^{\mu}_{B}(x') \actb \delta_{H})x) =\mu_{B}((x \actb
    \delta_{H}^{-1})x') \hat\psi_{H}(h \cdot \phi_{H}).
  \end{align*}
  The definition of $\hat \phi$ and the relation $\psi=\delta\cdot
  \phi$ imply
  \begin{align*}
    \hat \phi(|x\rangle\langle x'| \oo \psi_{H} \cdot h) &= \hat
    \phi(|x\rangle\langle x'| \oo \delta\sigma^{\phi}(h) \cdot
    \phi_{H}) \\ &=
    \hat \phi(\delta\sigma^{\phi}(h) S_{B}(x) \cdot \phi \cdot x')  
    = \hat \phi(\psi \cdot hS_{C}^{-1}(x)x') 
    = \eps(hS_{C}^{-1}(x)x').
  \end{align*}
Using \eqref{eq:coupled-counit-functional} similarly as above, we find
  \begin{align*}
        \hat \phi(|x\rangle\langle x'| \oo \psi_{H} \cdot h) &=
        \mu_{B}(x' S_{B}^{-1}S_{C}^{-1}(x)) \eps_{H}(h) = \mu_{B}(xx')
        \hat\phi_{\hat H}(\psi_{H} \cdot h).
  \end{align*}

  The verification of the formula for the antipode is straightforward.

  Let us finally prove the formula for the left comultiplication $\hat \Delta_{\hat
    B}$.  Consider elements
  \begin{align*}
    \upsilon &= hS(x)\cdot \phi \cdot x' = |x\rangle\langle x'| \oo
    h\cdot \phi_{H}, & \upsilon' &= h'S(x'') \cdot \phi \cdot x''' =
    |x''\rangle\langle x'''| \oo h' \cdot \phi_{H}
  \end{align*}
  of $\hat A$. Write $\omega:= h\cdot \phi_{H}$ and $\omega':=h' \cdot
  \phi_{H}$,  choose $g_{i},g'_{i} \in H$ such that
  \begin{align*}
    \sum \omega_{(1)} \oo \omega_{(2)}\omega' &= \sum_{i} g_{i} \cdot
    \phi_{H} \oo g'_{i} \cdot \phi_{H}
  \end{align*}
and  write 
  \begin{align*}
    \upsilon''_{i} &:= (x''g_{i}S(x) \cdot \phi \cdot x') = (g_{i}S(x)
    \cdot \phi \cdot x')S(x'') = (|x\rangle\langle x'| \oo g_{i} \cdot
    \phi_{H})S(x''),  \\ \upsilon'''_{i} &:= g'_{i} \cdot \phi \cdot
    x''' = |1\rangle\langle x'''|\oo g'_{i} \cdot \phi_{H}.
  \end{align*}
We have to show that
\begin{align*}
  \hat \Delta_{\hat B}(\upsilon)(1 \oo \upsilon') &= \sum_{i} \upsilon''_{i} \oo \upsilon'''_{i},
\end{align*}
and do so by evaluating both sides on elements $a\oo b \in A\ooc A$
using dual versions of the embeddings in Lemma
\ref{lemma:dual-embeddings-mult}.  Equation
\eqref{eq:dual-galois-multipliers} implies
  \begin{align*}
    (\hat \Delta_{\hat B}(\upsilon)(1 \oo \upsilon'))(\dual{a} \oo
    \dual{b}) &= \upsilon(a\rho(\upsilon')(b)) =
    \phi(x'a\rho(\upsilon')(b)hS(x)).
  \end{align*}
  On the other side,
  \begin{align*}
    \sum_{i} (\upsilon''_{i} \ooc \upsilon'''_{i})(a \oo b) &=
    \sum_{i} \upsilon''_{i}(a \cupsilon'''_{i}(b)) = \sum_{i} \phi(x'a
    \cupsilon'''_{i}(b)x''g_{i}S(x)).
  \end{align*}
  Therefore, we only need to show that $\rho(\upsilon')(b)h = \sum_{i
  }\cupsilon'''_{i}(b)x''g_{i}$.  Take $b$ of the form $b=\tilde x
  \tilde y  \tilde h$. Using \eqref{eq:coupled-rho} and
  \eqref{eq:reverse-crossed}, we find that
  \begin{align} \label{eq:coupled-comult-1}
    \rho(\upsilon')(b)h &= \rho(h'S_{B}(x'') \cdot \phi \cdot
    x''')(\tilde x  \tilde y \tilde h)h =\sum \mu_{B}(x'''\tilde x) \phi_{H}(\tilde h_{(2)}h') x'' \tilde
    y \tilde h_{(1)} h.
  \end{align}
On the other side,
\begin{align} \label{eq:coupled-comult-2}
  \sum_{i} \cupsilon'''_{i}(b)x''g_{i} &= \sum_{i} \cphic(x'''\tilde x
  \tilde y \tilde h g'_{i}) x'' g_{i} = \sum_{i} \mu_{B}(x'''\tilde x)
  \phi_{H}(\tilde hg'_{i})\tilde y x'' g_{i}.
\end{align}
To see that \eqref{eq:coupled-comult-1} and
\eqref{eq:coupled-comult-2} coincide,
we have to show that
\begin{align*}
  \sum \phi_{H}(\tilde h_{(2)}h') \tilde h_{(1)}h = \sum_{i}
  \phi_{H}(\tilde hg'_{i}) g_{i}.
\end{align*}
We multiply on the left by some arbitrary $h'' \in H$, apply
$\phi_{H}$, and get
\begin{align*}
  \sum \phi_{H}(\tilde h_{(2)}h') \phi_{H}(h''\tilde h_{(1)}h) &=
\omega(h''\tilde h_{(1)})  \omega'(\tilde h_{(2)}) \\ &= \sum
\omega_{(1)}(h'') (\omega_{(2)}\omega')(\tilde h) = \sum
\phi_{H}(h''g_{i}) \phi_{H}(\tilde hg'_{i}).
\end{align*}
Since $\phi_{H}$ is faithful and $h'' \in H$ was arbitrary, the
desired equality follows.
\end{proof}

\appendix

\section{Factorisable functionals on bimodules}
\label{section:factorisation}

In this appendix, we put the definition of the space $\dmAm$
associated to a regular multiplier Hopf algebroid $\mathcal{A}$ with
base weight $\mu$, and the construction of the relative tensor
products for functionals in $\dmAm$ given at the beginning of
Subsection \ref{subsection:uniqueness} into a bicategorical perspective.

We denote by $\BimodKMS$ the bicategory whose
\begin{itemize}
\item $0$-cells are algebras equipped with faithful KMS-functionals,
\item $1$-cells are bimodules between bialgebras, with the usual tensor product
  as composition,
\item $2$-cells are morphisms of bimodules, with the usual composition and
  tensor product.
\end{itemize}

Then there exists a lax bifunctor $D\colon \BimodKMS \to \BimodKMS^{\co}$  as
follows. 
\begin{enumerate}
\item On $0$-cells, the functor $D$ is given by $(A,\mu) \mapsto (A^{\op},\mu^{\op})$.
\item Given $0$-cells $(A,\mu)$, $(B,\nu)$ and an $A$-$B$-bimodule $M$, consider
  the space
  \begin{align*}
    D(M) &:= \{ \omega \in \dual{M} \mid \exists {_{A}\omega} \in \dual{_{A}M},
    \omega_{B} \in \dual{M_{B}} : \mu \circ {_{A}\omega} = \omega = \nu \circ
    \omega_{B} \}.
  \end{align*}
  This space becomes a $A^{\op}$-$B^{\op}$-bimodule if we define
  \begin{align*}
a^{\op}     \omega  &:= \omega(\sigma_{\mu}^{-1}(a) -), &  \omega b^{\op}  &:=
    \omega(-\sigma_{\nu}(b)).
  \end{align*}
  Indeed, one easily verifiies that $a\omega, \omega b \in D(M)$ and that
  \begin{align*}
    {_{A}(a^{\op}\omega )}(m) &= {_{A}\omega}(m)a, & (a^{\op}\omega )_{B}(m) &=
    \omega_{B}(\sigma_{\mu}^{-1}(a)m),  \\
    {_{A}(\omega b^{\op} )}(m) &= {_{A}\omega}(m\sigma_{\nu}(b)), & (\omega b^{\op} )_{B}(m) &=
    b\omega_{B}(m)
  \end{align*}
  for all $\omega \in D(M)$, $a\in A$, $b\in B$.
\item Given $0$-cells $(A,\mu)$, $(B,\nu)$ and a morphism of $A$-$B$-bimodules $T
  \colon M \to N$, there exists a morphism of $A^{\op}$-$B^{\op}$-bimodules $D(T)
  \colon D(N) \to D(M)$, $\omega \mapsto \omega \circ T$.
\item Given a $0$-cell $(A,\mu)$, there exists a morphism of $A$-bimodules
  ${_{A}A_{A}} \to D(_{A^{\op}}A^{\op}_{A^{\op}})$ given by $a \mapsto \mu^{\op}(a^{\op}-)$.
\item Given $0$-cells $(A,\mu)$, $(B,\nu)$, $(C,\phi)$, an $A$-$B$-bimodule $M$ and a
  $B$-$C$-bimodule $N$, there exists a morphism of $A^{\op}$-$C^{\op}$-bimodules
  $D(M) \otimes D(N) \to D(M \otimes N)$ such that
  \begin{align*}
    D(\omega \otimes \upsilon) &= \omega \circ (\id \otimes
    {_{B}\upsilon}) = \nu \circ (\omega_{B} \otimes
    {_{B\upsilon}}) = \upsilon \circ (\omega_{B} \otimes
    \id)
  \end{align*}
  and
  \begin{align*} {_{A}(D(\omega \otimes \upsilon))} &=
    {_{A}\omega} \circ (\id \otimes {_{B}\upsilon}), &
    (D(\omega \otimes \upsilon))_{C} &= \upsilon_{C} \circ
    (\omega_{B} \otimes \id)
  \end{align*}
  for all $\omega \in D(M)$, $\upsilon \in D(N)$.
\end{enumerate}
Moreover, there exists a lax transformation $\sigma \colon \id_{\BimodKMS} \to D^{2}$
as follows.
\begin{enumerate}\setcounter{enumi}{5}
\item For each $0$-cell $(A,\mu)$,  we have $D^{2}(A,\mu) = (A,\mu)$ so that we can
  take  $\sigma_{(A,\mu)} = 1_{(A,\mu)} = {_{A}A_{A}}$.
\item Given $0$-cells $(A,\mu)$, $(B,\nu)$ and an $A$-$B$-bimodule $M$, there exists
  a morphism of $A$-bimodules $\sigma_{M}\colon M\to D^{2}(M)$ given by
  $(\sigma_{M}(m))(\omega):=\omega(m)$, and 
  \begin{align*}
    {_{A}\sigma_{M}(m)}(\omega) &= {_{A}\omega}(m), &
    \sigma_{M}(m)_{B}(\omega) &= \omega_{B}(m).
  \end{align*}
\end{enumerate}

\subsection*{Acknowledgements} The author would like to thank Alfons
Van Daele for inspiring and fruitful discussions.

    \bibliographystyle{abbrv}
\def\cprime{$'$}

% % \listoffixmes

% %  \item Assume that $\dtA$ separates the points of $A$. Then an element
% %    $\tphi \in \dtAt$ is a left integral if and only if
% %    $\tomega \ast \tphi = t(\tomega(1)) \cdot \tphi$ for all $\tomega
% %    \in \dtLzz$ because
% %    \begin{align*}
% %      (\tomega \ast \tphi)(a) &= (\tomega \circ (\id
% %      \otimes \tphi) \circ \Delta_{B})(a), \\
% %      \tphi(at(\tomega(1))) &=\tomega(1)\tphi(a) =
% %      \tomega(t((\tphi(a)))) \quad \text{for all } a \in A.
% %    \end{align*}
% %    Likewise, if $\dbA$ separates the points of $A$, then an element
% %    $\spsi \in \dbAb$ is a right integral if and only if $\spsi \ast
% %    \somega = s(\somega(1)) \cdot \spsi$ for all $\somega \in \dsLzz$.

\end{document}